\documentclass[11pt, oneside, article]{memoir}

\selectfont 
\settrims{0pt}{0pt} 
\settypeblocksize{*}{35.4pc}{*} 
\setlrmargins{*}{*}{1} 
\setulmarginsandblock{0.85in}{0.85in}{*} 
\setheadfoot{\onelineskip}{2.66\onelineskip} 
\setheaderspaces{*}{1.5\onelineskip}{*} 
\checkandfixthelayout[lines] 

\makepagestyle{dap}
\makeevenhead{dap}{\small\slshape\leftmark}{}{\small\slshape\rightmark}
\makeoddhead{dap}{\small\slshape\leftmark}{}{\small\slshape\rightmark}
\makeevenfoot{dap}{}{\thepage}{}
\makeoddfoot{dap}{}{\thepage}{}
\makeheadrule{dap}{\textwidth}{\normalrulethickness}
\makepsmarks{dap}{%
  \nouppercaseheads
  \createmark{chapter}{left}{shownumber}{}{.\enspace}%
  \createmark{section}{right}{shownumber}{}{.\enspace}%
}
\usepackage{amsthm}
\usepackage{mathtools}

\usepackage[inline]{enumitem}
\usepackage{calc}
\usepackage{ifthen}
\usepackage[utf8]{inputenc} 
\usepackage{xcolor}

\usepackage[backend=biber, backref=true, maxbibnames = 10, style = alphabetic]{biblatex}
\usepackage[bookmarks=true, colorlinks=true, linkcolor=blue!50!black,
citecolor=orange!50!black, urlcolor=orange!50!black, pdfencoding=unicode]{hyperref}
\usepackage{xurl} 
\usepackage[capitalize]{cleveref}

\usepackage{tikz}
\usepackage{pdflscape}
\usepackage{afterpage}

\usepackage{amssymb}
\usepackage{newpxtext}
\usepackage[varg,bigdelims]{newpxmath}
\usepackage{mathrsfs}
\usepackage{dutchcal}
\usepackage[scaled=0.95]{helvet}
\DeclareMathAlphabet{\mathsf}{T1}{phv}{m}{n}
\usepackage{fontawesome}
\usepackage{ebproof}
\usepackage{stmaryrd}

  \crefformat{enumi}{\card#2#1#3}
  \crefalias{chapter}{section}

  \hypersetup{final}

  \setlist{nosep}
  \setlistdepth{6}

  \usetikzlibrary{
  	cd,
  	math,
  	decorations.markings,
		decorations.pathreplacing,
		decorations.pathmorphing,
  	positioning,
  	arrows.meta,
  	shapes,
		shadows,
		shadings,
  	calc,
  	fit,
  	quotes,
  	intersections,
  	circuits.logic.US,
  	backgrounds,
  }

  \tikzset{
biml/.tip={Glyph[glyph math command=triangleleft, glyph length=.95ex]},
bimr/.tip={Glyph[glyph math command=triangleright, glyph length=.95ex]},
}

\tikzset{
	tick/.style={postaction={
  	decorate,
    decoration={markings, mark=at position 0.5 with
    	{\draw[-] (0,.4ex) -- (0,-.4ex);}}}
  }
}
\tikzset{
	slash/.style={postaction={
  	decorate,
    decoration={markings, mark=at position 0.5 with
    	{\draw[-] (.3ex,.3ex) -- (-.3ex,-.3ex);}}}
  }
}
\tikzset{
	dot/.style={circle, draw, fill=black, inner sep=0, minimum size=3pt}
}

\tikzset{
	oriented WD/.style={
		every to/.style={out=0,in=180,draw},
    label/.style={
    	font=\everymath\expandafter{\the\everymath\scriptstyle},
      inner sep=0pt,
      node distance=2pt and -2pt},
    semithick,
    node distance=1 and 1,
    decoration={markings, mark=at position \stringdecpos with \stringdec},
    ar/.style={postaction={decorate}},
    execute at begin picture={\tikzset{
    	x=\bbx, y=\bby,
      every fit/.style={inner xsep=\bbx, inner ysep=\bby}}}
    },
    string decoration/.store in=\stringdec,
    string decoration={\arrow{stealth}},
    string decoration pos/.store in=\stringdecpos,
    string decoration pos=.5,
    bbx/.store in=\bbx,
    bbx = 1.5cm,
    bby/.store in=\bby,
    bby = 1.5ex,
    bb port sep/.store in=\bbportsep,
    bb port sep=1.5,
    bb port length/.store in=\bbportlen,
    bb port length=4pt,
    bb penetrate/.store in=\bbpenetrate,
    bb penetrate=0,
    bb min width/.store in=\bbminwidth,
    bb min width=1cm,
    bb rounded corners/.store in=\bbcorners,
    bb rounded corners=2pt,
    bb small/.style={
    	bb port sep=1, bb port length=2.5pt, bbx=.4cm, bb min width=.4cm, bby=.7ex},
    bb/.code 2 args={
    	\pgfmathsetlengthmacro{\bbheight}{\bbportsep * (max(#1,#2)+1) * \bby}
      \pgfkeysalso{draw,minimum height=\bbheight,minimum
       width=\bbminwidth,outer sep=0pt,
         rounded corners=\bbcorners,thick,
         prefix after command={\pgfextra{\let\fixname\tikzlastnode}},
         append after command={\pgfextra{\draw
            \ifnum #1=0{} \else foreach \i in {1,...,#1} {
            	($(\fixname.north west)!{\i/(#1+1)}!(\fixname.south west)$) +(-\bbportlen,0) coordinate (\fixname_in\i) -- +(\bbpenetrate,0) coordinate (\fixname_in\i')}\fi
            \ifnum #2=0{} \else foreach \i in {1,...,#2} {
            	($(\fixname.north east)!{\i/(#2+1)}!(\fixname.south east)$) +(-\bbpenetrate,0) coordinate (\fixname_out\i') -- +(\bbportlen,0) coordinate (\fixname_out\i)}\fi;
           }}}
		},
		bb name/.style={
     	append after command={
				\pgfextra{\node[anchor=north] at (\fixname.north) {#1};}
			}
		}
}

\newcommand{\boxpic}[2]{%
  \begin{tikzpicture}[oriented WD, bb small, baseline=(M.south)]
    \node[bb={1}{1}] (M) {};
    \node[left=-1pt] at (M_in1) {\scriptsize #1};
    \node[right=-1pt] at (M_out1) {\scriptsize #2};
  \end{tikzpicture}%
}

\newcommand{\boxpicout}[1]{%
  \begin{tikzpicture}[oriented WD, bb small, baseline=(M.south)]
    \node[bb={0}{1}] (M) {};
    \node[right=-1pt] at (M_out1) {\scriptsize #1};
  \end{tikzpicture}%
}

\newcommand{\tankpic}{%
  \begin{tikzpicture}[oriented WD, bb small, font=\tiny, baseline=(current bounding box.center)]
    \node[bb port sep=4.5mm, bb={0}{0}, align=center] (pipe) {$\fun{pipe}$\\$\kappa,z$};
    \node[bb={0}{1}, align=center, above left=.7 and 4.5 of pipe.180] (t1) {$\fun{tank}_1$\\$v_1$};
    \node[bb={0}{1}, align=center, below left=.7 and 4.5 of pipe.180] (t2) {$\fun{tank}_2$\\$v_2$};
    \node[bb={0}{0}, fit=(t1) (t2) (pipe)] (outer) {};
    \draw[ar] (t1_out1) -- node[above right=-2pt and 0, pos=0]{$e_1$} node[below left=-2pt and 0, pos=1]{$I(e_1,e_2)$} (t1_out1-|pipe.west);
    \draw[ar] (t2_out1) -- node[above right=-2pt and 0, pos=0]{$e_2$} node[below left=-2pt and 0, pos=1]{$-I(e_1,e_2)$} (t2_out1-|pipe.west);
  \end{tikzpicture}%
}

\usepackage{aliascnt}

\theoremstyle{definition}
\newtheorem{definitionx}{Definition}[section]
\crefname{definitionx}{Definition}{Definitions}

\newaliascnt{examplex}{definitionx}
\newtheorem{examplex}[examplex]{Example}
\aliascntresetthe{examplex}
\crefname{examplex}{Example}{Examples}

\newaliascnt{warning}{definitionx}
\newtheorem{warning}[warning]{Warning}
\aliascntresetthe{warning}
\crefname{warning}{Warning}{Warnings}

\newaliascnt{remarkx}{definitionx}
\newtheorem{remarkx}[remarkx]{Remark}
\aliascntresetthe{remarkx}
\crefname{remarkx}{Remark}{Remarks}

\newaliascnt{notation}{definitionx}
\newtheorem{notation}[notation]{Notation}
\aliascntresetthe{notation}
\crefname{notation}{Notation}{Notations}

\newaliascnt{conventionx}{definitionx}

\aliascntresetthe{conventionx}
\crefname{conventionx}{Convention}{Conventions}

\theoremstyle{plain}

\newaliascnt{theorem}{definitionx}
\newtheorem{theorem}[theorem]{Theorem}
\aliascntresetthe{theorem}
\crefname{theorem}{Theorem}{Theorems}

\newaliascnt{proposition}{definitionx}
\newtheorem{proposition}[proposition]{Proposition}
\aliascntresetthe{proposition}
\crefname{proposition}{Proposition}{Propositions}

\newaliascnt{corollary}{definitionx}
\newtheorem{corollary}[corollary]{Corollary}
\aliascntresetthe{corollary}
\crefname{corollary}{Corollary}{Corollaries}

\newaliascnt{lemma}{definitionx}
\newtheorem{lemma}[lemma]{Lemma}
\aliascntresetthe{lemma}
\crefname{lemma}{Lemma}{Lemmas}

\newaliascnt{fact}{definitionx}

\aliascntresetthe{fact}
\crefname{fact}{Fact}{Facts}

\newaliascnt{conjecture}{definitionx}

\aliascntresetthe{conjecture}
\crefname{conjecture}{Conjecture}{Conjectures}

\newtheorem*{theorem*}{Theorem}
\newtheorem*{proposition*}{Proposition}
\newtheorem*{corollary*}{Corollary}
\newtheorem*{lemma*}{Lemma}
\newtheorem*{warning*}{Warning}

\newenvironment{example}
  {\pushQED{\qed}\examplex}
  {\popQED\endexamplex}

 \newenvironment{remark}
  {\pushQED{\qed}\remarkx}
  {\popQED\endremarkx}

  \newenvironment{definition}
  {\pushQED{\qed}\definitionx}
  {\popQED\enddefinitionx}

\DeclareSymbolFont{stmry}{U}{stmry}{m}{n}
\DeclareMathSymbol\fatsemi\mathop{stmry}{"23}

\DeclareFontFamily{U}{mathx}{\hyphenchar\font45}
\DeclareFontShape{U}{mathx}{m}{n}{
      <5> <6> <7> <8> <9> <10>
      <10.95> <12> <14.4> <17.28> <20.74> <24.88>
      mathx10
      }{}
\DeclareSymbolFont{mathx}{U}{mathx}{m}{n}
\DeclareFontSubstitution{U}{mathx}{m}{n}
\DeclareMathAccent{\widecheck}{0}{mathx}{"71}

\renewcommand{\ss}{\subseteq}

\DeclarePairedDelimiter{\pairing}{\langle}{\rangle}

\DeclarePairedDelimiter{\ihom}{[}{]}
\DeclarePairedDelimiter{\absval}{\lvert}{\rvert}

\DeclareMathOperator{\ob}{Ob}

\newcommand{\cat}[1]{\mathcal{#1}}
\newcommand{\Cat}[1]{\mathbf{#1}}
\newcommand{\fun}[1]{\mathtt{#1}}
\newcommand{\Fun}[1]{\text{\normalfont\textsf{#1}}}

\newcommand{\notsafe}[1]{{#1}}
\makeatletter
\NewDocumentCommand{\trackTermSymbol}{m m}{%
  \expandafter\newcommand\csname termraw@#1\endcsname{\notsafe{#2}}%
}
\NewDocumentCommand{\termraw}{m}{%
  \@ifundefined{termraw@#1}%
    {\PackageError{termraw}{Term '#1' not registered}{Use \string\trackTermSymbol{#1}{...} first.}}%
    {\@nameuse{termraw@#1}}%
}
\DeclareRobustCommand{\termref}[2]{\notsafe{\hypersetup{linkcolor=black}\hyperlink{term.#1}{#2}}}
\NewDocumentCommand{\trackTerm}{m m m}{%
  \trackTermSymbol{#2}{#3}%
  \newcommand{#1}{\termref{#2}{\termraw{#2}}}%
}
\NewDocumentCommand{\defineTerm}{m}{\notsafe{\hypertarget{term.#1}{\termraw{#1}}}}
\makeatother

\newcommand{\id}{\mathrm{id}}
\newcommand{\then}{\mathbin{\fatsemi}}

\newcommand{\too}{\longrightarrow}
\newcommand{\To}[2][]{\xrightarrow[#1]{#2}}
\renewcommand{\Mapsto}[1]{\xmapsto{#1}}
\newcommand{\from}{\leftarrow}
\newcommand{\From}[1]{\xleftarrow{#1}}
\newcommand{\inj}{\hookrightarrow}
\newcommand{\surj}{\twoheadrightarrow}
\newcommand{\tto}{\rightrightarrows}

\newcommand{\card}{\,^{\#}}
\makeatletter
\newcommand{\leftvec}[1]{\mathpalette\leftvec@{#1}}
\newcommand{\leftvec@}[2]{\reflectbox{$\m@th#1\vec{\reflectbox{$\m@th#1#2$}}$}}
\makeatother
\newcommand{\bk}[2]{#1[#2]}

\newcommand{\inv}{^{-1}}
\newcommand{\op}{^\tn{op}}

\newcommand{\tn}[1]{\textnormal{#1}}

\newcommand{\nn}{\mathbb{N}}
\newcommand{\rr}{\mathbb{R}}

\trackTerm{\finset}{finset}{\Cat{Fin}}
\trackTerm{\smset}{smset}{\Cat{Set}}
\trackTerm{\smsetiso}{smsetiso}{\Cat{Set}_\cong}
\trackTerm{\smcat}{smcat}{\Cat{Cat}}
\trackTerm{\vect}{vect}{\Cat{Vect}}
\trackTerm{\vectiso}{vectiso}{\Cat{Vect}_\cong}

\newcommand{\List}{\Fun{List}}

\trackTerm{\yon}{yon}{\mathcal{y}}
\trackTerm{\poly}{poly}{\Cat{Poly}}

\newcommand{\biglens}[2]{
     \begin{bmatrix}{\vphantom{f_f^f}#2} \\ {\vphantom{f_f^f}#1} \end{bmatrix}
}
\newcommand{\littlelens}[2]{
     \begin{bsmallmatrix}{\vphantom{f}#2} \\ {\vphantom{f}#1} \end{bsmallmatrix}
}
\newcommand{\lens}[2]{
  \relax\if@display
     \biglens{#1}{#2}
  \else
     \littlelens{#1}{#2}
  \fi
}

\newcommand{\qqand}{\qquad\text{and}\qquad}
\newcommand{\qqwhere}{,\quad\textit{where}\qquad}
\newcommand{\qqthus}{\qquad\text{thus}\qquad}

\newcommand{\coalg}{\tn{-}\Cat{Coalg}}
\newcommand{\alg}{\tn{-}\Cat{Alg}}
\newcommand{\Coalg}{\Fun{Coalg}}
\newcommand{\enrcat}[1]{#1\tn{-}\smcat}

\trackTerm{\pc}{pc}{\mathbb{P}\Cat{C}}
\trackTerm{\pciso}{pciso}{\mathbb{P}\Cat{C}_{\cong}}
\trackTerm{\arr}{rwd}{\mathbb{A}\Cat{rr}}

\trackTerm{\mfd}{mfd}{\Cat{Mfd}}
\trackTerm{\etaLR}{etaLR}{\eta_{\mathrm{LR}}}
\trackTerm{\Store}{Store}{\Fun{Store}}
\newcommand{\inc}{\Fun{inc}}
\trackTermSymbol{cot}{\Fun{cot}}
\RenewDocumentCommand{\cot}{}{\termref{cot}{\termraw{cot}}}
\NewDocumentCommand{\cotof}{m}{\termref{cot}{\termraw{cot}}(#1)}
\trackTermSymbol{chiQ}{\chi}
\newcommand{\chiQ}[1]{\termref{chiQ}{\termraw{chiQ}}_{#1}}
\trackTermSymbol{nu}{\nu}
\newcommand{\nuro}{\termref{nu}{\termraw{nu}}}
\trackTermSymbol{Psisem}{\Psi}
\newcommand{\Psisem}[1]{\termref{Psisem}{\termraw{Psisem}}_{#1}}
\newcommand{\ev}{\fun{ev}}
\newcommand{\coev}{\fun{coev}}
\newcommand{\src}{\mathrm{src}}
\newcommand{\tgt}{\mathrm{tgt}}
\newcommand{\Part}{\bullet}
\newcommand{\RealPart}{\circ}
\newcommand{\Bond}{\mathord{\sim}}
\newcommand{\boxob}{\mathrm{box}}
\trackTerm{\bang}{bang}{\mathord{!}}

\providecommand{\Sm}{\mathrm{Sm}}
\trackTerm{\sarr}{srwd}{\mathbb{A}\Cat{rr}_{\Sm}}
\trackTerm{\Para}{Para}{\mathbb{P}\Cat{ara}}
\newcommand{\para}[3][]{\Para_{#2}^{#1}(#3)}
\trackTerm{\rmfdc}{rmfdc}{\Cat{RMfdC}}
\trackTerm{\rvect}{rvect}{\Cat{RVect}}
\trackTerm{\SLens}{SLens}{\Cat{SLens}}
\newcommand{\rv}[1]{\mathbf{#1}}
\trackTermSymbol{Lens}{\Cat{Lens}}
\newcommand{\Lens}[1]{\termref{Lens}{\termraw{Lens}}_{#1}}
\newcommand{\lmfd}{\Lens{\mfd}}
\trackTermSymbol{LensFun}{\Cat{Lens}}
\newcommand{\LensFun}[1]{\termref{LensFun}{\termraw{LensFun}}^{#1}}
\trackTermSymbol{LensMor}{\Cat{Lens}}
\newcommand{\LensMor}[2]{\termref{LensMor}{\termraw{LensMor}}_{#1}^{#2}}
\trackTermSymbol{Lcokl}{\Cat{Lens}_{\cat C}^{\Fun T}}
\newcommand{\Lcokl}[2]{\termref{Lcokl}{\Cat{Lens}_{#1}^{#2}}}
\newcommand{\plmfd}{\Lcokl{\mfd}{\rr}}
\newcommand{\ltens}{\boxtimes}
\newcommand{\bigltens}{\mathop{\vcenter{\hbox{\begin{tikzpicture}[line width=0.6pt,line join=miter]\def\s{5.5pt}\draw(-\s,-\s)rectangle(\s,\s);\draw(-\s,-\s)--(\s,\s);\draw(-\s,\s)--(\s,-\s);\end{tikzpicture}}}}\displaylimits}
\newcommand{\ltpow}[1]{^{\ltens #1}}
\newcommand{\inpt}[1]{#1^{-}}
\newcommand{\outp}[1]{#1^{+}}
\newcommand{\lensob}[1]{\binom{\inpt{#1}}{\outp{#1}}}
\newcommand{\lensobs}[2]{\binom{\inpt{#1}\otimes\inpt{#2}}{\outp{#1}\otimes\outp{#2}}}
\newcommand{\lensunit}[1][1]{\binom{#1}{#1}}
\trackTerm{\potd}{potd}{\mathrm{d}}
\trackTerm{\potalg}{potalg}{\mathbf z}
\newcommand{\md}{\mathsf{Md}}
\newcommand{\conf}{\mathrm{conf}}
\newcommand{\phase}{\mathrm{phase}}
\newcommand{\lrn}{\mathrm{lrn}}
\newcommand{\euc}{\mathrm{Euc}}
\trackTerm{\rand}{rand}{\fun{rand}}

\newcommand{\Phip}[1]{\Phi'_{#1}}
\trackTermSymbol{Phiphase}{\Phi_{\phase}}
\newcommand{\Phiphase}{\termref{Phiphase}{\termraw{Phiphase}}}
\trackTermSymbol{Phiconf}{\Phi_{\conf}}
\newcommand{\Phiconf}{\termref{Phiconf}{\termraw{Phiconf}}}
\trackTermSymbol{sharpEuc}{\sharp_{\euc}}
\newcommand{\sharpEuc}{\termref{sharpEuc}{\termraw{sharpEuc}}}
\trackTermSymbol{dyn}{\mathrm{dyn}}
\newcommand{\dyn}{\termref{dyn}{\termraw{dyn}}}
\trackTermSymbol{dynphase}{\mathrm{dyn}_{\mathrm{phase}}}
\newcommand{\dynphase}{\termref{dynphase}{\termraw{dynphase}}}
\trackTerm{\intg}{intg}{\mathfrak{i}}
\trackTerm{\interp}{interp}{\mathfrak{p}}
\trackTerm{\interpsm}{interpsm}{\mathrm{sm}}
\trackTermSymbol{sharpR}{\sharp}
\newcommand{\sharpR}{\termref{sharp_map}{\termraw{sharpR}}}
\trackTermSymbol{sharp_symp}{\sharp}
\trackTermSymbol{sharpS}{\sharp}
\newcommand{\sharpS}{\termref{sharp_symp}{\termraw{sharpS}}}

\font\MnSyCfont=MnSymbolC10
\font\MnSyCfonts=MnSymbolC7
\font\MnSyCfontss=MnSymbolC5
\newcommand{\smallsquare}{\mathchoice
  {\mbox{\MnSyCfont\symbol{"68}}}%
  {\mbox{\MnSyCfont\symbol{"68}}}%
  {\mbox{\MnSyCfonts\symbol{"68}}}%
  {\mbox{\MnSyCfontss\symbol{"68}}}}
\newcommand{\blank}{\notsafe{\mathord{\mkern1mu\smallsquare\mkern2mu}}}

\newcommand{\thanksAFOSR}[1]{This material is based upon work supported by the Air Force Office of Scientific Research under award number #1}

\setsecnumdepth{subsection}
\settocdepth{section}
\begin{document}

\title{Compositional Dynamics in Learning and Mechanics}

\author{David I. Spivak}

\date{\vspace{-.2in}}

\maketitle

\begin{abstract}
We give a single compositional setting in which gradient-based learning and Hamiltonian-style mechanics appear as functorial semantics. The syntax is an operad $\sarr$ whose objects are input-output interfaces (pairs of manifolds) and whose morphisms are \emph{smooth adaptive arrangements}, which consist of a responsive parameter space, a lens given by smooth output and input maps, and a real-valued potential.

The main technical result of the paper is what we call \emph{lens internalization}, a lax symmetric monoidal functor $\Lens{\cat{C}}\to\cat{C}$ associated to any symmetric monoidal closed category $\cat{C}$. Using it, we provide two functors $\Phiphase,\Phiconf\colon\sarr\to\pc$ into the 2-category of polynomial coalgebras---input-output discrete dynamical systems---which we take as the semantics category. $\Phiphase$ stores both position and momentum, whereas $\Phiconf$ stores only position. 

When applied to a parameterized function, $\Phiconf$ recovers the gradient descent training algorithm, with backpropagation as the lens' backward pass. When applied to harmonic particles wired together---in series, or according to any finite directed graph---one diagram yields two different regimes, both of which are governed by the graph Laplacian: $\Phiphase$ gives the discrete wave equation, which is second-order, and $\Phiconf$ gives the discrete heat equation, which is first-order. They are two semantics of one adaptive arrangement, e.g.\ with the same potential in each case. And because $\sarr$ is an operad, such diagrams nest---larger systems wired from smaller ones---and each semantics assembles a system's dynamics functorially from its parts. These dynamics are moreover executable: a parameterized neural network and a graph of particles both compile, by the same construction, to explicit state machines one can run.
\end{abstract}


\tableofcontents*

\chapter{Introduction}\label{ch.intro}

There is a commonality in Hamiltonian dynamics and gradient descent, namely a map $T^*M\to TM$ from a cotangent bundle to a tangent bundle, transforming covectors into vectors that change the state. In each case, the covector it transforms is the differential $dU$ of a potential $U$---the potential energy in Hamiltonian dynamics or the loss in gradient descent---so it is the potential that drives the state. We refer to vector spaces with this sort of covector-to-vector structure as \emph{responsive}. In Hamiltonian dynamics, the response structure is the \emph{symplectic sharp} map on phase space~\cite[\S\S 18.1--18.2]{cannasdasilva2008lectures}; in gradient descent it is the sharp map induced by a Riemannian metric, scaled by a learning rate on parameter space. In both cases, the dynamics is compositional: local subsystems---each with their own dynamics---exchange information through which their dynamics assemble into an integrated system.

The goal of this paper is to build a compositional syntax for such smooth responsive, adaptive systems---one in which the couplings between subsystems can themselves adapt---and a functorial semantics that turns this syntax into explicit discrete update rules.

\section{Motivation for smooth adaptive arrangements}\label{sec.intro_smooth_rewiring}

Operads are a natural categorical setting for establishing compositionality results~\cite{Spivak:2013b,Yau:2015a}. An operad $\cat{W}$ establishes a composition \emph{syntax}---what sorts of diagrams can be composed---and a functor $\Phi\colon\cat{W}\to\smset$ establishes the corresponding semantics. For example consider the diagram:
\begin{equation}\label{eqn.varphi}
\varphi=\begin{tikzpicture}[oriented WD, bb port length=4pt, bb port sep=1.5, bb min width=.35cm, bbx=.4cm, bby=.8ex, baseline=(box1), particle/.style={path picture={\fill (path picture bounding box.center) circle (2.4pt);}}]
  \def\N{6}
  \node[bb={1}{1}, particle] (box1) {};
  \foreach \i [evaluate=\i as \prev using int(\i-1)] in {2,...,\N} {
    \node[bb={1}{1}, particle, right=.7 of box\prev] (box\i) {};
    \draw (box\prev_out1) -- (box\i_in1);
  }
  \node[bb={0}{0}, fit={(box1) (box\N)}] (Y) {};
  \node[coordinate] at (Y.west|-box1_in1) (Y_in1) {};
  \node[coordinate] at (Y.east|-box\N_out1) (Y_out1) {};
  \draw (Y_in1) -- (box1_in1);
  \draw (box\N_out1) -- (Y_out1);
\end{tikzpicture}
\end{equation}
For interfaces $b_1,\ldots,b_K,c$, a morphism $\varphi\colon b_1,\ldots,b_K\to c$ in $\cat{W}$ represents an arrangement of $K$ internal boxes, with interfaces $b_1,\ldots,b_K$, inside an external box with interface $c$. The picture above has $K=6$: it wires up six boxes inside a box, all of which have the same type:
\boxpic{}{}. The dots are \emph{particles}: $0$-ary morphisms, boxes with no further boxes inside. As syntax, $\varphi$ records the architecture of the chain---the types of each wire and how boxes are connected---but it does not say what the boxes do or what signals pass along each wire; that is the job of the semantics $\Phi$. It is also static: $\varphi$ doesn't change no matter how the particles move.

In the examples below we need a richer composition syntax: one where the effective coupling is itself part of the data and may adapt to the signals flowing between boxes. The way it adapts is governed by a dynamical system whose behavior depends on continuous parameters, on the values exposed by the internal boxes, and on values received by the external box. For example, consider replacing $\varphi$ by an interaction pattern representing an adaptive spring material for which each connecting edge has a variable spring constant that responds to a chosen potential $U$ and the edge lengths.

We call these richer composition patterns \emph{smooth adaptive arrangements}: each carries an internal state---the parameter---that the dynamics evolve over time, which is what \emph{adaptive} names, together with a real-valued \emph{potential} $U$---the potential energy in mechanics, the loss in learning---which drives that evolution and is the datum at the center of this paper. Formally, they will be the morphisms in the operad $\sarr$ of smooth adaptive arrangements constructed in \cref{sec.potentialized_lenses}. Instead of giving this syntax static semantics in $\smset$, we give it dynamical semantics in $\pc$,%
\footnote{
In previous work \cite{spivak2021learners,shapiro2022dynamic}, the author has referred to $\pc$ as $\mathbb{O}\Cat{rg}$.
}
whose morphisms are polynomial coalgebras, i.e.\ input-output dynamical systems. Thus a $0$-ary morphism can represent a component, such as a single dynamical particle, and a $K$-ary morphism can represent an adaptive interaction among $K$ components.

Having semantics land in $\pc$ amounts to using polynomial functors as input-output interfaces. Polynomial functors are built from one variable $\yon$ by coproducts and products, so a polynomial
\[
p=\sum_{i:I}\yon^{p[i]}
\]
is just a family of \emph{directions} $p[i]$ indexed by a family of \emph{positions} $i:I$. We regard positions as the possible outputs and the directions as the inputs accepted at that output. A $p$-coalgebra $S\to p(S)$ is then a deterministic dynamical system with state: given a state $s$ it specifies an output position and, for every acceptable input direction, specifies a next state. 


Most relevantly for smooth adaptive arrangements, $\poly$ also carries the Dirichlet product $\otimes$, which combines interfaces in parallel, and this monoidal structure is closed. A map of polynomials $p\to p'$ can be understood as a static coupling pattern: it sends each output position of the internal interface $p$ to an output position of the external interface $p'$ and each compatible input direction of $p'$ back to an input direction of $p$. Thus the $K$-many internal interfaces $p_1,\ldots,p_K$ of a wiring diagram combine as $p_1\otimes\cdots\otimes p_K$, and a static wiring into an external interface $p'$ is a map $\varphi\colon p_1\otimes\cdots\otimes p_K\to p'$ as in \eqref{eqn.varphi}.

Because $\poly$ is monoidal closed, such maps are equivalently the positions of the internal hom
\[
w\coloneqq\ihom{p_1\otimes\cdots\otimes p_K,p'},
\]
so a fixed $\varphi$ is precisely the same data as a one-state coalgebra on the polynomial $w$. To make the wiring dynamic, we replace that one-state coalgebra by an arbitrary coalgebra $f\colon S\to w(S)$: its state $s:S$ outputs the current wiring pattern $\varphi_s\colon p_1\otimes\cdots\otimes p_K\to p'$, and each acceptable input direction determines the next state $s'$. These internal-hom coalgebra categories are the morphisms in $\pc$. We will explain all this in detail in \cref{sec.poly_dynamic_bg}.

\section{From smooth data to dynamics}\label{sec.intro_smooth_to_dynamics}

The bridge from manifolds to polynomial functors is the cotangent functor $\cot\colon\mfd\to\poly$. It sends a manifold $M$ to the polynomial
\[
\cotof{M}\coloneqq\sum_{m\in M}\yon^{T_m^*M},
\]
whose positions are points $m\in M$ and whose directions at $m$ are the covectors $T_m^*M$. A smooth map $f\colon M\to N$ acts forward on positions by $f$ and backward on directions by cotangent pullback $(T_mf)^\top\colon T^*_{f(m)}N\to T^*_mM$. Crucially, $\cot$ is strong symmetric monoidal: $\cotof{M\times N}\cong\cotof{M}\otimes\cotof{N}$. Thus Dirichlet products of polynomial interfaces correspond to ordinary products of manifolds.

We are now ready to describe the source syntax: ordinary smooth data that will be sent through $\cot$ and then converted into the dynamics of a polynomial coalgebra. A neural network, for example, begins as a smooth map
\[f\colon Q\times M\to N\]
i.e.\ a choice of parameter, input, and output manifolds, together with a model map; typically, $M=\rr^m$ and $N=\rr^n$ are Euclidean. Categorically, $f$ is a parameterized map $M\to N$ in the sense of the Para construction~\cite[Def.~2]{capucci2022cybernetics} (\cref{sec.para_general}), with parameter $Q$. Gradient descent updates this parameter via a map $T^*Q\to TQ$ that converts a cotangent (the gradient) into a tangent step on $Q$; such a map is exactly a response structure $\sharpR_Q$ (\cref{sec.rvect}), so $Q$ is implicitly responsive in the gradient-descent setting.

Aside from making the response structure explicit, smooth adaptive arrangements---the morphisms in $\sarr$---generalize the above familiar situation in three ways. First, an ordinary manifold $M$, i.e.\ in the role of $\rr^m$ and $\rr^n$ above---a pure output, \boxpicout{$M$}---becomes only one component of an interface $\binom{\inpt M}{\outp M}$, namely the output side $\outp M$; that is, we add a new input side $\inpt M$, which is an ordinary smooth port for information supplied by the surrounding system, so that $M$ becomes a box \boxpic{$\inpt M$}{$\outp M$}.
Second, an ordinary map $f\colon Q\times M\to N$ becomes only one component of a pair of maps
\begin{equation}\label{eqn.first_lens}
\outp f\colon Q\times\outp M\to\outp N
\qqand
\inpt f\colon Q\times\outp M\times\inpt N\to\inpt M.
\end{equation}
That is, we add a new input map $\inpt f\colon Q\times\outp M\times\inpt N\to\inpt M$, which routes the new external inputs back to the new input sides of the internal interfaces.
And third, we add a potential
\begin{equation}\label{eqn.potential}
U\colon Q\times\outp{M}\times\inpt{N}\to\rr,
\end{equation}
whose differential contributes to the dynamics.
Seen from the box's perspective, \eqref{eqn.potential} provides some intuition regarding the distinction between inputs and outputs of a box. As a slogan,
\begin{center}
\emph{	
Your inputs are what your interior can have preferences over, and\\
your outputs are what your exterior can have preferences over.
}
\end{center}
Indeed, the potential lives interior to the outer box $N$ and exterior to the inner box $M$,
so it reads $Q$, $\outp M$, and $\inpt N$, and can read neither $\inpt M$ nor $\outp N$.

All the data from \eqref{eqn.first_lens} and \eqref{eqn.potential} constitutes a morphism $f=(\rv Q,\outp f,\inpt f,U)$ and is depicted in \eqref{eqn.arrangement}:
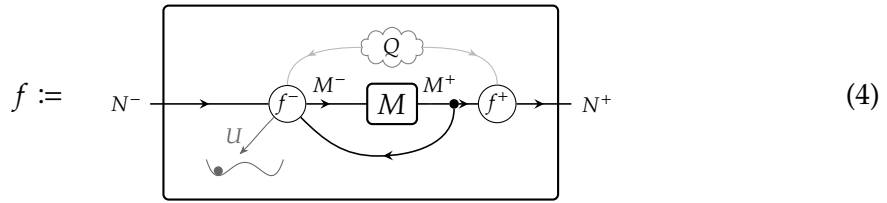
\begin{equation}\label{eqn.arrangement}
f\coloneqq\quad
\begin{tikzpicture}[oriented WD, bb port length=5pt, bb port sep=1, bb min width=.5cm, bbx=.7cm, bby=.7ex,
    ar/.style={postaction={decorate, decoration={markings, reset marks,
      mark=at position \stringdecpos with \stringdec}}},
    ann ar/.style={postaction={decorate, decoration={markings, reset marks,
      mark=at position .6 with {\arrow{Stealth[length=3.5pt]}}}}}, baseline=(M)]
  \node[bb={1}{1}] (M) {$M$};
  \node[circle, draw, thin, inner sep=.5pt, left=.9 of M_in1] (fin) {\scriptsize$\inpt f$};
  \node[circle, draw, thin, inner sep=.5pt, right=.9 of M_out1] (fout) {\scriptsize$\outp f$};
  \node[dot] at ($(M_out1)!.5!(fout.west)$) (split) {};
  \node[cloud, draw=gray, thin, cloud puffs=10, aspect=2.2, inner xsep=1pt, inner ysep=0pt,
        anchor=south] at ($(M.north)+(0,6pt)$) (Q) {\scriptsize$Q$};
  \coordinate (fblow) at ($(M.south)+(0,-12pt)$);
  \coordinate (pot) at ($(fin.center)+(-16pt,-27pt)$);
  \coordinate (potSW) at ($(pot)+(-15pt,-1pt)$);
  \coordinate (potNE) at ($(pot)+(15pt,7pt)$);
  \node[bb={0}{0}, fit=(M) (Q) (fin) (fout) (fblow) (potSW) (potNE),
        inner xsep=.55cm, inner ysep=8pt] (N) {};
  \node[coordinate] at (N.west|-M_in1) (N_in1) {};
  \node[coordinate] at (N.east|-M_out1) (N_out1) {};
  \coordinate (Nin) at ($(N_in1)+(-5pt,0)$);
  \coordinate (Nout) at ($(N_out1)+(5pt,0)$);
  \draw[thin, gray!55, ann ar] (Q) to[out=180, in=90] (fin);
  \draw[thin, gray!55, ann ar] (Q) to[out=0, in=90] (fout);
  \draw[ar] (Nin) to (fin);
  \draw[ar] (fin) to (M_in1);
  \draw (M_out1) to (split);
  \draw[ar] (split) to (fout);
  \draw[ar] (fout) to (Nout);
  \coordinate (fbmid) at ($(fin.-45|-fblow)!.5!(split|-fblow)$);
  \draw[ar] (split) to[out=-90, in=0] (fbmid) to[out=180, in=-45] (fin.-45);
  \draw[thin, black!60, -{Stealth[length=3.5pt]}] (fin.-135) to[out=-135, in=45] node[left=1pt, font=\scriptsize] {$U$} ($(pot)+(-2pt,8pt)$);
  \begin{scope}[shift={(pot)}]
    \draw[thin, black!60] plot[domain=-1.45:1.45, samples=41] ({\x*10pt},{(\x*\x-1)^2*5pt});
    \fill[black!60] (-10pt,1.8pt) circle (1.7pt);
  \end{scope}
  \node[above=.3pt] at ($(M_in1)+(-9pt,0)$) {\scriptsize$\inpt M$};
  \node[above=.3pt] at ($(M_out1)!.5!(split.west)$) {\scriptsize$\outp M$};
  \node[left=-1pt] at (Nin) {\scriptsize$\inpt N$};
  \node[right=0pt] at (Nout) {\scriptsize$\outp N$};
\end{tikzpicture}
\end{equation}

The first two pieces, $\outp f$ and $\inpt f$, are exactly the shape of a lens $\binom{\inpt f}{\outp f}$. Lenses are a standard language for bidirectional interaction, appearing in bidirectional programming, compositional game theory, and categorical accounts of backpropagation~\cite{foster2007combinators,hedges2017coherence,fong2019backprop,cruttwell2021categorical}. In our manifold setting, these two maps route output and input data between the internal boxes and the external box in a composable way.

To accommodate the potential---the last ingredient---we pass from lenses to the coKleisli category $\plmfd$ for the comonad on $\lmfd$ induced by the monad $\rr\times\blank$ on $\mfd$; this is the monad-with-strength construction of \cref{prop.backward_comonad}. Informally, the backward direction now handles both inputs and a real number. Thus we are now ready to define $\sarr$ to be the operad underlying the symmetric monoidal category $\para{\rvect}{\plmfd}$; we call its morphisms \emph{smooth adaptive arrangements}. Under the dynamics functor $\sarr\to\pc$, the parameters and potential let the effective coupling not only transport information but also contribute to the parameter update rules, again in a compositional way.

The passage from lenses to polynomial coalgebras uses a functor we call \emph{lens internalization}. In any symmetric monoidal closed category, a lens interface $\lensob c$ can be internalized as $\ihom{\inpt c,I}\otimes\outp c$, and this process constitutes a lax symmetric monoidal functor (\cref{cor.Theta}).

Together, $\cot$ and lens internalization turn ordinary smooth data---output map, input map, a responsive parameter, and a potential---into data for a discrete dynamical system, i.e.\ a polynomial coalgebra. At that point, the remaining choice is the state set $S$ that the coalgebra carries and how an incoming covector updates it. For example, in Hamiltonian dynamics, we not only store the data of position but also a momentum, which carries the position forward. We package this choice as an \emph{integrator}: a state space $\Fun S$, depending functorially on $Q$, together with a rule that updates the current state according to an incoming covector. Each integrator $\intg$ gives a dynamics functor $\Phi_\intg\colon\sarr\to\pc$.

Thus the integrator is where the dynamical regime gets decided: one smooth adaptive arrangement fed through different integrators can produce qualitatively different behaviors. We work out two. We show that the \emph{configuration integrator} ($\Fun S_Q\coloneqq Q$) generalizes both Newton's method (\cref{sec.newton_warmup}) and training an artificial neural network (\cref{sec.dl_warmup}). Indeed, when applied to the arrangement corresponding to a parameterized neural network, it recovers gradient descent, with backpropagation as the lens backward pass. 
We also show that the \emph{phase integrator} ($\Fun S_{Q}\coloneqq T^*Q$) applied to any finite directed graph of harmonic particles yields the discrete dynamics of its graph Laplacian (\cref{sec.graph_laplacian}), generalizing the wave equation. Fed the \emph{very same} adaptive arrangement---identical wiring, potential, and response---the configuration integrator instead gives the corresponding discrete heat equation (\cref{rmk.graph_heat}). Thus the second-order and first-order regimes are two semantics of a single syntactic object, differing mainly in terms of
whether to store position and momentum data or just position data.

This is the sense in which smooth adaptive arrangements are syntax: the diagram fixes the wiring, the potential, and the geometry of the parameter space, while the integrator alone selects the dynamical regime, second-order or first-order.
On the phase side, we will see in \cref{prop.closed_conservation} that each step of the phase dynamics applied to a closed harmonic system---one with quadratic response and quadratic potential---preserves the canonical symplectic pairing exactly.

That is a lot of moving parts, and almost none of them are special to manifolds: the construction needs only that spaces form a cartesian category, parameters a symmetric monoidal one, and potentials a monoidal monad (the writer monad of $(\rr,0,+)$ being the main instance). We therefore build $\sarr\to\pc$ in two passes. \Cref{ch.framework} carries out the construction at this level of generality---isolating what actually makes it work, and leaving room to specialize it differently later (e.g.\ \cref{ex.syntactic_datum})---and \cref{ch.smooth_rwd,ch.smooth_dynamics} then instantiate it in the smooth setting above.

The paper aims to be self-contained for a reader comfortable with monoidal categories and (co)monads and who has a passing familiarity with smooth manifolds; the background section reviews polynomial functors, the Para construction, and the 2-category $\pc$, while lenses and lens internalization have their own section (\cref{ch.lenses}).

The main contributions are the following.
\begin{itemize}
\item \emph{Lens internalization} (\cref{cor.Theta}): a lax symmetric monoidal functor $\Theta\colon\Lens{\cat C}\to\cat C$ that internalizes each lens interface as an object of $\cat C$, as well as a monad/monoid generalization $\Theta_\potalg$ (\cref{thm.Theta_T_alpha}).
\item A vocabulary for smooth adaptive arrangements---\emph{stateless}, \emph{potential-free}, \emph{static}, \emph{quadratic}, \emph{linear}, \emph{unilateral}---each class spanning a suboperad of $\sarr$ (\cref{def.arrangement_terminology,prop.suboperads}).
\item A closed system of two species wired to a predation potential, with oppositely-signed responses, whose sharp and potential generate the Lotka--Volterra ODEs (\cref{ex.lotka_volterra}).
\item Two functorial semantics $\Phiconf,\Phiphase\colon\sarr\to\pc$ (\cref{cor.functor}), each assembling a system's discrete dynamics from its smooth adaptive arrangement.
\item A proof that, for a harmonic closed system, each step of the phase dynamics preserves the canonical symplectic pairing exactly (\cref{prop.closed_conservation}).
\item $\Phiconf$ recovers gradient descent with backpropagation (\cref{sec.dl_warmup}), and Newton's method (\cref{sec.newton_warmup}).
\item One adaptive arrangement of harmonic particles yields, via $\Phiphase$, the discrete wave equation and, via $\Phiconf$, the discrete heat equation. Both are governed by the graph Laplacian (\cref{sec.graph_laplacian,rmk.graph_heat}).
\end{itemize}

The next section places this work alongside the open-energy-driven framework of Capucci--Lynch--Spivak; readers interested only in the constructions of this paper may skip to \cref{sec.plan}.

\section{Relation to open energy-driven systems}\label{sec.related_cls}

The closest predecessor of this paper is Capucci--Lynch--Spivak~\cite{capucci2024organizing} (hereafter CLS), which identifies the response $T^*X\to TX$---what they call a \emph{reaction}---in both Hamiltonian mechanics and gradient descent and builds a functorial semantics
\[
\mathbf{OpenErg}\To{T^*}\mathbf{OpenODE}\To{\mathrm{collapse}}\mathbf{COrg}
\]
 from their category of \emph{open energy-driven systems}, through \emph{open ODEs}, to $\mathbf{COrg}$, which is a continuous-time analog of $\pc$ (also known as $\mathbb{O}\Cat{rg}$). A first difference is that our semantics $\Phiconf,\Phiphase\colon\sarr\to\pc$ is discrete-time, rather than continuous.

However, neither their framework nor the one in this paper subsumes the other. There is a common subcategory $\cat{R}$ whose objects are manifolds and whose morphisms $M\to N$ are---in CLS notation---tuples $(\rv Q,w,E)$ where
\[
\rv Q:\rvect,\qquad
w\colon Q\times M\to N,\qquad
E\colon Q\times M\to\rr,
\]
and there is a wide strong monoidal functor $\cat{R}\to\mathbf{OpenErg}$ and a fully faithful operad functor $\cat{R}\to\sarr$, the latter sending $M$ to $\binom{1}{M}$, output map $\outp f\coloneqq w$, unique input map $\inpt f\coloneqq \bang$, and potential $U\coloneqq E$. 
In other words, the two extensions of $\cat{R}$ go in orthogonal directions: $\sarr$ has richer interfaces and additional input-routing data in its morphisms, while CLS allows manifold (rather than vector-space) state spaces. We conjecture that the latter gap is removable by replacing $\rvect$ with a category of responsive manifolds equipped with connections (\cref{rmk.rmfdc_generalization}), but we do not pursue this here.

There is also a semantic comparison between the two.

\begin{proposition*}[Proven as \cref{prop.euler_submersion_lenses}]
For $\cat R$ as above, there is a category $\cat S$, whose objects are submersions, and a diagram
\[
\begin{tikzcd}[column sep=large, row sep=large]
\mathbf{OpenErg}\ar[d, "{\mathrm{collapse}\circ T^*}"'] &
\cat{R}\ar[l, "{\mathrm{wide}}"']\ar[r, "{\mathrm{ff}}"]\ar[d] &
\sarr\ar[d, "\Phiconf"]\\
\mathbf{COrg} &
\cat{S}\ar[l, "{\mathrm{wide}}"]\ar[r] &
\pc
\end{tikzcd}
\]
where the middle vertical functor sends $M$ to the submersion $T^*M\surj M$, and the lower right functor sends a submersion $p\colon E\surj B$ to the polynomial $\sum_{b:B}\yon^{E_b}$.
\end{proposition*}
The point is that, on the overlap $\cat R$, both semantics are built from the same smooth data: cotangent pullback along $w$, the differential of $E$, and the response $\sharpR_Q$. The category $\cat{S}$ records this common data, at which point it can be read in either of two ways: as a continuous-time morphism in $\mathbf{COrg}$ on the CLS side, or as a discrete coalgebra in $\pc$ on our side, where the tangent component $\dot q$ becomes the Euler step $q\mapsto q+\dot q$.

Another related framework is that of bond graphs, a graphical language for network modeling in which components are joined by edges, each carrying an effort and a flow whose product is power~\cite[\S1]{vanderschaft2014port}.%
\footnote{
We will talk about bonds in this paper; see
\cref{rmk.symmetric_bonds}. Our bonds are like chemical bonds---stateless interactions with binding energy---whereas a bond in the sense of bond graphs is a route along which power passes, which is just an identification that a wiring diagram can make.
}
 In \cite{coya2017bondgraphs}, Coya categorified bond graphs using cospans and corelations and assigned two functorial semantics in Lagrangian relations: one in terms of effort and flow and the other in terms of potential and current. Bond graph semantics are relational and choose no causal orientation, and arbitrary bond graphs need not admit an evolution at all. By contrast, every expression in $\sarr$ denotes a deterministic discrete-time system (\cref{thm.dynamics_functor}), and correspondingly we do not express idealized constraints between systems; e.g.\ a ``rigid'' coupling can only be specified as a stiff potential, rather than an equation.

\section{Plan of the paper}\label{sec.plan}

\Cref{ch.background} fixes notation and recalls the categorical background used later: polynomial functors, coalgebras, and the 2-category $\pc$; wiring-diagram operads; and the Para construction.

\Cref{ch.lenses} reviews lenses and the backward monad construction and builds the \emph{lens internalization} functor (\cref{cor.Theta}), which turns a lens interface into an object of the ambient category.

\Cref{ch.framework} is the abstract heart of the paper. From an \emph{adaptive-arrangement datum} $D$---a cartesian category of spaces, a parameter category, and a potentials monad---it builds a syntax operad $\arr_D$. An interface functor into $\poly$, a semantic effect carrying the potentials monad into $\poly$, and a potential algebra then give a \emph{polynomial interpretation} via lens internalization (\cref{cor.Theta}), and an \emph{integrator} turns this into an operad functor $\arr_D\to\pc$ (\cref{thm.dynamics_functor}).

\Cref{ch.smooth_rwd} instantiates the syntax: it introduces responsive vector spaces and assembles the smooth ingredients into the adaptive-arrangement datum $\Sm$, whose syntax operad is the operad $\sarr$ of smooth adaptive arrangements---$\rvect$-parameterized manifold lenses equipped with potentials (\cref{def.potlens}).

\Cref{ch.smooth_dynamics} instantiates the dynamics by applying the framework of \cref{ch.framework} to the datum $\Sm$: the cotangent polynomial interpretation and the two integrators---configuration and phase---together yield operad functors $\Phiconf\colon\sarr\to\pc$ and $\Phiphase\colon\sarr\to\pc$ (\cref{cor.functor}); we then unpack the differential geometry underlying $\Phiconf$ and $\Phiphase$ in full detail.

\Cref{ch.applications} works out the main examples: Newton's method, gradient descent with backpropagation, the discrete wave equation for a chain of harmonic particles, and the graph-Laplacian dynamics of harmonic particles on an arbitrary finite directed graph.

\section*{Acknowledgments}
Thanks to Jake Araujo-Simon, Lukas Barth, David Jaz Myers, Maximilien P\'eroux, Eduardo Sontag, Ryan Tavenner, and Scott Viteri for stimulating conversations. I also thank Claude and Codex for assistance in editing the document and for various helpful conversations.

\thanksAFOSR{FA9550-23-1-0376}. 

\chapter{Background}\label{ch.background}

This section collects much of the background required for this paper. We first fix our typographic conventions. Then we recall polynomial functors, polynomial coalgebras, and the 2-category $\pc$ they assemble into. Finally we review wiring-diagram operads and the Para construction, which together with the lenses of \cref{ch.lenses} provide the compositional ingredients used in the rest of the paper.

\section{Notational conventions}\label{sec.prelim}

We fix some typographic conventions. We denote generic categories using calligraphic script, $\cat{C},\cat{M},\cat{A}$, and specifically-named categories using boldface, $\defineTerm{smset},\poly,\mfd$. We write specifically-named functions or operations in roman, $\ev,\coev$, and named functors in sans-serif, $\Fun{T}$. Number systems use blackboard bold, $\nn,\rr$; named bicategorical structures are written with a blackboard-bold initial, such as $\pc$ and $\Para$. 

For convenience in the electronic version, many recurring pieces of notation are hyperlinked to where they are defined, e.g.\ $\pc$ is clickable.

We write $\blank$ for a ``blank'' placeholder argument in expressions, e.g.\ $\rr\times\blank\colon\mfd\to\mfd$ for the functor $M\mapsto\rr\times M$. We write $\defineTerm{bang}$ for a unique map when the domain and codomain make it clear, e.g.\ the unique map $\emptyset\to A$ or the unique map $X\to 1$. We write $\inc$ for any inclusion functor that is clear from its domain and codomain categories, and we sometimes elide them, e.g. sometimes writing $F$ for $F\circ\inc$. In any category, we may use an object name as the name of the identity morphism on that object, e.g.\ $X\otimes f$ in place of $\id_X\otimes f$. For example, $\cat{C}$ can stand as the name of the identity functor $\id_\cat{C}$. For composable maps $A\To{f}B\To{g}C$, we may write $f\then g$ or $g\circ f$ as their composite, based on convenience.

We reserve $\eta$ and $\mu$ for units and multiplications, and $\varepsilon$ and $\delta$ for comonoid counits and comultiplications. Since $\eta$ is also the usual notation for a learning rate, we write learning rates as $\defineTerm{etaLR}$.

We sometimes refer to a $(2,1)$-category simply as a category. More precisely, its homs are groupoids and composition is associative and unital up to coherent isomorphism.

By \emph{operad} we mean \emph{colored operad} unless otherwise stated. We sometimes elide the difference between monoidal categories $(\cat{C},I,\otimes)$ and their underlying operads $\cat{O}_{\cat{C}}$, which have $K$-ary morphisms $c_1\otimes\cdots\otimes c_K\to c'$. Lax monoidal functors between monoidal categories correspond exactly to functors between their underlying operads.%
\footnote{That is, the functor $\Cat{SMC}_{\tn{lax}}\to\Cat{Opd}$ is fully faithful.}

\section{Polynomial functors, coalgebras, and \texorpdfstring{$\termraw{pc}$}{PC}}\label{sec.poly_dynamic_bg}

This section reviews the polynomial-functor machinery underlying our model of dynamics: the symmetric monoidal closed category $\poly$ (\cref{sec.poly}), the store comonad (\cref{sec.substitution_store}), polynomial coalgebras as discrete-time dynamical systems (\cref{sec.coalgebras}), and the monoidal 2-category $\pc$ of polynomial coalgebras (\cref{sec.pc}).

\subsection{The symmetric monoidal closed category \texorpdfstring{$\termraw{poly}$}{Poly}}\label{sec.poly}

We briefly review the category $\defineTerm{poly}$ of polynomial functors; see~\cite{niu2025polynomial} for a comprehensive treatment.

For any set $A$, let $\yon^A\colon\smset\to\smset$ denote the functor $\smset(A,\blank)$. A \emph{polynomial functor} is a coproduct $p\coloneqq\sum_{i: I}\yon^{A_i}$ of such functors. We refer to $I$ as the set of \emph{positions} of $p$ and $A_i$ is the set of \emph{directions} at position $i$. For example $(\yon^2+\yon^1)\colon\smset\to\smset$ is the polynomial that sends a set $X$ to $(X\times X)+X$; it has two positions, with two and one direction, respectively.

The category $\poly$ has polynomial functors as objects and natural transformations as morphisms. It has all categorical sums and products, and these are given by the usual sum and product of polynomials, e.g.\ $(\yon^2+\yon)\times\yon\cong\yon^3+\yon^2$. Since $p(1)=\sum_{i:I}1=I$, we use the compact notation $p(1)$ for the position set and $p[i]\coloneqq A_i$ for the direction set at $i: p(1)$, so that
\begin{equation}\label{eqn.poly}
p=\sum_{i: p(1)}\yon^{p[i]}=\sum_{i: p(1)}\;\prod_{d: p[i]}\yon.
\end{equation}
We use juxtaposition for multiplication: $pp'\coloneqq p\times p'$, and we write $\defineTerm{yon}\coloneqq\yon^1$ and $1\coloneqq\yon^0$.

One can derive from the universal property of coproducts and the Yoneda lemma that the maps in $\poly$, i.e.\ the natural transformations $\varphi\colon p\to p'$, are given by:
\begin{equation}\label{eqn.poly_map}
\poly(p,p')=\prod_{i: p(1)}\;\sum_{j: p'(1)}\;\prod_{e: p'[j]}\; \sum_{d: p[i]}1.
\end{equation}
The outer $\prod_i\sum_j$ yields a ``forward on positions'' function $\varphi_1\colon p(1)\to p'(1)$, which is the component of $\varphi$ at $1:\smset$, and the inner $\prod_e\sum_d$ yields, for each $i:p(1)$, a ``backward on directions'' function from the directions of $p'$ at $\varphi_1(i)$ to the directions of $p$ at $i$; we denote it as follows:%
\footnote{We extend the directions bracket to maps: where $p[i]$ is the set of directions of $p$ at $i$, $\bk{\varphi}{i}$ is the backward map $\varphi$ induces there, so $[i]$ acts on a map as it does on a polynomial. This map is often denoted $\varphi^\sharp_i$; we avoid the $\sharp$ because $\sharpR$ is used throughout the paper in the context of responsive vector spaces.}
\[\bk{\varphi}{i}\colon p'[\varphi_1(i)]\to p[i].\]

For polynomials $p,p':\poly$, their \emph{Dirichlet product} is
\begin{equation}\label{eqn.dirichlet}
p\otimes p'\coloneqq\sum_{(i,j):p(1)\times p'(1)}\yon^{p[i]\times p'[j]}\;.
\end{equation}
Together with unit $\yon$, the Dirichlet product forms a symmetric monoidal structure, which is also closed: the \emph{internal hom} is
\begin{equation}\label{eqn.dirichlet_hom}
\ihom{p,p'}=\prod_{i: p(1)}\;\sum_{j: p'(1)}\;\prod_{e: p'[j]}\;\sum_{d: p[i]}\yon
\end{equation}
satisfying the hom-tensor adjunction
\begin{equation}\label{eqn.dirichlet_adjunction}
\poly(a\otimes p,\, p')\cong\poly(a,\,\ihom{p,p'}),
\end{equation}
naturally in $a$. Note that $\ihom{p,p'}(1)=\poly(p,p')$, as is immediate from comparing \eqref{eqn.dirichlet_hom} with \eqref{eqn.poly_map}. We can rewrite it in standard form as
\begin{equation*}
\ihom{p,p'}\cong\sum_{\varphi: \poly(p,p')}\yon^{\sum_{i: p(1)}p'[\varphi_1(i)]}.
\end{equation*}

\subsection{The store comonad}\label{sec.substitution_store}

\begin{definition}[The store comonad]\label{def.store}
For any set $S$, define $\defineTerm{Store}(S)\coloneqq S\yon^S$. It arises as the comonad associated to the adjunction $S\times\blank\;\dashv\;(\blank)^S$ on $\smset$, i.e.\ the comonad with underlying endofunctor $X\mapsto S\times X^S$.
\end{definition}

\begin{lemma}\label{lem.store_monoidal}
The assignment $S\mapsto\Store(S)$, sending a bijection $f\colon S\To{\cong}T$ to $f\yon^{f\inv}$, defines a strong symmetric monoidal functor
\[
\Store\colon(\defineTerm{smsetiso},1,\times)\to(\poly,\yon,\otimes).
\]
\end{lemma}
\begin{proof}
The Dirichlet product formula \eqref{eqn.dirichlet} gives $1\yon^1=\yon$ and $S\yon^S\otimes T\yon^T=(S\times T)\yon^{S\times T}$. Functoriality on bijections is direct: $\Store(g\circ f)=(g\circ f)\yon^{f\inv\circ g\inv}=\Store(g)\circ\Store(f)$.
\end{proof}

\subsection{Polynomial coalgebras as dynamical systems}\label{sec.coalgebras}

For a polynomial $p$, a \emph{$p$-coalgebra} is a pair $(S,f)$ where $S:\smset$ is a set and $f\colon S\to p(S)$ is a function. We call elements of $S$ \emph{states}. Unpacking $p(S)=\sum_{i:p(1)} S^{p[i]}$, the function $f$ is equivalently a pair
\begin{equation}\label{eqn.coalg_unpack}
f_1\colon S\to p(1),\qquad
\bk{f}{\blank}\colon\prod_{s:S} S^{p[f_1(s)]},
\end{equation}
where $f_1(s)$ returns a position and $\bk{f}{s}\colon p[f_1(s)]\to S$ returns a state-update function---from directions to new states---for each $s:S$. Given a state $s:S$, we call $r_s\coloneqq f_1(s):p(1)$ its \emph{readout} position and $\bk{f}{s}\colon p[r_s]\to S$ its \emph{update} function, sending each direction $d:p[r_s]$ to a new state $s_d\coloneqq\bk{f}{s}(d)$.

A morphism $(S,f)\to(S',f')$ of $p$-coalgebras is a function $g\colon S\to S'$ such that $p(g)\circ f=f'\circ g$. We denote the category of $p$-coalgebras and their morphisms by $p\coalg$.

When $p=B\yon^A$ then a $p$-coalgebra $S\to B\times S^A$ is equivalently a readout function $S\to B$ and an update function $S\times A\to S$, which is independent of the readout. 
Such a coalgebra is called an $(A,B)$-\emph{Moore machine}, for which the canonical reference is~\cite{moore1956gedanken}. 

A $p$-coalgebra is thus a deterministic dynamical system whose interface is $p$: the state exposes an output in $p(1)$ and transitions in response to inputs in $p[f_1(s)]$, capturing Moore machines, discretized differential equations, Markov decision processes,%
\footnote{
 Writing $\defineTerm{rand}\coloneqq\sum_{N:\nn}\sum_{P:\fun{Dist}(N)}\yon^N$, whose positions are finitely-supported distributions and whose directions are outcomes, a partially observable Markov decision process with rewards in $R:\smset$ can be modeled as a coalgebra for the composite polynomial functor $O\yon^A\circ\rand\circ R\yon$, i.e.\ a map $S\to O\times\bigl(\sum_{(N,P)}(R\times S)^N\bigr)^A$: the state emits an observation in $O$, receives an action in $A$, emits a distribution, receives an outcome, and emits a reward in $R$.
See~\cite{Spivak2023lotteries} for its monad structure.
} and other mode-dependent dynamics in a single framework. It is a special case of the general theory of coalgebras for endofunctors; see~\cite{jacobs2017introduction} for a comprehensive treatment.

There is a connection between $p$-coalgebras and the $\Store$ comonad from \cref{def.store}.

\begin{proposition}[Coalgebras as polynomial maps]\label{prop.coalg_as_poly_map}
For $p:\poly$ and $S:\smsetiso$, there is a natural bijection
\begin{equation}\label{eqn.store_coalg}
\smset(S,\,p(S))\;\cong\;\poly(\Store(S),\,p).
\end{equation}
\end{proposition}

\begin{proof}
As in \eqref{eqn.coalg_unpack}, a coalgebra $f\colon S\to p(S)$ unpacks as a readout $f_1\colon S\to p(1)$ and an update $\bk{f}{\blank}\colon\prod_{s:S}S^{p[f_1(s)]}$. A polynomial map $\phi\colon S\yon^S\to p$ unpacks as a pair $\phi_1\colon S\to p(1)$ and $\bk{\phi}{\blank}\colon\prod_{s:S}S^{p[\phi_1(s)]}$ with the same type signature by \eqref{eqn.poly_map}, so they are in bijection.

Naturality in $p:\poly$ is just composition. For naturality in $S$, a bijection $g\colon S\To{\cong}T$ acts on the left side of \eqref{eqn.store_coalg} via post-composition with $p(g)$ and pre-composition with $g\inv$, and on the right side of \eqref{eqn.store_coalg} via $\Store(g)=g\yon^{g\inv}$ (\cref{lem.store_monoidal}); we check
\[
\begin{tikzcd}
	S\ar[r,"f"]\ar[d, "g"']&p(S)\ar[d, "p(g)"]\\
	T\ar[r,"f'"']&p(T)
\end{tikzcd}
\quad\leftrightsquigarrow^?\quad
\begin{tikzcd}
	S\yon^S\ar[r,"\phi"]\ar[d, "g\yon^{g\inv}"']&
	p\ar[d, equal]\\
	T\yon^T\ar[r,"\phi'"']&
	p
\end{tikzcd}
\]
The first commutes iff $f'_1\circ g=f_1$ and $g\circ\bk{f}{s}=\bk{f'}{g(s)}$ for all $s:S$; the second iff $\phi'_1\circ g=\phi_1$ and $\bk{\phi}{s}=g\inv\circ\bk{\phi'}{g(s)}$. These coincide under the established bijection $\phi_\blank=f_\blank$.\qedhere
\end{proof}

\begin{example}[$\yon$-coalgebras and trajectories]\label{ex.y_coalgebra_trajectory}
The simplest case is $p\coloneqq\yon$, in which case its positions $p(1)=1$ and directions $p[*]=1$ are both singletons, and $\yon(S)=S$, so a $\yon$-coalgebra is just a function $f\colon S\to S$. The readout is trivial and the update takes no input direction: at each tick, the state $s$ transitions to $f(s)$. Given an initial state $s_0\in S$, iteration produces a function
\[
s\colon\nn\to S,\qquad s(0)\coloneqq s_0,\qquad s(t+1)\coloneqq f(s(t)).
\]
We call this the \emph{trajectory} of the $\yon$-coalgebra from initial state $s_0$, and use this notation freely below.
\end{example}

\subsection{The monoidal 2-category $\pc$ of polynomial coalgebras}\label{sec.pc}

The semantics category of this paper, called $\mathbb{O}\Cat{rg}$ in \cite{spivak2021learners,shapiro2022dynamic}, arises naturally from a simple categorical construction: \emph{change of base of enrichment} along a lax monoidal functor.\footnote{I thank Maximilien P\'eroux for pointing out that this construction is a change of base of enrichment.}

Recall that for any monoidal category $(\cat{V},I,\otimes)$, a \emph{$\cat{V}$-category} $\cat{C}$ has objects as usual, but instead of having sets of morphisms, it has $\cat{V}$-objects: $\cat{C}(c_1,c_2)\in\ob\cat{V}$. The identities and compositions are given by maps 
\[
I\to\cat{C}(c,c)
\qqand
\cat{C}(c_1,c_2)\otimes\cat{C}(c_2,c_3)\too\cat{C}(c_1,c_3)
\]
in $\cat{V}$, which satisfy unitality and associativity equations in $\cat{V}$. For example, any monoidal closed category---having an internal hom $\ihom{\blank,\blank}$---is self-enriched: $\cat{C}(c,c')\coloneqq\ihom{c,c'}\in\ob\cat{C}$ and there are natural maps $I\to\ihom{c,c}$ and $\ihom{c_1,c_2}\otimes\ihom{c_2,c_3}\to\ihom{c_1,c_3}$ that satisfy unitality and associativity.

The category of $\cat{V}$-categories is denoted $\enrcat{\cat V}$. For any lax monoidal functor $F\colon\cat V\to\cat W$, there is an induced functor $F_*\colon\enrcat{\cat V}\to\enrcat{\cat W}$ that carries each $\cat V$-enriched category to a $\cat W$-enriched category with the same objects, applying $F$ to the hom-objects and transporting the identities and composites along the lax structure maps of $F$. When $\cat{V}$ and $\cat{W}$ are symmetric monoidal and $F$ is symmetric lax monoidal, $F_*$ carries monoidal (resp.\ monoidal closed) $\cat{V}$-categories to monoidal (resp.\ monoidal closed) $\cat{W}$-categories.

For a polynomial $q$, let $q\coalg$ denote the category of $q$-coalgebras $(S,f\colon S\to q(S))$ and their morphisms; the assignment $q\mapsto q\coalg$ is a lax symmetric monoidal functor, which we write
\[
\Coalg\colon(\poly,\yon,\otimes)\to(\defineTerm{smcat},1,\times).
\]
Indeed, given $S\yon^S\to p$ and $T\yon^T\to q$, one obtains $ST\yon^{ST}\cong S\yon^S\otimes T\yon^T\to p\otimes q$. 

Since $(\poly,\yon,\otimes,\ihom{\blank,\blank})$ is monoidal closed (\cref{sec.poly}), it is self-enriched and since $\Coalg$ is lax monoidal, we can form the $\smcat$-enriched category
\[\defineTerm{pc}\coloneqq\Coalg_*\poly.\]
Its objects are polynomials and its hom-objects are
\[
\pc(p,p')\coloneqq\ihom{p,p'}\coalg.
\]
The identity $p\to p$ in $\pc$ is the one-state coalgebra $1\To{\id_p}\poly(p,p)=\ihom{p,p}(1)$.

This 2-category has been discussed under the name $\mathbb{O}\Cat{rg}\cong\pc$ previously~\cite{spivak2021learners,smithe2022open,shapiro2022dynamic}.

\begin{proposition}\label{prop.pc}
$\pc$ is symmetric monoidal closed: the monoidal product $(\yon,\otimes)$ is induced by the Dirichlet product on polynomials, and the internal hom is $\ihom{p,p'}$.
\end{proposition} 
\begin{proof}
Since $(\poly,\yon,\otimes,\ihom{\blank,\blank})$ is symmetric monoidal closed and $\Coalg$ is lax symmetric monoidal, the result follows from the change-of-base discussion at the top of \cref{sec.pc}.
\end{proof}

We often elide the difference between $\pc$ as a monoidal 2-category and $\pc$ as a category-enriched operad, with maps
\[
\pc(p_1,\ldots,p_K;p')\coloneqq\ihom{p_1\otimes\cdots\otimes p_K,p'}\coalg.
\]
The operadic perspective is closer to the intended semantics---internal interfaces $p_1,\ldots,p_K$ within a larger external interface $p'$---than the monoidal perspective is, but the operadic notation $\pc(p_1,\ldots,p_K;p')$ is heavier than the monoidal notation $\pc(p,p')$ is.

To be explicit, a morphism $p\to p'$ in $\pc$ is an $\ihom{p,p'}$-coalgebra $(S,\,f\colon S\to\ihom{p,p'}(S))$. We refer to $S$ as the \emph{state set}. We regard $(S,\,f)$ as an interaction, regarding $p$ as an \emph{internal} interface and $p'$ as an \emph{external} one: at each state $s\in S$ the coalgebra converts each internal output $i\in p(1)$ into an external output $j\in p'(1)$, then converts each external input $e\in p'[j]$ into an internal input $d\in p[i]$ together with a new state $s'\in S$. This can be read off from the formula
\[
\smset(S,\ihom{p,p'}(S))\cong\prod_{s:S}\prod_{i: p(1)}\;\sum_{j: p'(1)}\;\prod_{e: p'[j]}\;\sum_{d: p[i]}\sum_{s':S}1.
\]
Since $\ihom{\yon,p}\cong p$, a morphism $\yon\to p$ in $\pc$ is precisely a $p$-coalgebra $S\to p(S)$.

Another instance of the construction $F_*\poly$ will be useful. Recall that one can always replace a category with its \emph{core groupoid}: the same objects but only the isomorphisms. The functor $\Fun{core}\colon\smcat\to\Cat{Gpd}$ preserves finite products, so the composite $\Fun{core}\circ\Coalg\colon\poly\to\Cat{Gpd}$ is lax monoidal. We can thus define $\defineTerm{pciso}\coloneqq(\Fun{core}\circ\Coalg)_*\poly$,
\[
\pciso(p,p')\coloneqq\Fun{core}\bigl(\ihom{p,p'}\coalg\bigr)\cong\sum_{S:\smsetiso}\smset(S,\ihom{p,p'}(S)).
\]
It is the wide sub-$(2,1)$-category of $\pc$ with the same objects $p:\poly$, same 1-cells  $S\To{f}\ihom{p,p'}(S)$, but only the \emph{invertible} 2-cells $(S,f)\cong (S',f')$ between them. The inclusion $\Fun{core}\Rightarrow\id_{\smcat}$ is monoidal, and base-changing it gives the inclusion $\pciso\inj\pc$.

\section{Wiring-diagram operads}\label{sec.wd_operads}

Categories, monoidal categories, traced categories, operads, etc.\ can be understood in terms of the sort of composition syntax they interpret. For example, each of the following diagrams is interpretable in the sort of structure named beneath it:
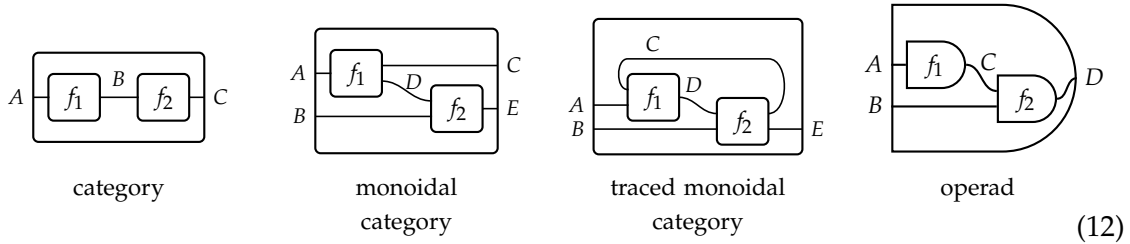
\begin{equation}\label{eqn.wd_examples}
\begin{tikzpicture}[baseline=(Catg)]
\begin{scope}[font=\footnotesize, text height=1.5ex, text depth=.5ex]
  \begin{scope}[oriented WD, bb port sep=1, bb port length=0, bb min width=.4cm, bby=.2cm, inner xsep=.2cm, x=.5cm, y=.3cm]
  	\node[bb={1}{1}] (Catf) {$f_1$};
  	\node[bb={1}{1}, right=1 of Catf] (Catg) {$f_2$};
  	\node[bb={0}{0}, inner ysep=8pt, fit=(Catf) (Catg)] (Cat) {};
  	\node[coordinate] at (Cat.west|-Catf_in1) (Cat_in1) {};
  	\node[coordinate] at (Cat.east|-Catg_out1) (Cat_out1) {};
  	\draw (Cat_in1) -- (Catf_in1);
  	\draw (Catf_out1) -- (Catg_in1);
  	\draw (Catg_out1) -- (Cat_out1);
  	\draw[label]
  		node[left=3pt of Cat_in1] {$A$}
  		node at ($(Catf_out1)!.5!(Catg_in1)+(0,.8)$) {$B$}
  		node[right=3pt of Cat_out1] {$C$}
  	;
  	\node[bb={1}{2}, above right=-1 and 6.1 of Catf] (Monf) {$f_1$};
  	\node[bb={2}{1}, below right=-.5 and 1.3 of Monf] (Mong) {$f_2$};
  	\node[bb={0}{0}, inner ysep=8pt, fit=(Monf) (Mong)] (Mon) {};
  	\node[coordinate] at (Mon.west|-Monf_in1) (Mon_in1) {};
  	\node[coordinate] at (Mon.west|-Mong_in2) (Mon_in2) {};
  	\node[coordinate] at (Mon.east|-Monf_out1) (Mon_out1) {};
  	\node[coordinate] at (Mon.east|-Mong_out1) (Mon_out2) {};
  	\draw  (Mon_in1) -- (Monf_in1);
  	\draw (Monf_out1) -- (Mon_out1);
  	\draw (Monf_out2) to (Mong_in1);
  	\draw (Mon_in2) -- (Mong_in2);
  	\draw (Mong_out1) -- (Mon_out2);
  	\draw[label]
  		node[left=3pt of Mon_in1] {$A$}
  		node[left=3pt of Mon_in2] {$B$}
  		node[right=3pt of Mon_out1] {$C$}
  		node at ($(Monf_out2)!.5!(Mong_in1)+(.2,.4)$) {$D$}
  		node[right=3pt of Mon_out2] {$E$}
  	;
  	\node[bb={2}{1}, below right= -1 and 6.5 of Monf] (Trf) {$f_1$};
  	\node[bb={2}{2}, below right=-1 and 1 of Trf] (Trg) {$f_2$};
  	\node[bb={0}{0}, fit={($(Trf.north west)+(-.5,2)$) ($(Trg.south east)+(.5,0)$)}] (Tr) {};
  	\node[coordinate] at (Tr.west|-Trf_in2) (Tr_in1) {};
  	\node[coordinate] at (Tr.west|-Trg_in2) (Tr_in2) {};
  	\node[coordinate] at (Tr.east|-Trf_out1) (Tr_out1) {};
  	\node[coordinate] at (Tr.east|-Trg_out2) (Tr_out2) {};
  	\draw (Tr_in1) -- (Trf_in2);
  	\draw (Trf_out1) to (Trg_in1);
  	\draw (Tr_in2) -- (Trg_in2);
  	\draw (Trg_out2) -- (Tr_out2);
  	\draw let \p1=(Trg.east), \p2=(Trf.north west), \n1=\bbportlen, \n2=\bby in
  		(Trg_out1) to[in=0] (\x1+\n1,\y2+\n2) -- (\x2-\n1,\y2+\n2) to[out=180] (Trf_in1);
  	\draw[label]
  		node[left=3pt of Tr_in1] {$A$}
  		node[left=3pt of Tr_in2] {$B$}
  		node at ($(Trf_out1)!.4!(Trg_in1)+(0,.8)$) {$D$}
  		node[right=3pt of Tr_out2] {$E$}
  		node at ($(Trf.north)+(0,1.3)$) {$C$}
  	;
  \end{scope}
  \begin{scope}[circuit logic US, thick, every to/.style={out=0,in=180}]
  	\node[and gate, draw, above right=-.2 and 3 of Trf] (Opdf) {$f_1$};
  	\node[and gate, draw, below right=-.1 and 0.5 of Opdf] (Opdg) {$f_2$};
  	\node[and gate, inner ysep=-4pt, draw, fit={($(Opdf.north west)+(0,-.2)$) ($(Opdg.south east)+(0,.2)$)}] (Opd) {};
  	\node[coordinate] at (Opd.input 1|-Opdf.west) (Opd_in1) {};
  	\node[coordinate] at (Opd.input 1|-Opdg.input 2) (Opd_in2) {};
  	\node[coordinate] at (Opd.output) (Opd_out) {};
  	\draw (Opd_in1) to (Opdf.west);
  	\draw (Opd_in2) to (Opdg.input 2);
  	\draw (Opdf.output) to (Opdg.input 1);
  	\draw (Opdg.output) to (Opd_out);
  	\draw[label]
  		node[left=-1pt of Opd_in1] {$A$}
  		node[left=-1pt of Opd_in2] {$B$}
  		node at ($(Opdf.output)!.5!(Opdg.input 1)+(.1,.2)$) {$C$}
  		node[right=-1pt of Opd_out] {$D$}
  	;
  \end{scope}
	\node[below=.3 of Cat.south] (Cat name) {category};
	\node[text width=1.5cm, align=center] at (Cat name-|Mon) {monoidal\\category};
	\node[text width=2.5cm, align=center] at (Cat name-|Tr) {traced monoidal\\category};
	\node at (Cat name-|Opd) {operad};
\end{scope}
\end{tikzpicture}
\end{equation}
\smallskip

\noindent In each case there is an associated operad $\cat{W}$ whose objects are the legal interfaces (boxes-with-object-labeled-ports) and whose morphisms $\varphi\colon b_1,\ldots,b_K\to c$ are the legal arrangements of internal boxes in an external box. That is, for a set $O$ of objects, there are operads $\cat{W}_{O\tn{-Cat}}$, $\cat{W}_{O\tn{-MnCat}}$, $\cat{W}_{O\tn{-TrCat}}$, and $\cat{W}_{O\tn{-Opd}}$ each with a 2-ary morphism that can be drawn as in \eqref{eqn.wd_examples}, because each of those wiring diagrams has two internal boxes.%
\footnote{
There is a potentially confusing level-shift in that there is an operad $\cat{W}_{O\tn{-Opd}}$ for which algebras are operads with object set $O$. In case it's helpful, the object set for $\cat{W}_{O\tn{-Opd}}$ would be $\List(O)\times O$, because every ``box'' has a list of inputs and an output.
}

Category-like structures that can interpret the compositionality syntax from $\cat{W}$ are precisely the $\cat{W}$-algebras, i.e.\ lax monoidal functors
\begin{equation}\label{eqn.W_alg}
\cat{C}\colon\cat{W}\to\smset.
\end{equation}
So categories are lax monoidal functors $\cat{W}_{O\tn{-Cat}}\to\smset$ for various $O$, monoidal categories are lax monoidal functors $\cat{W}_{O\tn{-MnCat}}\to\smset$ for various $O$, etc. Indeed, such a functor $\cat{C}$ assigns to each box $b$ the set $\cat{C}(b)$ of ``fillers'' and to each arrangement $\varphi$ the function $\cat{C}(\varphi)\colon \cat{C}(b_1)\times\cdots\times \cat{C}(b_K)\to \cat{C}(c)$ sending fillers for internal boxes to a filler of the external box.%
\footnote{
By a $\cat{W}$-algebra, we mean an operad functor $\cat{W}\to\smset$, where $\smset$ is the operad (multicategory) underlying the monoidal category $(\smset,1,\times)$. In general, the \emph{underlying operad} of a symmetric monoidal category $(\cat{C},I,\otimes)$ has the same objects as $\cat{C}$ and morphisms $\cat{C}(a_1,\ldots,a_n;b)\coloneqq\cat{C}(a_1\otimes\cdots\otimes a_n,b)$, with operadic composition by tensoring and composing in $\cat{C}$.
}

We'll make this explicit for categories.
\begin{example}[$O$-categories as $\cat{W}_{O\tn{-Cat}}$-algebras]
For a set $O$, let $\ob\cat{W}_{O\tn{-Cat}}\coloneqq O\times O$ and
\[\cat{W}_{O\tn{-Cat}}((o_0',o_1),(o_1',o_2),\ldots(o_{K-1}',o_K);(o_0,o_K'))\coloneqq
\begin{cases}
\{1\}&\tn{ if }o_k=o_k'\quad \forall\; 0\leq k\leq K\\
\varnothing&\tn{ otherwise}
\end{cases}
\]
Thus the diagram
\begin{equation}\label{eqn.cat_composite}
\begin{tikzpicture}[oriented WD, bb port sep=1, bb port length=0, bb min width=.2cm, bby=.2cm, inner xsep=.2cm, x=.2cm, y=.3cm, baseline=(Catf)]
\begin{scope}[font=\footnotesize, text height=1.5ex, text depth=.5ex]
  \node[bb={1}{1}] (Catf) {$f_1$};
  \node[bb={1}{1}, right=1 of Catf] (Catg) {$f_2$};
  \node[bb={0}{0}, fit=(Catf) (Catg)] (Cat) {};
  \node[coordinate] at (Cat.west|-Catf_in1) (Cat_in1) {};
  \node[coordinate] at (Cat.east|-Catg_out1) (Cat_out1) {};
  \draw (Cat_in1) -- (Catf_in1);
  \draw (Catf_out1) -- (Catg_in1);
  \draw (Catg_out1) -- (Cat_out1);
  \draw[label]
    node[left=3pt of Cat_in1] {$A_0$}
    node[above left =2pt of Catf_in1] {$A'_0$}
    node[above right=2pt of Catf_out1] {$A_1$}
    node[above left=2pt of Catg_in1] {$A'_1$}
    node[above right=2pt of Catg_out1] {$A_2$}
    node[right=3pt of Cat_out1] {$A'_2$}
  ;
\end{scope}
\end{tikzpicture}
\end{equation}
	represents a map $\fun{compose}\colon (A'_0,A_1),(A'_1,A_2)\to(A_0,A'_2)$ in $\cat{W}_{O\tn{-Cat}}$, in the case $A_0=A_0'$, $A_1=A_1'$ and $A_2=A_2'$, i.e.\ in the case that each wire in the diagram has a consistent type
\[
\fun{compose}\coloneqq
	\begin{tikzpicture}[oriented WD, bb port sep=1, bb port length=0, bb min width=.2cm, bby=.2cm, inner xsep=.2cm, x=.2cm, y=.3cm, baseline=(Catf)]
\begin{scope}[font=\footnotesize, text height=1.5ex, text depth=.5ex]
  \node[bb={1}{1}] (Catf) {$f_1$};
  \node[bb={1}{1}, right=1 of Catf] (Catg) {$f_2$};
  \node[bb={0}{0}, fit=(Catf) (Catg)] (Cat) {};
  \node[coordinate] at (Cat.west|-Catf_in1) (Cat_in1) {};
  \node[coordinate] at (Cat.east|-Catg_out1) (Cat_out1) {};
  \draw (Cat_in1) -- (Catf_in1) node[label, midway, above] {$A_0$};
  \draw (Catf_out1) -- (Catg_in1) node[label, midway, above] {$A_1$};
  \draw (Catg_out1) -- (Cat_out1) node[label, midway, above] {$A_2$};
\end{scope}
\end{tikzpicture}
\hphantom{\fun{compose}\coloneqq}
\]

If $\cat{C}\colon\cat{W}_{O\tn{-Cat}}\to\smset$ is an algebra, then \eqref{eqn.cat_composite} shows $f_1:\cat{C}(A_0,A_1)$, $f_2:\cat{C}(A_1,A_2)$ and $\cat{C}(\fun{compose})(f_1,f_2):\cat{C}(A_0,A_2)$. In other words the $2$-ary morphism $\varphi$ represents composition 
\[
\cat{C}(\fun{compose})\colon\cat{C}(A_0,A_1)\times\cat{C}(A_1,A_2)\to\cat{C}(A_0,A_2)
\]
in a category. The category of small categories with object set $O$ is equivalent to the category of $\cat{W}_{O\tn{-Cat}}$-algebras.%
\footnote{It is not hard to let the objects change as well. For any function $O\to O'$, there is an induced operad functor $\cat{W}_{O\tn{-Cat}}\to\cat{W}_{O'\tn{-Cat}}$ and hence a functor $\cat{W}_{O'\tn{-Cat}}\alg\to\cat{W}_{O\tn{-Cat}}\alg$. The Grothendieck construction of this data is equivalent to the category of small categories. The same idea applies to all of the wiring diagram operads.}
\end{example}

Each such $\cat{W}$ is a \emph{composition syntax}: its morphisms specify legal arrangements of boxes---how systems compose---with no commitment about what fills them. A $\cat{W}$-algebra is a semantics. A variety of examples of these wiring-diagram operads, whose algebras are structured categories, can be found in the literature. \cite{Spivak.Schultz.Rupel:2016a} treats the traced and compact categories, showing in particular that the operad $\cat{W}_{\tn{TrCat}}$ for traced categories is the operad $1\tn{-}\Cat{Cob}$ of oriented 1-dimensional cobordisms. \cite{fong2019hypergraph} treats hypergraph categories, whose associated operad $\cat{W}_{\tn{HypCat}}$ is given by cospans of finite sets. The symmetric monoidal case~\cite{patterson2021wiring} is in some sense the most basic of the three, but its operad $\cat{W}_{\tn{MnCat}}$ is less easy to describe directly.

An advantage of the operadic perspective is that it allows us to more uniformly treat various sorts of category, including ones that may not have an existing name. For example, 
consider the notion of ``traced monoidal categories without identity''; one might balk at this notion, worrying that it might not be coherent for one reason or another. But as an operad it is perfectly sensible and well-defined because wiring diagrams of this sort---disallowing wires that pass straight through the diagram---can be nested.

Another advantage is that wiring diagram algebras also easily represent $\cat{V}$-enriched categories for monoidal categories $\cat{V}$, namely by using algebras $\cat{W}\to\cat{V}$ valued in $\cat{V}$ in place of $\smset$ as in \eqref{eqn.W_alg}. For example, a $\vect$-enriched category with object set $O$ is an operad functor $\cat{W}_{O\tn{-Cat}}\to\vect$.

For our purposes, the goal is to build \emph{dynamic} versions of the structures in \eqref{eqn.wd_examples}, in which the ``fillers'' of a box are ``morphisms that change in time'', i.e.\ coalgebras. In \cite{shapiro2022dynamic}, a dynamic category is defined to be a category enriched in $\pc$; similarly for dynamic monoidal categories or dynamic operads. Given the ideas above, these are just operad functors
\[
\cat{W}_{O\tn{-Cat}}\to\pc
\qquad
\cat{W}_{O\tn{-MnCat}}\to\pc
\qquad
\cat{W}_{O\tn{-Opd}}\to\pc
\] 
for any $O:\smset$. We continue with this approach, taking values in the 2-category $\pc$ of \cref{sec.pc}. In \cref{sec.potentialized_lenses} we will choose the particular composition syntax $\cat{W}$ used in this paper.

\section{The Para construction}\label{sec.para_general}

Recall that an \emph{action} of a symmetric monoidal category $(\cat A,I,\otimes)$ on a category $\cat D$ is a functor $(\cdot)\colon\cat A\times\cat D\to\cat D$ equipped with coherent natural isomorphisms $I\cdot d\cong d$ and $(a\otimes b)\cdot d\cong a\cdot(b\cdot d)$; such data is also called an \emph{$\cat A$-actegory}~\cite{capucci2024actegories}. When $\cat D$ too is symmetric monoidal, the action functor $(\cdot)$ is strong monoidal, and the action coherences are monoidal natural isomorphisms, $\cat D$ is called a \emph{monoidal $\cat A$-actegory}~\cite[Def.~5.1.1]{capucci2024actegories}.

\begin{lemma}[Strong monoidal functors give actions]\label{lem.smc_action}
Any strong monoidal functor $J\colon\cat A\to\cat D$ with $\cat D$ symmetric monoidal induces an action of $\cat A$ on $\cat D$, making $\cat D$ a monoidal $\cat A$-actegory, by left multiplication,
\[
a\cdot d\coloneqq J(a)\otimes d.
\]
\end{lemma}

We use the following form of the Para construction. If a symmetric monoidal category $\cat{A}$ acts on a category $\cat{D}$ (\cref{lem.smc_action}), the category $\defineTerm{Para}_{\cat{A}}(\cat{D})$ has the same objects as $\cat{D}$, and a morphism $d_1\to d_2$  in $\para{\cat{A}}{\cat{D}}$ is a pair $(a,f)$ with $a:\cat{A}$ and
\[
f\colon a\cdot d_1\to d_2
\]
in $\cat{D}$. We call such morphisms \emph{parameterized maps} in $\cat D$. The identity on $d_1$ has parameter $I$ and action given by the unitor $I\cdot d_1\cong d_1$. The composite of $(a,f)\colon d_1\to d_2$ and $(b,g)\colon d_2\to d_3$ has parameter $b\otimes a$ and uses
\begin{equation*}
(b\otimes a)\cdot d_1\cong b\cdot(a\cdot d_1)\To{b\cdot f}b\cdot d_2\To{g}d_3,
\end{equation*}
i.e.\ $(b,g)\circ(a,f)=(b\otimes a, g\circ (b\cdot f))$.

\begin{lemma}[Para of a monoidal actegory {\cite[\S5.1, eq.~(5.7); \S5.4]{capucci2024actegories}}]\label{lem.para_monoidal}
When $\cat D$ is a monoidal $\cat A$-actegory, $\para{\cat A}{\cat D}$ is symmetric monoidal: on objects, $d_1\otimes d_2$ is the $\cat D$-tensor, and the tensor of maps $(a,f)\colon d_1\to d_2$ and $(a',f')\colon d_1'\to d_2'$ has parameter $a\otimes a'$ and underlying map
\[
(a\otimes a')\cdot(d_1\otimes d_1')\To{\cong}(a\cdot d_1)\otimes(a'\cdot d_1')\To{f\otimes f'}d_2\otimes d_2',
\]
the first isomorphism coming from $\cat D$ being a monoidal $\cat A$-actegory.
\end{lemma}

The Para construction extends to a bicategory: a 2-cell $(a,f)\Rightarrow(a',f')$ in the hom-category $\para{\cat A}{\cat D}(d_1,d_2)$ consists of a morphism $g\colon a\to a'$ in $\cat A$ such that $f'\circ(g\cdot d_1)=f$. When $\cat A$ is a groupoid, all 2-cells are invertible, and each hom-category $\para{\cat A}{\cat D}(d_1,d_2)$ is a groupoid.
In this paper, all of our parameter categories are groupoids, so these Para constructions are naturally $(2,1)$-categories; by the convention of \cref{sec.prelim}, we will usually call them categories.

The following is straightforward.

\begin{proposition}[Action squares from monoidal natural transformations]\label{prop.para_strong_induced}
Let $J\colon\cat A\to\cat D$ and $J'\colon\cat A'\to\cat D'$ be strong monoidal functors, inducing actions of $\cat A$ on $\cat D$ and of $\cat A'$ on $\cat D'$ by left multiplication (\cref{lem.smc_action}). A monoidal natural transformation $\theta\colon J'\circ F\Rightarrow G\circ J$ as on the left, with $F$ strong monoidal and $G$ lax (resp.\ strong) monoidal, induces a lax (resp.\ strong) monoidal action square as on the right:
\[
\begin{tikzcd}[ampersand replacement=\&]
\cat A\ar[r,"J"]\ar[d,"F"', ""{name=left}]\& \cat D\ar[d,"G", ""' {name=right}]\\
\cat A'\ar[r,"J'"']\& \cat D'
\arrow[Rightarrow, from=left, to=right, shorten <=3mm, shorten >=3mm, "\theta"]
\end{tikzcd}
\quad\leadsto\quad
\begin{tikzcd}[ampersand replacement=\&]
\cat A\times\cat D\ar[r,"(\cdot)"]\ar[d,"F\times G"', ""{name=left}]\& \cat D\ar[d,"G", ""' {name=right}]\\
\cat A'\times\cat D'\ar[r,"(\cdot')"']\& \cat D'.
\arrow[Rightarrow, from=left, to=right, shorten <=6mm, shorten >=6mm]
\end{tikzcd}
\]
\end{proposition}
Unwinding, the induced 2-cell is the composite
\[
F(a)\cdot' G(d)=J'(F(a))\otimes' G(d)\To{\theta_a\otimes' G(d)}G(J(a))\otimes' G(d)\To{}G(J(a)\otimes d)=G(a\cdot d).
\]

\begin{proposition}[Para functoriality from action squares]\label{prop.para_square}
Suppose $\cat A,\cat A'$ are symmetric monoidal categories acting on $\cat D,\cat D'$, that $F\colon\cat A\to\cat A'$ is a strong monoidal functor, $G\colon\cat D\to\cat D'$ is a functor, and $\theta\colon F(a)\cdot' G(d)\to G(a\cdot d)$ is a natural transformation compatible with the action unitors and associators, i.e.\ an action square as in~\cite[\S 3.6, eq.~(3.59)]{capucci2024actegories}:
\begin{equation*}
\begin{tikzcd}[ampersand replacement=\&]
\cat A\times\cat D\ar[r,"(\cdot)"]\ar[d,"F\times G"', ""{name=left}]\& \cat D\ar[d,"G", ""' {name=right}]\\
\cat A'\times\cat D'\ar[r,"(\cdot')"']\& \cat D'.
\arrow[Rightarrow, from=left, to=right, shorten <=6mm, shorten >=6mm, "\theta"]
\end{tikzcd}
\end{equation*}
Then there is an induced pseudo-functor
\[
\para[\theta]{F}{G}\colon\para{\cat A}{\cat D}\;\longrightarrow\;\para{\cat A'}{\cat D'},
\]
which is strict if $F$ is strict monoidal.
When $\cat D$ is a monoidal $\cat A$-actegory and $\cat D'$ a monoidal $\cat A'$-actegory, $G$ is strong (resp.\ normal lax, resp.\ lax) symmetric monoidal, and $\theta$ is a monoidal natural transformation, then $\para[\theta]{F}{G}$ extends to a strong (resp.\ normal lax, resp.\ lax) monoidal functor.
\end{proposition}

\begin{proof}
On objects $\para[\theta]{F}{G}(d)\coloneqq G(d)$, and a parameterized map $(a,f)\colon d_1\to d_2$ goes to $(F(a),\hat f)$ with $\hat f\colon F(a)\cdot' G(d_1)\To{\theta}G(a\cdot d_1)\To{G(f)}G(d_2)$. Functoriality comes from the compatibility of $\theta$ with the action's unitors and associators, as well as $F$'s productor isomorphism $F(a_1)\otimes' F(a_2)\cong F(a_1\otimes a_2)$. In fact, when $F$ is strict this productor isomorphism is the identity, in which case we find also that $\para[\theta]{F}{G}$ is strict. For the monoidal claim, $\para[\theta]{F}{G}$ agrees with $G$ on objects and because $\theta$ is monoidal, the comparison cells of $\para[\theta]{F}{G}$---the unit $I_{\cat D'}\to G(I_{\cat D})$ and the multiplication $G(d_1)\otimes' G(d_2)\to G(d_1\otimes d_2)$---are precisely those of $G$. Hence it is strong (resp.\ normal lax, resp.\ lax) monoidal exactly when $G$ is.\qedhere
\end{proof}

When $\theta$ is the identity, we drop it from the notation.

\chapter{Lenses and their internalization}\label{ch.lenses}

This section reviews the category $\Lens{\cat C}$ of lenses in a symmetric monoidal category $\cat C$, develops the \emph{backward-monad} construction that will let us add a real-valued potential to the backward direction, and builds the \emph{lens internalization} functor (\cref{sec.lens_internalization}). The lens and backward-monad material is standard, lightly adapted.

\section{Lenses}\label{sec.lenses}

The adaptive arrangements we will build have boxes on the inside and a box on the outside, and some sort of (possibly varying) message-passing wiring between them. Each box includes two objects $(\inpt c, \outp c)$ in some other category (for us, eventually $\mfd$), their input and output side; these are the objects in the category $\Lens{\cat{C}}$.

\subsection{The category of lenses}\label{sec.category_of_lenses}
We first recall the category $\defineTerm{Lens}_{\cat{C}}$ of lenses in a symmetric monoidal category $(\cat{C},I,\otimes)$; the notion of lens originates, in its asymmetric form, with \cite[Def.~3.1]{foster2007combinators}, and the bimorphic form we use is treated in \cite{hedges2017coherence} and \cite[Def.~2.9]{spivak2019generalized}. Let $(\cat{C}, I,\otimes)$ be a symmetric monoidal category.

An \emph{interface} in $\cat C$ is a pair $\lensob c$ whose output part $\outp c$ carries a comonoid structure $(\outp c,\varepsilon,\delta)$.%
\footnote{
When $(\cat{C},1,\times)$ is cartesian monoidal, every object has a unique comonoid structure and every morphism is automatically a comonoid homomorphism; this is one direction of Fox's theorem~\cite{nlab:cartesian_monoidal}.
} 
A \emph{lens} $f\colon\lensob c\to\lensob d$ consists of a forward map $\outp f\colon\outp c\to\outp d$ (a comonoid homomorphism) and a backward map $\inpt f\colon\outp c\otimes\inpt d\to\inpt c$ (any map in $\cat{C}$). Such a morphism can be depicted as follows:
\begin{equation}\label{eqn.lens_map}
\begin{tikzpicture}[oriented WD, bb port length=5pt, bb port sep=1, bb min width=.4cm, bbx=.5cm, bby=.7ex, baseline=(fo)]
  \node[bb={1}{1}] (fo) {$\outp f$};
  \node[bb={1}{2}, above left=6 and 1.5 of fo] (fi) {$\inpt f$};
  \node[dot, left=1 of fo] (dot) {};
  \coordinate (co) at ($(dot)+(-3,0)$);
  \coordinate (do) at ($(fo.east)+(1,0)$);
  \coordinate (ci) at (co |- fi_in1);
  \coordinate (di) at (do |- fi_out2);
  \node[left=0 of co] {$\outp c$};
  \node[right=0 of do] {$\outp d$};
  \node[left=0 of ci] {$\inpt c$};
  \node[right=0 of di] {$\inpt d$};
  \draw[ar] (co) to (dot);
  \draw[ar] (dot) to (fo_in1);
  \draw[ar] (fo_out1) to (do);
  \draw[ar] (fi_in1) to (ci);
  \draw[ar] (di) to (fi_out2);
  \draw[ar] (dot) to[out=30, in=0] (fi_out1);
\end{tikzpicture}
\end{equation}
The identity lens has $\outp{\id_{\lensob{c}}}\coloneqq\id_{\outp{c}}$ and $\inpt{\id_{\lensob{c}}}\coloneqq\varepsilon\otimes\id_{\inpt{c}}$, i.e.\ use the counit to discard the $\outp c$ part.

Lenses compose as follows: given also $g\colon \lensob{d}\to\lensob{e}$, the composite $f\then g$ is given by
\[
\outp{(f\then g)}=\outp f\then\outp g
\qqand
\inpt{(f\then g)}=
	(\delta_{\outp{c}}\otimes\inpt{e})\then
	(\outp{c}\otimes\outp{f}\otimes\inpt{e})\then
	(\outp{c}\otimes\inpt{g})\then
	\inpt{f}.
\]
This dense notation can be depicted in a string-diagram as follows:
\begin{equation*}
\begin{tikzpicture}[oriented WD, bb port length=5pt, bb port sep=1, bb min width=.4cm, bbx=.5cm, bby=.7ex, baseline=(fo)]
  \node[bb={1}{1}] (fo) {$\outp f$};
  \node[bb={1}{1}, right=5 of fo] (go) {$\outp g$\vphantom{$\inpt f$}};
  \node[bb={1}{2}, above left=6 and 1.5 of fo] (fi) {$\inpt f$};
  \node[bb={1}{2}, above left=6 and 1.8 of go] (gi) {$\inpt g$\vphantom{$\inpt f$}};
  \node[dot, left=1 of fo] (dot1) {};
  \node[dot] at ($(fo.east)!.75!(go.west)$) (dot2) {};
  \coordinate (co) at ($(dot1)+(-3,0)$);
  \coordinate (eo) at ($(go.east)+(1,0)$);
  \coordinate (ci) at (co |- fi_in1);
  \coordinate (ei) at (eo |- gi_out2);
  \node[left=0 of co] {$\outp c$};
  \node[right=0 of eo] {$\outp e$};
  \node[left=0 of ci] {$\inpt c$};
  \node[right=0 of ei] {$\inpt e$};
  \draw[ar] (co) to (dot1);
  \draw[ar] (dot1) to (fo_in1);
  \draw[ar] (fo_out1) to node[above, font=\scriptsize] {$\outp d$} (dot2);
  \draw[ar] (dot2) to (go_in1);
  \draw[ar] (go_out1) to (eo);
  \draw[ar] (fi_in1) to (ci);
  \draw[ar] (ei) to (gi_out2);
  \draw[ar] (gi_in1) to[out=180, in=0] node[above, font=\scriptsize] {$\inpt d$} (fi_out2);
  \draw[ar] (dot1) to[out=30, in=0] (fi_out1);
  \draw[ar] (dot2) to[out=30, in=0] (gi_out1);
\end{tikzpicture}
\end{equation*}

\begin{definition}\label{def.lens}
For a symmetric monoidal category $(\cat C, I,\otimes)$, the category $\Lens{\cat C}$ of \emph{interfaces} and \emph{lenses} has as objects the pairs $\lensob c$ with $\outp c$ a comonoid $(\varepsilon_{\outp c},\delta_{\outp c})$, and as morphisms $\lensob c\to\lensob d$ the pairs $\binom{\inpt f}{\outp f}$ with $\outp f\colon\outp c\to\outp d$ a comonoid homomorphism and $\inpt f\colon\outp c\otimes\inpt d\to\inpt c$ any map in $\cat{C}$, with identity and composition as above.

It is symmetric monoidal with unit $\lensunit[I]$ and
\[
\lensob{c}\ltens\lensob{d}\coloneqq\lensobs{c}{d}.
\qedhere
\]
\end{definition}

Our main example will be $\cat{C}=\mfd$ with its cartesian product. But the following is simple, instructive, and will come in handy.
\begin{example}\label{ex.lens_finsetop}
The category $\defineTerm{finset}\op$, opposite to the category of finite sets, is cartesian monoidal where the product is given by the disjoint union $(0,+)$, i.e.\ the coproduct in $\finset$. Thus every object carries a canonical comonoid structure $A\from 0$ and $A\from A+A$, and every morphism preserves it automatically. 

An interface in $\Lens{\finset\op}$ is just a pair of finite sets $\binom{\inpt A}{\outp A}$. A morphism $\binom{\inpt f}{\outp f}\colon\binom{\inpt A}{\outp A}\to\binom{\inpt B}{\outp B}$ consists of two maps, together sometimes called a \emph{prism}:
\begin{equation}\label{eqn.lens_finsetop}
\outp A\From{\outp f} \outp B
\qqand
\outp A+\inpt B\From{\inpt f} \inpt A,
\end{equation}
with no further homomorphism condition, as it is automatic.
\end{example}

\begin{proposition}\label{prop.lens_functoriality}
The $\Lens{\blank}$ construction is functorial in symmetric monoidal categories and strong monoidal functors.
\end{proposition}
\begin{proof}
If $F\colon\cat{C}\to\cat{D}$ is colax monoidal, it preserves comonoid structures, so we can define $\Lens F$ on objects. If $F$ is lax monoidal, it has a productor $F_2\colon F(A)\otimes F(B)\to F(A\otimes B)$ from which we can define $\Lens F$ on morphisms: given a lens morphism $\binom{\inpt f}{\outp f}$ we set $\Lens F\binom{\inpt f}{\outp f}\coloneqq\binom{F(\inpt f)\circ F_2}{F(\outp f)}$. So if $F$ is strong, we can define $\Lens F$ on objects and morphisms and the invertibility of the coherence maps assures functoriality; we leave the details to the reader.
\qedhere
\end{proof}

\begin{example}
The functor $\blank\yon\colon(\smset,1,\times)\to(\poly,\yon,\otimes)$ sending $A\mapsto A\yon$ is strong monoidal: $A\yon\otimes B\yon\cong (A\times B)\yon$. As such it induces a strong monoidal functor
\begin{equation}\label{eqn.lens_yon}
\Lens{\blank\yon}\colon\Lens{\smset}\to\Lens{\poly}.\qedhere
\end{equation}
\end{example}

Since $\finset\op$ is the free cartesian category on one object, any object $M$ in a cartesian category $\cat{M}$ determines a strong monoidal functor $M^\blank\colon\finset\op\to\cat{M}$ sending $A\mapsto M^A$.

\begin{lemma}\label{lem.lens_pow}
For any cartesian category $\cat M$ and object $M:\cat M$, the functor $M^\blank\colon\finset\op\to\cat{M}$ is strong monoidal, hence induces a strong monoidal functor
\[
\Lens{M^\blank}\colon\Lens{\finset\op}\to\Lens{\cat M}.
\]
\end{lemma}

\subsection{Lenses can represent serial, parallel, and feedback}\label{sec.plenspot_ops}

Recall from \cref{sec.wd_operads} that operads---and hence monoidal categories---can be used to encode compositional syntax. In fact, we can think of the operad underlying $\Lens{\cat{C}}$ as giving a syntax for something like ``traced symmetric monoidal semicategories'', as we now explain. This section is meant for intuition and in particular does not define the previously quoted phrase. It will also not be used in any formal development.

\begin{warning}
The diagrams in the previous section, e.g.\ \eqref{eqn.lens_map}, are string diagrams for maps in $\cat{C}$. The wiring diagrams below as in \eqref{eqn.arb_wd} will denote maps in $\Lens{\cat{C}}$. A different way to say it: the pictures above are implicit in all the wiring diagram maps we'll choose, but are at a different level of abstraction, and we ask the reader not to confuse the two picture languages.
\end{warning}

By a \emph{semicategory} we mean almost a category, except without the notion of identities (like a semigroup is a monoid without identities). So to say that $\Lens{\cat{C}}$ is a composition syntax for something like traced symmetric monoidal semicategories, we should show that it interprets diagrams that look like this. 
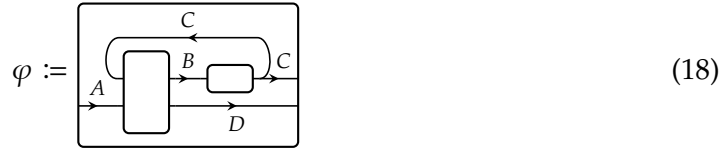
\begin{equation}\label{eqn.arb_wd}
\varphi\coloneqq
\begin{tikzpicture}[oriented WD, bb min width=.6cm, bb port sep=1, bbx=.6cm, bby=1ex, bb port length=2.5pt, baseline=(X2)]
  \node[bb port sep=2, bb={2}{2}] (X1) {};
  \node[bb={1}{1}, right=.7 of X1_out1] (X2) {};
  \node[bb={0}{0}, fit={(X1) (X2) ($(X1.north)+(0,2.5)$)}] (Y) {};
  \node[coordinate] at (Y.west|-X1_in2) (Y_in1) {};
  \node[coordinate] at (Y.east|-X2_out1) (Y_out1) {};
  \node[coordinate] at (Y.east|-X1_out2) (Y_out2) {};
  \draw[ar] (Y_in1) -- node[above, font=\scriptsize] {$A$} (X1_in2);
  \draw[ar] (X1_out1) -- node[above, font=\scriptsize] {$B$} (X2_in1);
  \draw[ar] (X2_out1) -- node[above right=0 and -5pt, font=\scriptsize] {$C$} (Y_out1);
  \draw[ar] (X1_out2) -- node[below, font=\scriptsize] {$D$} (Y_out2);
  \draw[ar] let \p1=(X2.north east), \p2=(X1.north west), \n1={\y2+\bby}, \n2=\bbportlen in
          (X2_out1) to[in=0,in looseness=.9] (\x1+.7*\n2,\n1) -- node[above, font=\scriptsize] {$C$} (\x2-.7*\n2,\n1) to[out=180] (X1_in1);
  \end{tikzpicture}
 \end{equation}

We will consider boxes as lenses $\binom{\inpt c}{\outp c}$, where the ports on the left of the box constitute the input $\inpt c$ and the ports on the right of the box constitute the output $\outp c$. The wiring diagram itself is a morphism $\varphi$ in $\Lens{\cat{C}}$ from the tensor product of the internal boxes to one external box.

For example, the wiring diagram \eqref{eqn.arb_wd} includes three boxes, two internal boxes and the external box, so it should be represented in $\Lens{\cat{C}}$ by a map of the form
\[
\varphi\colon\binom{C\otimes A}{B\otimes D}\ltens\binom{B}{C}\to\binom{A}{C\otimes D}
\]
where let us say that $B,C,D$ are comonoids. Unfolding, we need two maps in $\cat{C}$:
\begin{equation}\label{eqn.wd_as_maps}
\outp f\colon B\otimes D\otimes C\to C\otimes D
\qqand
\inpt f\colon (B\otimes D\otimes C)\otimes A\to C\otimes A\otimes B
\end{equation}
But these maps are obvious: they are just counit projections and symmetries.

One could take $\cat{C}\coloneqq(\mfd,1,\times)$. Then $A,B,C, D$ are manifolds and the wiring diagram represents information flow, shuffling coordinates around. But there are many more smooth functions of the form \eqref{eqn.wd_as_maps} than there are wiring diagrams, e.g.\ there is more than one smooth function $\rr\to\rr$.

One can also interpret \eqref{eqn.arb_wd} in $\finset\op$, the opposite of the category of finite sets. Taking the sets $A\coloneqq\{a\}$, $B\coloneqq\{b\}$, $C\coloneqq\{c\}$, and $D\coloneqq\{d\}$ all to have one element for simplicity. Now \eqref{eqn.wd_as_maps} consists of two functions
\[
\{b,d,c\}\From{\outp f}\{c,d\}
\qqand
\{b,d,c\}+\{a\}\From{\inpt f}\{c,a,b\}.
\]
These two functions entirely determine the wiring diagram \eqref{eqn.arb_wd} as an operad morphism, saying how each external outgoing port $c,d$ is ``fed by'' some internal outgoing port (in $\{b,d,c\}$), and how each internal ingoing port $c,a,b$ is ``fed by'' either an internal outgoing port (in $\{b,d,c\}$) or the external ingoing port, $a$. Applying $\Lens{M^\blank}$ from \cref{lem.lens_pow}, for any manifold $M\in\mfd$, realizes such wiring diagrams as lenses between manifolds.

\section{Backward monads on lenses}\label{sec.backwards}

Our potentials will arise from a commutative monoid, via the following standard fact.

\begin{lemma}\label{lem.monoid_to_monad}
If $\cat{C}$ is a symmetric monoidal category and $(R,0,+)$ is a commutative monoid object in it, then $R\otimes\blank\colon\cat{C}\to\cat{C}$ carries a monoidal monad structure, functorially in $(\cat C,R)$.
\end{lemma}

\begin{proof}
The monad unit and multiplication are $0\otimes\id$ and $(+)\otimes\id$. The monoidal structure has unitor $0\colon I\to R\cong R\otimes I$ and productor
\begin{equation}\label{eqn.monoid_productor}
(R\otimes X)\otimes(R\otimes Y)\cong R\otimes R\otimes X\otimes Y\To{(+)\otimes\id}R\otimes X\otimes Y.
\end{equation}
The coherence diagrams arise from the monoid laws and commutativity of $R$. Functoriality in $(\cat C,R)$ is straightforward.
\end{proof}

We will add a notion of potentials to our lenses by proving a more general result: any monad-with-strength on $\cat{C}$ induces a comonad on $\Lens{\cat{C}}$ by applying it to the input / backward-pass part. We pun on the word ``backwards'' to refer to the fact that it applies to the backward pass and the fact that, as such, it turns a monad on $\cat{C}$ into a comonad on $\Lens{\cat{C}}$.

\label{txt.monoidal_monad}Recall that a \emph{strength} on a monad $(\Fun{T},\eta,\mu)$ on a monoidal category $(\cat{C},I,\otimes)$ is a natural map $\sigma\colon X\otimes\Fun{T}Y\to\Fun{T}(X\otimes Y)$ satisfying four diagrams~\cite{kock1972strong}; a monad equipped with one is a \emph{monad-with-strength}. The monad is \emph{monoidal} if $\Fun{T}$ is lax monoidal (with unitor $\tau_0\colon I\to\Fun{T}(I)$ and productor $\tau_{X,Y}\colon \Fun{T}(X)\otimes\Fun{T}(Y)\to\Fun{T}(X\otimes Y)$) and if $\eta$ and $\mu$ are monoidal. Each monoidal monad carries a canonical strength, with $\sigma_{X,Y}\colon X\otimes \Fun{T}Y\to\Fun{T}(X\otimes Y)$ the composite
\[
X\otimes \Fun{T}Y\To{\eta_X\otimes\Fun{T}Y}\Fun{T}(X)\otimes \Fun{T}(Y)\To{\tau_{X,Y}}\Fun{T}(X\otimes Y).
\]

\begin{proposition}\label{prop.backward_comonad}
Let $(\Fun{T},\eta,\mu,\sigma)$ be a monad-with-strength on a symmetric monoidal category $(\cat{C},I,\otimes)$. Then there is a comonad $\defineTerm{LensFun}^{\Fun{T}}$ on $\Lens{\cat{C}}$ sending $\lensob{c}\mapsto \binom{\Fun{T}\inpt c}{\outp c}$. We denote its coKleisli category by $\defineTerm{Lcokl}\coloneqq(\Lens{\cat{C}})_{\LensFun{\Fun T}}$.

If moreover $\Fun{T}$ is a monoidal monad, then $\LensFun{\Fun T}$ is a colax monoidal comonad, thus the coKleisli category $(\Lens{\cat{C}}^{\Fun{T}},\lensunit[I],\ltens)$ is naturally monoidal.
\end{proposition}
\begin{proof}
On a lens map $\binom{\inpt f}{\outp f}$ with $\inpt f\colon\outp c\otimes\inpt d\to\inpt c$, define the forward component of $\LensFun{\Fun T}\binom{\inpt f}{\outp f}$ to be $\outp f$ and the backward component to be the following composite:
\begin{equation*}
\outp c\otimes\Fun{T}\inpt d\To{\sigma}\Fun{T}(\outp c\otimes\inpt d)\To{\Fun{T}\inpt f}\Fun{T}\inpt c.
\end{equation*}
The counit and comultiplication of $\LensFun{\Fun{T}}$ are given by the identity on the forward map and, on the backward map, by the unit and multiplication of the monad, together with the counit $\varepsilon_{\outp c}$:
\begin{equation*}
\outp c\otimes\inpt c\To{\varepsilon_{\outp c}\otimes\id}\inpt c\To{\eta}\Fun{T}\inpt c,
\end{equation*}
\begin{equation*}
\outp c\otimes\Fun{T}\Fun{T}\inpt c\To{\varepsilon_{\outp c}\otimes\id}\Fun{T}\Fun{T}\inpt c\To{\mu}\Fun{T}\inpt c.
\end{equation*}
We need to prove that $\LensFun{\Fun{T}}$ preserves identity and composition (is a functor) and that the functor is counital and coassociative. These four axioms correspond to the four axioms of a monad strength.

We include only the least trivial here: preservation of composition. For $f\colon\lensob c\to\lensob d$ and $g\colon\lensob d\to\lensob e$, it suffices to show that the following commutes (at which point we post-compose with $\Fun{T}\inpt f\colon\Fun{T}(\outp c\otimes\inpt d)\to\Fun{T}\inpt c$ to obtain equality of maps $\outp c\otimes\Fun{T}\inpt e\to\Fun{T}\inpt c$):
\[
\begin{tikzcd}[column sep=1.6em, row sep=3em, every label/.append style={font=\scriptsize}, font=\small, ampersand replacement=\&]
\outp c\otimes\Fun{T}\inpt e
  \ar[r, "\delta"]
  \ar[d, "\sigma"']
\& (\outp c)^{\otimes 2}\otimes\Fun{T}\inpt e
  \ar[r, "\outp f"]
  \ar[d, "\sigma"]
\&[5pt] \outp c\otimes\outp d\otimes\Fun{T}\inpt e
  \ar[r, "\sigma"]
  \ar[dr, "\sigma" description]
\& \outp c\otimes\Fun{T}(\outp d\otimes\inpt e)
  \ar[r, "\inpt g"]
  \ar[d, "\sigma"]
\&[5pt] \outp c\otimes\Fun{T}\inpt d
  \ar[d, "\sigma"]
\\
\Fun{T}(\outp c\otimes\inpt e)
  \ar[r, "\delta"']
\& \Fun{T}((\outp c)^{\otimes 2}\otimes\inpt e)
  \ar[rr, "\outp f"']
\&
\& \Fun{T}(\outp c\otimes\outp d\otimes\inpt e)
  \ar[r, "\inpt g"']
\& \Fun{T}(\outp c\otimes\inpt d)
\end{tikzcd}
\]
The first two squares are naturality of $\sigma$ in its first argument, the middle triangle is the associator axiom of strength~\cite[(1.8)]{kock1972strong}, and the last square is naturality of $\sigma$ in its second argument.

When $\Fun T$ is moreover a monoidal monad, we obtain a colax monoidal structure on $\LensFun{\Fun{T}}$: the counitor $\binom{\Fun T(I)}{I}\to\lensunit[I]$ amounts to $\tau_0\colon I\to\Fun T I$ and the coproductor map $\binom{\Fun{T}(\inpt c\otimes\inpt d)}{\outp c\otimes\outp d}\to\binom{\Fun{T}(\inpt c)\otimes\Fun T(\inpt d)}{\outp c\otimes\outp d}$ amounts to $\tau_{\inpt c,\inpt d}\colon \Fun{T}(\inpt c)\otimes\Fun T(\inpt d)\to \Fun{T}(\inpt c\otimes\inpt d)$.
\end{proof}

A map in the coKleisli category $\Lens{\cat{C}}^{\Fun{T}}$ is a map of the form $\binom{\Fun{T}\inpt c}{\outp c}\to\lensob{d}$ in $\Lens{\cat{C}}$, i.e.
\[
\outp f\colon \outp c\to\outp d
\qqand
\inpt f\colon \outp c\otimes\inpt d\to\Fun{T}\inpt c
\]
where $\outp f$ is a comonoid homomorphism.

%

\begin{lemma}\label{lem.lens_T_functoriality}
The construction of \cref{prop.backward_comonad} is functorial in $(\cat{C},\Fun{T})$: a strong symmetric monoidal $F\colon\cat C\to\cat C'$ and a morphism of monads-with-strength $\theta\colon F\circ\Fun T\Rightarrow\Fun T'\circ F$ induce
\[
\defineTerm{LensMor}_{F}^{\theta}\colon\Lens{\cat C}^{\Fun T}\to\Lens{\cat C'}^{\Fun T'}\;,\qquad
\lensob c\mapsto\binom{F(\inpt c)}{F(\outp c)},\qquad
\binom{\inpt f}{\outp f}\mapsto\binom{\theta\circ F(\inpt f)\circ F_2}{F(\outp f)},
\]
where $F_2$ is $F$'s productor. If $\Fun T,\Fun T'$ are monoidal monads and $\theta$ is a morphism of monoidal monads, $\LensMor{F}{\theta}$ is strong symmetric monoidal.
\end{lemma}

\begin{proof}
By \cref{prop.lens_functoriality}, $F$ induces $\Lens F\colon\Lens{\cat C}\to\Lens{\cat C'}$. By \cref{prop.backward_comonad}, $\Fun T$ and $\Fun T'$ induce comonads $\LensFun{\Fun T}$ and $\LensFun{\Fun T'}$. The monad morphism $\theta$ gives the data necessary to define a natural transformation,
\[
\gamma\colon\LensFun{\Fun T'}\circ\Lens F\Rightarrow \Lens F\circ\LensFun{\Fun T},
\]
whose component at $\lensob c$ has output map $\outp{\gamma_{\lensob c}}\coloneqq\id_{F(\outp c)}$ and input map $\inpt{\gamma_{\lensob c}}$ given by
\[
F(\outp c)\otimes F(\Fun T\inpt c)\To{\varepsilon_{F(\outp c)}\otimes\id}F(\Fun T\inpt c)\To{\theta_{\inpt c}}\Fun T'F(\inpt c).
\]
Its naturality follows from that of $\theta$ and the counit $\varepsilon_{F(\outp c)}$.

We define $\LensMor{F}{\theta}$ on objects by $F\lensob{c}\coloneqq\binom{F(\inpt{c})}{F(\outp c)}$. On morphisms, recall that
\[
\Lens{\cat C}^{\Fun T}\left(\lensob c,\lensob d\right)=\Lens{\cat C}\left(\LensFun{\Fun T}\lensob c,\lensob d\right).
\]
For $f\in\Lens{\cat C}^{\Fun T}(\lensob c,\lensob d)$, define $\LensMor{F}{\theta}(f)$ to be the element of $\Lens{\cat C'}^{\Fun T'}(\Lens F\lensob c,\Lens F\lensob d)$ represented by the composite
\[
\LensFun{\Fun T'}\circ\Lens F\lensob c\To{\gamma_{\lensob c}}\Lens F\circ\LensFun{\Fun T}\lensob c\To{\Lens F(f)}\Lens F\lensob d.
\]
Unfolding, its backward component is
\[
F(\outp c)\otimes F(\inpt d)\To{F_2}F(\outp c\otimes\inpt d)\To{F(\inpt f)}F(\Fun T\inpt c)\To{\theta_{\inpt c}}\Fun T'F(\inpt c),
\]
and its output component is $F(\outp f)$, as claimed. The axioms for $\theta$ as a morphism of monads-with-strength say exactly that $\gamma$ is compatible with the counits and comultiplications of the two comonads; hence $\LensMor{F}{\theta}$ preserves coKleisli identities and composition.

If the monads are monoidal and $\theta$ is monoidal, the natural transformation $\gamma$ is monoidal, so $\LensMor{F}{\theta}$ is strong symmetric monoidal.\qedhere
\end{proof}


The following is straightforward.

\begin{proposition}\label{prop.lens_inclusion}
Let $(\cat{C}, 1, \times)$ be a cartesian monoidal category and $\Fun{T}$ a monoidal monad on $\cat{C}$ with unit $\eta$. There is a strong symmetric monoidal functor
\begin{equation*}
\inc\colon\cat{C}\to\Lens{\cat{C}}^{\Fun{T}},\qquad X\mapsto\binom{1}{X},
\end{equation*}
sending $f\colon X\to Y$ to the coKleisli lens with $\outp f=f$ and $\inpt f=(\bang\then\eta_1)\colon X\times 1\To{\bang} 1\To{\eta_1}\Fun{T}(1)$.\end{proposition}

\begin{lemma}\label{lem.parameter_lens_action}
Let $\cat A$ and $\cat C$ be symmetric monoidal, let $\Fun T$ be a monad-with-strength on $\cat C$, and let $J\colon\cat A\to\cat C$ be a strong monoidal functor such that each object $J(a)$ carries a comonoid structure $(\delta_{J(a)},\varepsilon_{J(a)})$ in $\cat C$ and each morphism $J(u)$ is a comonoid homomorphism. Then we have an action $\cat A\times\Lens{\cat C}^{\Fun T}\to\Lens{\cat C}^{\Fun T}$ given by
\[
a\cdot\lensob c\coloneqq\binom{\inpt c}{J(a)\otimes\outp c}.
\]
If moreover $\Fun T$ is a monoidal monad, this action makes $\Lens{\cat C}^{\Fun T}$ a monoidal $\cat A$-actegory.
\end{lemma}

\begin{proof}
Given $f\colon\lensob c\to\lensob d$ in $\Lens{\cat C}^{\Fun T}$ and $u\colon a\to a'$ in $\cat A$, set
\[
u\cdot f\coloneqq\binom{(\varepsilon_{J(a)}\otimes\outp c\otimes\inpt d)\then\inpt f}{J(u)\otimes\outp f}\colon
\binom{\inpt c}{J(a)\otimes\outp c}
\to
\binom{\inpt d}{J(a')\otimes\outp d}.
\]
Checking functoriality reduces immediately to the strength axioms for $\sigma$---naturality in its first argument together with unit-coherence give $\sigma_{X,Y}\then\Fun T(\varepsilon_X\otimes\id_Y)=\varepsilon_X\otimes\id_{\Fun T Y}$---and the comonoid-homomorphism assumption on $J$, i.e.\ that $\varepsilon_{J(a)}$ equals the composite
\[
J(a)\To{J(u)}J(a')\To{\varepsilon_{J(a')}} I.
\]
Together with the coKleisli composition of \cref{prop.backward_comonad}, these give preservation of identities and composition.

When $\Fun T$ is monoidal, $\Lens{\cat C}^{\Fun T}$ is symmetric monoidal by \cref{prop.backward_comonad}, and the comparisons $J(I)\cong I$ and $(a\otimes a')\cdot(\lensob c\ltens\lensob d)\cong(a\cdot\lensob c)\ltens(a'\cdot\lensob d)$ come directly from the strong monoidality of $J$.
\end{proof}

\section{Lens internalization}\label{sec.lens_internalization}

This section introduces \emph{lens internalization}---a normal lax symmetric monoidal functor $\Theta\colon\Lens{\cat C}\to\cat C$ that turns each interface $\lensob c$ into the $\cat C$-object $\ihom{\inpt c,I}\otimes\outp c$---which \cref{ch.framework} will use to interpret adaptive-arrangement syntax as parameterized polynomial maps (\cref{thm.poly_interpretation}). We begin with a generalization.

All lax monoidal functors preserve monoids, so in particular monoidal monads (\cpageref{txt.monoidal_monad}) do.

\begin{definition}\label{def.T_monoid}
Let $(\Fun{T},\eta,\mu,\tau)$ be a monoidal monad on a symmetric monoidal category $(\cat{C},I,\otimes)$. A \emph{$\Fun{T}$-monoid in $\cat{C}$} is a $\otimes$-monoid $(z,e_z,m_z)$ in $\cat{C}$ together with a $\Fun{T}$-algebra structure $\alpha\colon\Fun{T}z\to z$
\[
\begin{tikzcd}[ampersand replacement=\&]
z\ar[r,"\eta_z"]\ar[dr,equal]\& \Fun{T}z\ar[d,"\alpha"]\\
\& z
\end{tikzcd}
\hspace{1in}
\begin{tikzcd}[ampersand replacement=\&]
\Fun{T}\Fun{T}z\ar[r,"\mu_z"]\ar[d,"\Fun{T}\alpha"']\& \Fun{T}z\ar[d,"\alpha"]\\
\Fun{T}z\ar[r,"\alpha"']\& z
\end{tikzcd}
\]
such that $\alpha$ is a monoid homomorphism
\begin{equation}\label{eqn.T_monoid}
\begin{tikzcd}[ampersand replacement=\&]
I\ar[r,"\tau_0"]\ar[d,"e_z"']\& \Fun{T}I\ar[r,"\Fun{T}e_z"]\& \Fun{T}z\ar[d,"\alpha"]\\
z\ar[r,"\eta_z"']\& \Fun{T}z\ar[r,"\alpha"']\& z
\end{tikzcd}
\hspace{1in}
\begin{tikzcd}[ampersand replacement=\&]
\Fun{T}z\otimes\Fun{T}z\ar[r,"\tau_{z,z}"]\ar[d,"\alpha\otimes\alpha"']\& \Fun{T}(z\otimes z)\ar[r,"\Fun{T}m_z"]\& \Fun{T}z\ar[d,"\alpha"]\\
z\otimes z\ar[rr,"m_z"']\&\& z
\end{tikzcd}
\end{equation}
We write $\defineTerm{potalg}\coloneqq(z,e_z,m_z,\alpha)$ for this data, with \emph{carrier} $z$.
\end{definition}

\begin{remark}\label{rmk.T_monoid_EM}
A $\Fun{T}$-monoid $\potalg$ in $\cat{C}$ is the same data as a $\Fun{T}$-algebra in $\Cat{Mon}_\otimes(\cat{C})$: $\Fun{T}$ being monoidal (see~\cite{nlab:monoidal_monad}) lets $\Fun{T}$ lift to a monad on $\Cat{Mon}_\otimes(\cat{C})$, and \eqref{eqn.T_monoid} is exactly the condition that the structure map $\alpha$ is a morphism there.
\end{remark}

Recall that in a symmetric monoidal closed category the functor $-\otimes a$ is left adjoint to the internal hom $\ihom{a,-}$ for each object $a$. We write $\ev_a\colon\ihom{a,b}\otimes a\to b$ for the counit (\emph{evaluation}) and $\coev_a\colon b\to\ihom{a,b\otimes a}$ for the unit (\emph{coevaluation}) of this adjunction; they satisfy the two triangle identities $(\coev_a\otimes a)\then\ev_a=\id_{b\otimes a}$ and $\coev_a\then[a,\ev_a]=\id_{[a,b]}$.

\begin{theorem}\label{thm.Theta_T_alpha}
Let $(\Fun{T},\tau)$ be a monoidal monad on a symmetric monoidal closed category $(\cat{C},I,\otimes, [-,-])$, and let $\potalg=(z,e,m,\alpha)$ be a $\Fun{T}$-monoid in $\cat{C}$ (\cref{def.T_monoid}). Then there is a lax symmetric monoidal functor that acts on objects as follows:
\[
\Theta_\potalg\colon\Lens{\cat{C}}^\Fun{T}\to\cat{C},\qquad
\lensob{c}\mapsto\ihom{\inpt c,z}\otimes \outp c
\]
It is normal if and only if $e\colon I\to z$ is an isomorphism.
\end{theorem}
\begin{proof}
Define $H_X\coloneqq[\inpt X,z]$ for $X\in\{c,d,e\}$. 
Given a coKleisli morphism $f\colon\lensob{c}\to\lensob{d}$, define $\psi_f\colon H_c\otimes \outp c\otimes\inpt d\to z$ to be the composite
\begin{equation}\label{eqn.psi_f}
H_c\otimes \outp c\otimes\inpt d\To{\inpt f}
H_c\otimes\Fun{T}(\inpt c)\To{\sigma}
\Fun{T}(H_c\otimes\inpt c)\To{\Fun T(\ev_c)}\Fun{T}(z)\To{\alpha}z
\end{equation}
and let $\Theta_\potalg(f)$ be the composite
\begin{equation}\label{eqn.theta_morphism}
H_c\otimes\outp c\To{\delta}H_c\otimes\outp c\otimes\outp c\To{\coev_d\otimes\outp f}[\inpt d,H_c\otimes\outp
  c\otimes\inpt d]\otimes\outp d\To{\psi_f}H_d\otimes\outp d.
\end{equation}

Preservation of identities is a short diagram chase from the triangle identity $\coev_a\then[a,\ev_a]=\id$ and the $\Fun T$-algebra unit law $\alpha\circ\eta_z=\id$; we leave it to the reader.
To prove that $\Theta_\potalg$ preserves composites, expand $\psi_{g\circ f}$ using the coKleisli composite and expand the composite $\Theta_\potalg(g)\circ\Theta_\potalg(f)$. After removing the common opening maps $\delta_{\outp c}$, $\outp f$, and $\coev_e$, it remains to verify that $\psi_{g\circ f}$ is the composite
\[
H_c\otimes \outp c \otimes \outp d\otimes\inpt e\To{\coev_d}[\inpt d,H_c\otimes \outp c\otimes\inpt d]\otimes \outp d\otimes \inpt e\To{\psi_f}H_d\otimes\outp d\otimes\inpt e\To{\psi_g}z.
\] 
To do so, we write $\psi_{g\circ f}$ and this composite respectively as the left-bottom and top-right paths in the following diagram; we suppress~$\otimes$ and write $c$ for~$\inpt{c}$, $d$ for~$\inpt{d}$, and $e$ for~$\inpt{e}$:
\[
\begin{tikzcd}[column sep=1.1em, row sep=1.2em, every label/.append style={font=\scriptsize}, font=\small, ampersand replacement=\&]
H_c\outp c\outp d e
  \ar[r, "{\coev}"]
  \ar[d, "{\inpt g}"']
\&[5pt] {[d,H_c\outp c d]}\outp d e
  \ar[r, "{\inpt f}"]
  \ar[d, "{\inpt g}"']
\&[-5pt] {[d,H_c\Fun T(c)]}\outp d e
  \ar[r, "{\sigma}"]
  \ar[d, "{\inpt g}"']
\& {[d,\Fun T(H_c c)]}\outp d e
  \ar[r, "{\ev}"]
  \ar[d, "{\inpt g}"']
\& {[d,\Fun Tz]}\outp d e
  \ar[r, "{\alpha}"]
  \ar[d, "{\inpt g}"']
\&[-5pt] H_d\outp d e
  \ar[d, "{\inpt g}"]
\\
H_c\outp c\Fun T(d)
  \ar[d, "{\sigma}"']
  \ar[r, "{\coev}"]
  \ar[ddr, dotted, "\sigma", end anchor=170]
\& {[d,H_c\outp c d]}\Fun T(d)
  \ar[r, "{\inpt f}"]
  \ar[d, "{\sigma}"']
\& {[d,H_c\Fun T(c)]}\Fun T(d)
  \ar[r, "{\sigma}"]
  \ar[d, "{\sigma}"']
\& {[d,\Fun T(H_c c)]}\Fun T(d)
  \ar[r, "{\ev}"]
  \ar[d, "{\sigma}"']
\& {[d,\Fun Tz]}\Fun T(d)
  \ar[r, "{\alpha}"]
  \ar[d, "{\sigma}"']
\& H_d\Fun T(d)
  \ar[d, "{\sigma}"]
\\
H_c\Fun T(\outp c d)
  \ar[dd, "{\inpt f}"']
  \ar[dr, bend right=10pt, "{\sigma}"', end anchor=180]
\& \Fun T({[d,H_c\outp c d]}d)
  \ar[r, "{\inpt f}"]
  \ar[d, "{\ev}"]
\& \Fun T({[d,H_c\Fun T(c)]}d)
  \ar[r, "{\sigma}"]
  \ar[dd, "{\ev}"']
\& \Fun T({[d,\Fun T(H_c c)]}d)
  \ar[r, "{\ev}"]
  \ar[dd, "{\ev}"']
\& \Fun T({[d,\Fun Tz]}d)
  \ar[r, "{\alpha}"]
  \ar[dd, "{\ev}"']
\& \Fun T(H_d d)
  \ar[dd, "{\ev}"]
\\
\& \Fun T(H_c\outp c d)
	\ar[u, dotted, bend left=40pt, "\coev"]
  \ar[dr, "{\inpt f}"]
\&\&\&\&
\\
H_c\Fun T\Fun T(c)
  \ar[d, "{\mu}"']
  \ar[rr, "{\sigma}"]
\&\& \Fun T(H_c\Fun T(c))
  \ar[r, "{\sigma}"]
\& \Fun T(\Fun T(H_c c))
  \ar[r, "{\ev}"]
  \ar[d, "{\mu}"']
\& \Fun T(\Fun Tz)
  \ar[r, "{\Fun T(\alpha)}"]
  \ar[d, "{\mu}"']
\& \Fun Tz
  \ar[d, "{\alpha}"]
\\
H_c\Fun T(c)
  \ar[rrr, "{\sigma}"]
\&\&\& \Fun T(H_c c)
  \ar[r, "{\ev}"]
\& \Fun Tz
  \ar[r, "{\alpha}"]
\& z
\end{tikzcd}
\]
The sixteen rectangular tiles commute by functoriality of $\otimes$, naturality of the strength $\sigma$, naturality of evaluation inside $\Fun T$, naturality of $\mu$, and the $\Fun T$-algebra axiom $\alpha\circ\mu=\alpha\circ\Fun T(\alpha)$. The lower-left pentagon is the $\mu$-coherence axiom for the strength. For the remaining region, the dotted arrows show how its commutativity follows from naturality and the triangle identity $(\coev_a\otimes a)\then\ev_a=\id$. Thus $\Theta_\potalg$ is indeed a functor.

The unitor is $e_z\colon I\to z=\Theta_\potalg\lensunit[I]$, and the productor is the canonical map
\[
([\inpt c,z]\otimes\outp c)\otimes([\inpt d,z]\otimes\outp d)
\to
[\inpt c\otimes\inpt d,z\otimes z]\otimes\outp c\otimes\outp d\To{m}
[\inpt c\otimes\inpt d,z]\otimes\outp c\otimes\outp d
\]
Their naturality, unitality, and associativity amount to the monoidality of $\Fun{T}$, the assumption that $\alpha$ is a monoid homomorphism, and that $z$ is a monoid.

Finally, $e\colon I\to z$ is an isomorphism iff $z\cong I$ as a $\otimes$-monoid, iff $\Theta_\potalg\lensunit[I]=[I,z]\otimes I\cong z\cong I$ and the unitor of $\Theta_\potalg$ realizes this isomorphism, proving the last claim.
\end{proof}

The following is an immediate corollary of \cref{thm.Theta_T_alpha} with $\Fun{T}=\id$ and $z\coloneqq I$.

\begin{corollary}[Lens internalization]\label{cor.Theta}
Let $(\cat{C},I,\otimes, [-,-])$ be a symmetric monoidal closed category. There is a normal lax symmetric monoidal functor
\[\Theta\colon\Lens{\cat{C}}\to\cat{C},\qquad
\lensob{c}\mapsto\ihom{\inpt c,I}\otimes\outp c.\]
\end{corollary}

We call $\Theta$, and more generally $\Theta_\potalg$ the \emph{lens internalization} functor, since it internalizes the interface $\lensob c$ as an object of $\cat{C}$; see \cref{ex.moore}.

\begin{example}\label{ex.moore}
The functor $\Lens{\blank\yon}\colon\Lens{\smset}\to\Lens{\poly}$ from \eqref{eqn.lens_yon} sends lenses $\binom{A}{B}$ in $\smset$ to lenses $\binom{A\yon}{B\yon}$ in $\poly$. Applying $\Theta$ one obtains:%
\footnote{
In fact, the composite functor $(\Lens{\blank\yon}\then\Theta)\colon\Lens{\smset}\to\Lens{\poly}\to\poly$ is fully faithful: a map $\binom{A}{B}\to\binom{A'}{B'}$ in $\Lens{\smset}$ is a pair of maps $B\to B'$ and $B\times A'\to A$, which is exactly the data of a polynomial map $B\yon^A\to B'\yon^{A'}$.
}
\[
\Theta\binom{A\yon}{B\yon}=[A\yon,\yon]\otimes B\yon\cong B\yon^A.
\]
A coalgebra on $B\yon^A$ is a function of the form $S\to B\times S^A$, which is equivalently a pair of functions $S\to B$, $S\times A\to S$, i.e.\ an $A$-input, $B$-output Moore machine.
\end{example}

\begin{corollary}\label{cor.parameterized_representation}
Let $\cat C$, $(\Fun T,\tau)$, $\potalg$, and the resulting $\Theta_\potalg\colon\Lens{\cat{C}}^{\Fun T}\to\cat{C}$ be as in \cref{thm.Theta_T_alpha}. Suppose $\cat A$ is symmetric monoidal and $J\colon\cat A\to\cat C$ is a strong monoidal functor such that each $J(a)$ carries a comonoid structure in $\cat C$ and each $J(u)$ is a comonoid homomorphism. Then with the $\cat A$-actions
\[
a\cdot\lensob c\coloneqq\binom{\inpt c}{J(a)\otimes\outp c}
\qqand
a\cdot c\coloneqq J(a)\otimes c
\]
on $\Lens{\cat C}^{\Fun T}$ and on $\cat C$ respectively, $\Theta_\potalg$ extends to a lax symmetric monoidal functor
\begin{equation}\label{eqn.para_theta_z}
\para{\cat A}{\Theta_\potalg}\colon\para{\cat A}{\Lens{\cat C}^{\Fun T}}\to\para{\cat A}{\cat C},
\end{equation}
normal when $z\cong I$.
\end{corollary}

\begin{proof}
We obtain the two above-displayed actions of $\cat{A}$, on $\Lens{\cat C}^{\Fun T}$ and on $\cat{C}$, by \cref{lem.parameter_lens_action,lem.smc_action}. The functor $\Theta_\potalg$ is $\cat A$-equivariant:
\[
\begin{tikzcd}[ampersand replacement=\&]
\cat A\times\Lens{\cat C}^{\Fun T}
  \ar[r,"(\cdot)"]
  \ar[d,"\id_{\cat A}\times\Theta_\potalg"']
  \ar[dr,phantom,"\cong"]
\& \Lens{\cat C}^{\Fun T}\ar[d,"\Theta_\potalg"]\\
\cat A\times\cat C
  \ar[r,"(\cdot)"']
\& \cat C,
\end{tikzcd}
\]
where the $(a,\binom{\inpt c}{\outp c})$ component of the natural isomorphism is given by
\[
J(a)\otimes\Theta_\potalg\lensob c
= J(a)\otimes(\ihom{\inpt c,z}\otimes\outp c)
\cong\Theta_\potalg\left(a\cdot\lensob c\right).
\]
We can thus apply \cref{prop.para_square,thm.Theta_T_alpha} to obtain \eqref{eqn.para_theta_z} as desired. Since $\id_{\cat A}$ is strict monoidal, the normal-lax clause of \cref{prop.para_square} transports the normality of $\Theta_\potalg$ (which holds iff $z\cong I$, by \cref{thm.Theta_T_alpha}) to $\para{\cat A}{\Theta_\potalg}$.
\end{proof}

\chapter{Abstract framework for adaptive dynamics}\label{ch.framework}

This section produces a syntax operad $\arr_D$ of \emph{adaptive arrangements} from a short list of ingredients: an \emph{adaptive-arrangement datum} $D$, including a notion of space, a notion of parameter, and a \emph{potentials monad} through which the backward pass is threaded and potentials accumulate. Two further ingredients then yield functorial semantics, a functor $\arr_D\to\pc$: a \emph{polynomial interpretation}, which casts spaces as polynomial interfaces and fixes what potentials do, and an \emph{integrator}, which turns those effects into dynamics. The three play distinct roles: the adaptive-arrangement datum is the material out of which the syntax operad $\arr_D$ is built, the polynomial interpretation interprets that syntax in $\poly$, and the integrator runs it in $\pc$.

The smooth version $\sarr\To{\Phi}\pc$ will be a special case, where: spaces are manifolds, parameters are responsive vector spaces, potentials live in $\rr$, the cotangent bundle gives interfaces, potentials act via their differential, and integration is by Euler steps. This smooth instance is assembled in \cref{ch.smooth_rwd,ch.smooth_dynamics}: its adaptive-arrangement datum $\Sm$ and syntax operad $\sarr$ (the $\sarr$ of the introduction) appear in \cref{def.potlens}, while its polynomial interpretation (\cref{sec.smooth_interpretation}, built on the cotangent bundle of \cref{sec.cot}) and integrator (Euler steps, \cref{sec.integrator_semantics}) assemble into the dynamics functor $\Phi$ in \cref{sec.dynamics_functor}.

\section{Adaptive-arrangement data}\label{sec.rewiring_datum}

\begin{definition}[Adaptive-arrangement datum]\label{def.rewiring_datum}
We define an \emph{adaptive-arrangement datum} $D=(\cat M,\cat Q,J,\Fun R)$ to consist of
\begin{enumerate}
  \item a cartesian monoidal category $\cat M$, objects are called \emph{spaces};
  \item a symmetric monoidal category $\cat Q$, objects are called \emph{parameters};
  \item a strong monoidal functor $J\colon\cat Q\to\cat M$; and
  \item a monoidal monad $\Fun R$ on $\cat M$, the \emph{potentials monad}.
\end{enumerate}
Define the $D$-\emph{adaptive-arrangement operad}, denoted $\defineTerm{rwd}_D$, to be the operad underlying the symmetric monoidal category $\para{\cat Q}{\Lcokl{\cat M}{\Fun R}}$ (\cref{lem.para_monoidal}).
\end{definition}

Unwinding the above, we find that an object of $\arr_D$ is an \emph{interface} in the sense of \cref{def.lens}: a pair $\lensob M$ of spaces; we refer to $\outp M$ as the \emph{output space} and $\inpt M$ as the \emph{input space}. We think of $\lensob M$ as a box
\boxpic{$\inpt{M}$}{$\outp{M}$}
with an input port $\inpt M$ and an output port $\outp M$.

A morphism $(\lensob{M_1},\dots,\lensob{M_K})\to\lensob N$ in $\arr_D$ is a morphism from $\lensob M\coloneqq\lensob{M_1}\ltens\cdots\ltens\lensob{M_K}$ to $\lensob N$ in $\para{\cat Q}{\Lcokl{\cat M}{\Fun R}}$, i.e.\ a tuple $(Q,\outp f,\inpt f)$ where
\begin{itemize}
\item $Q:\cat Q$ parameterizes the various wirings,
\item $\outp f\colon J(Q)\otimes\outp M\to\outp N$ is a comonoid homomorphism shuttling internal outputs to external outputs, and
\item $\inpt f\colon J(Q)\otimes\outp M\otimes\inpt N\to\Fun R(\inpt M)$ is a $\Fun R$-effectful backward pass shuttling internal outputs and external inputs to internal inputs.
\end{itemize}
The backward pass $\inpt f$ always carries the effect of the potentials monad, its extra contribution beyond purely shuttling inputs. In our main instance, the writer monad $\Fun R\coloneqq R\otimes\blank$ of a commutative monoid $(R,0,+)$ in $\cat M$ (\cref{lem.monoid_to_monad}), this effect separates out explicitly, with $\inpt f$ splitting into an ordinary input map $J(Q)\otimes\outp M\otimes\inpt N\to\inpt M$ and a scalar potential $U\colon J(Q)\otimes\outp M\otimes\inpt N\to R$.

\begin{remark}\label{rem.filtration}
To understand $\arr_D$ as a generalized wiring diagram syntax, we should understand what it is a generalization of. In fact, there is a series of generalizations as follows
\[
\Lens{\finset\op}\To{\Lens{X^\blank}}\Lens{\cat M}\To{\LensMor{\cat M}{\eta}}\Lcokl{\cat M}{\Fun R}\To{\Para_I}\arr_D
\]
where $X:\cat{M}$ is any choice of object and $X^\blank$ sends a finite set $A$ to $X^A\in\cat{M}$, $\eta\colon\id_{\cat M}\Rightarrow\Fun R$ is the unit of the potentials monad---equivalently, the unique monad map from the initial monad $\id$---and $\Para_I$ regards a potentialized lens as an adaptive arrangement with unit parameter $I$. This chain gives rise to progressively richer adaptive arrangements, as we now explain.

The simplest---those coming from the leftmost category in the chain---are static arrangements, which come from ``prisms,'' i.e.\ lenses in $\finset\op$, which we saw in \cref{ex.lens_finsetop}. In this case the maps have trivial ($Q\coloneqq I$)-factor, all the maps factor as projections and diagonals, and the potential is zero. These simplest maps give the syntax for something like ``traced symmetric monoidal semicategories,'' as discussed in \cref{sec.plenspot_ops}.

We can realize each such static arrangement as a lens in $\cat M$ via $\Lens{X^\blank}$ (\cref{lem.lens_pow}), but compositions coming from $\Lens{\cat M}$ are more general: there are many more maps in $\cat M$ than just the projections. For example, there are two projections $\rr\times\rr\to\rr$ but uncountably many smooth maps. Moving to $\Lcokl{\cat M}{\Fun R}$ threads the backward pass through the potentials monad but still includes only the trivial parameter ($Q\coloneqq I$); in the writer instance $\Fun R=R\otimes\blank$ this assigns a value in $R$ to the interaction data consisting of the internal-box outputs together with the external-box input, e.g.\ a prediction error.

Finally, the richest adaptive arrangements are the maps of $\arr_D=\para{\cat Q}{\Lcokl{\cat M}{\Fun R}}$ themselves, which add a parameterizing object $Q:\cat Q$. Here, the interaction pattern itself depends on the current $v:Q$, and under the dynamics functor that parameter is updated over time according to the potential as well as by whatever is backpropagated from the external box. \Cref{rem.huge_wiring_diagram} draws the composite of two such arrangements explicitly.
\end{remark}

\begin{remark}\label{rem.huge_wiring_diagram}
The following diagram depicts composition of 1-ary arrangements in $\arr_D$, in the writer instance $\Fun R=R\otimes\blank$ where the backward pass splits off an $R$-valued potential. We include it only for the algorithmically-minded reader's convenience: \emph{it follows from the abstract framework we already have developed.}

Suppose given adaptive arrangements $f=(Q,\outp f,\inpt f,U)\colon\lensob L\to\lensob M$, with parameter $Q$ and potential $U$, and $g=(Q',\outp g,\inpt g,V)\colon\lensob M\to\lensob N$, with parameter $Q'$ and potential $V$. One may imagine $\lensob{L}$ as the inner-most box, $\lensob{M}$ as the middle box, and $\lensob{N}$ as the outer-most box.

The composite $g\circ f$ has parameter $Q'\otimes Q$---the outer parameter first, as in \cref{sec.para_general}, which is also how the two enter the picture below---forward and backward components inherited from $\Lcokl{\cat M}{R}$, and potential added via the monoid multiplication $(+)$.
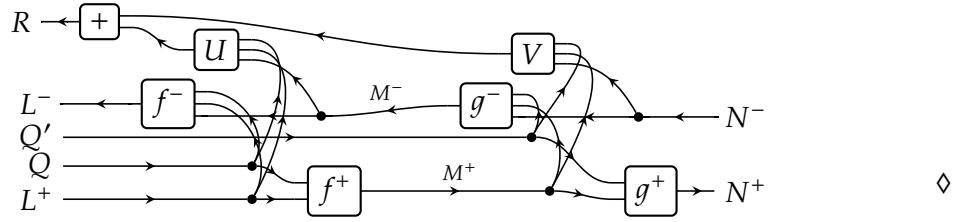
\[
\begin{tikzpicture}[oriented WD, bb port length=5pt, bb port sep=1, bb min width=.4cm, bbx=.5cm, bby=.7ex, baseline=(fo)]
  \node[bb={2}{1}] (fo) {$\outp f$};
  \node[bb={2}{1}, right=7 of fo] (go) {$\outp g$};
  \node[bb={1}{3}, above left=4 and 3 of fo] (fi) {$\inpt f$};
  \node[bb={1}{3}, above left=4 and 3 of go] (gi) {$\inpt g$};
  \node[bb={1}{3}, above right=1 and 0 of fi] (V) {$U$};
  \node[bb={1}{3}, above right=1 and 0 of gi] (W) {$V$};
  \node[bb={1}{2}, above left=-1 and 2 of V] (plus) {$+$};
  \node[dot, left=1 of fo_in2] (dotL) {};
  \node[dot] at ($(gi.east|-fo.east)+(1,0)$) (dotM) {};
  \coordinate (Lo) at ($(dotL)+(-5,0)$);
  \coordinate (No) at ($(go.east)+(1,0)$);
  \coordinate (Li) at (Lo |- fi_in1);
  \coordinate (Ni) at (No |- gi_out3);
  \node[left=0 of Lo] {$\outp L$};
  \node[right=0 of No] {$\outp N$};
  \node[left=0 of Li] {$\inpt L$};
  \node[right=0 of Ni] {$\inpt N$};
  \draw[ar] (Lo) to (dotL);
  \draw[ar] (dotL) to (fo_in2);
  \draw[ar] (fo_out1) to node[above, font=\scriptsize] {$\outp M$} (dotM);
  \draw[ar] (dotM) to (go_in2);
  \draw[ar] (go_out1) to (No);
  \draw[ar] (fi_in1) to (Li);
  \node[dot] at ($(gi_out3)+(3,0)$) (dotNI) {};
  \draw[ar] (Ni) to (dotNI);
  \draw[ar] (dotNI) to (gi_out3);
  \node[dot] at ($(fi_out3)+(3,0)$) (dotMI) {};
  \draw[ar] (gi_in1) to[out=180, in=0] node[above, font=\scriptsize] {$\inpt M$} (dotMI);
  \draw[ar] (dotMI) to (fi_out3);
  \draw[ar] (dotL) to[out=60, in=0] (fi_out2);
  \draw[ar] (dotM) to[out=60, in=0] (gi_out2);
  \node[above left=1 and 0 of Lo] (P) {$Q$};
  \node[above left=4 and 0 of Lo] (Q) {$Q'$};
  \node[dot] at (dotL|-P) (dotP) {};
  \node[dot, left=5pt] at (dotM|-Q) (dotQ) {};
  \draw[ar] (P) to (dotP);
  \draw[ar] (Q) to (dotQ);
  \draw[ar] (dotP) to (fo_in1);
  \draw[ar] (dotP) to[out=60, in=0] (fi_out1);
  \draw[ar] (dotQ) to (go_in1);
  \draw[ar] (dotQ) to[out=60, in=0] (gi_out1);
  \draw[ar] (dotP) to[out=60, in=0] (V_out1);
  \draw[ar] (dotL) to[out=60, in=0] (V_out2);
  \draw[ar] (dotMI) to[out=120, in=0] (V_out3);
  \draw[ar] (dotQ) to[out=60, in=0] (W_out1);
  \draw[ar] (dotM) to[out=60, in=0] (W_out2);
  \draw[ar] (dotNI) to[out=120, in=0] (W_out3);
  \draw[ar] (V_in1) to[out=180, in=0] (plus_out2);
  \draw[ar] (W_in1) to[out=180, in=0] (plus_out1);
  \coordinate (R) at ($(plus.west)+(-1,0)$);
  \node[left=0 of R] {$R$};
  \draw[ar] (plus_in1) to (R);
\end{tikzpicture}
\qedhere
\]
\end{remark}

We conclude this section with a free supply of monads---the power monads---that can be used in conjunction with any other monad $\Fun R$.

\begin{proposition}[Modes]\label{prop.moding}
Let $(\cat M,I,\otimes)$ be a symmetric monoidal category with finite products, and let $\md:\finset$ be a finite set. The power functor $(\blank)^{\md}$ is a monoidal monad on $\cat M$, and for every monoidal monad $\Fun R$ on $\cat M$ the composite $(\Fun R)^{\md}$ is again a monoidal monad.
\end{proposition}
\begin{proof}
The power monad $(\blank)^{\md}=\prod_{\md}$ has unit $X\to X^{\md}$ and multiplication $(X^{\md})^{\md}\to X^{\md}$ induced by $\md\to 1$ and $\md\to\md\times\md$, respectively, using only that $\cat M$ has $\md$-fold products. It is lax monoidal for $\otimes$ with unitor $I\to I^{\md}$ and productor $X^{\md}\otimes Y^{\md}\to(X\otimes Y)^{\md}$ whose $i$-th component is $\pi_i\otimes\pi_i$; the unit and multiplication are monoidal for this structure, as one checks componentwise through the projections. Note that $\otimes$ need not be the cartesian product. The canonical $\langle\Fun R(\pi_i)\rangle_i\colon\Fun R(\prod_{\md} X)\to\prod_{\md}\Fun R(X)$ is a distributive law of $\Fun R$ over $\prod_{\md}$, so $(\Fun R)^{\md}\coloneqq\prod_{\md}\circ\Fun R$ is a monoidal monad. The distributive-law and monoidality axioms are routine and left to the reader.
\end{proof}

We call the elements of $\md$ \emph{modes}, e.g.\ ``training mode and inference mode''; \cref{prop.moded_algebra} will interpret them as such.

\section{Polynomial interpretation of adaptive-arrangement data}\label{sec.poly_interpretation}

To interpret the syntax $\arr_D$ in $\poly$ we fix four further ingredients.

\begin{definition}[Polynomial interpretation datum]\label{def.polynomial_interpretation}
A \emph{polynomial interpretation datum} $\defineTerm{interp}=(\Fun c,\Fun R',\theta,\potalg)$ for an adaptive-arrangement datum $D=(\cat M,\cat Q,J,\Fun R)$ consists of
\begin{enumerate}
  \item an \emph{interface functor}: a strong symmetric monoidal functor $\Fun c\colon(\cat M,1,\times)\to(\poly,\yon,\otimes)$;
  \item a \emph{semantic effect}: a monoidal monad $\Fun R'$ on $\poly$;
  \item an \emph{effect comparison}: a monoidal monad morphism $\theta\colon\Fun c\circ\Fun R\Rightarrow\Fun R'\circ\Fun c$, which carries the syntactic effect to the semantic one;
  \item a \emph{potential algebra}: an $\Fun R'$-monoid  $\potalg=(z,e_z,m_z,\alpha)$, which applies the effects (\cref{def.T_monoid}).
  \qedhere
\end{enumerate}
\end{definition}

In the writer instance $\Fun R=R\otimes\blank$ (\cref{def.rewiring_datum}), the canonical choice for the semantic effect is $\Fun R'\coloneqq\Fun c(R)\otimes\blank$ with $\theta$ the comparison $\Fun c(R\otimes\blank)\cong\Fun c(R)\otimes\Fun c(\blank)$ induced by the strong monoidality of $\Fun c$ (\cref{lem.monoid_to_monad}); this is the choice used in \cref{ch.smooth_dynamics}.

With these in hand we build the polynomial interpretation.

\begin{theorem}[Polynomial interpretation]\label{thm.poly_interpretation}
A polynomial interpretation datum $\interp=(\Fun c,\Fun R',\theta,\potalg)$ for $D$ (\cref{def.polynomial_interpretation}) yields a lax symmetric monoidal functor
\[
  \Phip{\interp}\colon\arr_D\longrightarrow\para{\Fun{c}\circ J}{\poly},
\]
where $\cat Q$ acts on $\poly$ by $Q\cdot\blank\coloneqq \Fun{c}(JQ)\otimes\blank$. On
objects, $\Phip{\interp}\lensob{M}=\Fun{c}(\outp{M})\otimes\ihom{\Fun{c}(\inpt{M}),z}$. It is normal if and only if $z\cong\yon$.
\end{theorem}

\begin{proof}
By \cref{def.rewiring_datum}, $\arr_D$ is the underlying operad of $\para{\cat Q}{\Lcokl{\cat M}{\Fun R}}$. The interface functor $\Fun c$ (strong symmetric monoidal by definition) and the monad morphism $\theta\colon\Fun c\circ\Fun R\Rightarrow\Fun R'\circ\Fun c$ (\cref{def.polynomial_interpretation}) lift, by \cref{lem.lens_T_functoriality}, to a strong monoidal functor $\LensMor{\Fun c}{\theta}\colon\Lcokl{\cat M}{\Fun R}\to\Lcokl{\poly}{\Fun R'}$; and \cref{thm.Theta_T_alpha} (with $\Fun T\coloneqq\Fun R'$) supplies a lax monoidal $\Theta_\potalg\colon\Lcokl{\poly}{\Fun R'}\to\poly$, which is normal iff $z\cong\yon$. By \cref{prop.para_square} the $\cat Q$-action square for $\LensMor{\Fun c}{\theta}$ (from the strong monoidality of $\Fun c$) and the $\cat Q$-action square for $\Theta_\potalg$ recorded in \cref{cor.parameterized_representation} induce the two lax monoidal maps below; restricting to underlying operads gives the desired $\Phip{\interp}$:
\[
\para{\cat Q}{\Lcokl{\cat M}{\Fun R}}\To{\para{\cat Q}{\LensMor{\Fun c}{\theta}}}
\para{\cat Q}{\Lcokl{\poly}{\Fun R'}}\To{\para{\cat Q}{\Theta_\potalg}}
\para{\Fun c\circ J}{\poly}.
\qedhere
\]
\end{proof}

Recall from \cref{prop.moding} that for any finite set $\md$, we can upgrade any adaptive-arrangement datum $D\coloneqq(\cat{M},\cat Q,J,\Fun{R})$ to an $\md$-mode version $D^\md\coloneqq(\cat{M},\cat Q,J,\Fun{R}^\md)$.

\begin{proposition}[Moded interpretations]\label{prop.moded_algebra}
Let $D$ be an adaptive-arrangement datum and $\md$ a finite set. For any polynomial interpretation datum $\interp=(\Fun c,\Fun R',\theta,\potalg)$ for $D$, there is a polynomial interpretation datum $\interp^\md\coloneqq(\Fun c,(\Fun R')^\md,\theta^\md,\potalg^\md)$ for $D^\md$.
\end{proposition}
\begin{proof}
By \cref{prop.moding}, applied to the symmetric monoidal category $(\poly,\yon,\otimes)$ with its finite products, the composite $(\Fun R')^\md\coloneqq\prod_\md\circ\,\Fun R'$ is a monoidal monad on $(\poly,\otimes)$. The universal property of products gives a map $\Fun c\circ\prod_\md\To{}\prod_\md\circ\,\Fun c$, so we take the required comparison $\theta^\md$ to be the composite
\[
\Fun c\circ\Fun R^\md=\Fun c\circ\textstyle\prod_\md\circ\,\Fun R
\to\textstyle\prod_\md\circ\,\Fun c\circ\Fun R
\To{\prod_\md\theta}\textstyle\prod_\md\circ\,\Fun R'\circ\Fun c=(\Fun R')^\md\circ\Fun c.
\]

For the potential algebra, take $z^{\md}\coloneqq\prod_{\md} z$ to be the $\md$-fold product of the monoid $z$, which is naturally a monoid. Its $(\Fun R')^{\md}$-algebra structure $\alpha^\md$ is given, on the factor at each $i\in\md$, by the composite
\[(\Fun R')^\md(z^{\md})=\textstyle\prod_\md\Fun R'(\textstyle\prod_\md z)\To{\pi_i}\Fun R'(\textstyle\prod_\md z)\To{\Fun R'(\pi_i)}\Fun R'(z)\To{\alpha}z,\]
using the universal property of products. 
It is easy to check that $\potalg^{\md}$ is a $(\Fun R'\blank)^{\md}$-monoid.
\end{proof}

\section{Integrators and the dynamics functor}
  \label{sec.dynamics_functor_abstract}

This section assembles the dynamics functor $\Phi\colon\arr_D\to\pc$ in three steps: we repackage $\pciso$ as a Para construction (\cref{prop.pc_as_para,lem.poly_to_pc}), define \emph{integrators} and the semantics they induce (\cref{def.integrator,prop.integrator_to_pc}), and combine these with the polynomial interpretation to obtain $\Phi_{\interp,\intg}$ (\cref{thm.dynamics_functor}).

Following the convention of \cref{sec.prelim}, we refer to $(2,1)$-categories simply as categories, though we do take care of the 2-dimensional isomorphisms in proofs. Recall from \cref{lem.store_monoidal} that $\Store\colon\smsetiso\to\poly$ sends $S:\smset$ to the store comonad $S\yon^S$.

\begin{definition}[Store action]\label{def.store_action}
The \emph{store action} of $\smsetiso$ on $\poly$ is the action induced by $\Store$ via \cref{lem.smc_action}:
\[
(\cdot)\colon\smsetiso\times\poly\to\poly,
\qquad
S\cdot p\coloneqq \Store(S)\otimes p.
\qedhere\]
\end{definition}

\begin{proposition}\label{prop.pc_as_para}
There is an isomorphism of categories
\[
\para{\smsetiso}{\poly}\To{\cong}\pciso.
\]
\end{proposition}

\begin{proof}
Both $\para{\smsetiso}{\poly}$ and $\pciso$ are $(2,1)$-categories with object set $\ob\poly$, and the mapping groupoids $p\to p'$ agree
\begin{align*}
	\para{\smsetiso}{\poly}(p,p')&=
  \sum_{S:\smsetiso}\poly(\Store(S)\otimes p, p')\\&\cong
  \sum_{S:\smsetiso}\poly(S\yon^S,[p,p'])\\&\cong
  \sum_{S:\smsetiso}\smset(S,[p,p'](S))=\pciso(p,p')
\end{align*}
by \cref{def.store_action,eqn.dirichlet_adjunction,prop.coalg_as_poly_map}. It remains to show that the identities and composites agree.

In both cases the identity is indexed by the monoidal unit $1:\smsetiso$, and $1\To{\id}\poly(p,p)=[p,p](1)$ and $\yon\To{\id}[p,p]$ are identified as above. Given composable maps $p\To{(S,\varphi)}p'\To{(T,\psi)}p''$, in both cases the composite is of the form $(TS,\psi\circ\varphi)$, and the same identifications match these composites.
\end{proof}

\begin{lemma}\label{lem.poly_to_pc}
Let $\cat A$ be a symmetric monoidal category and let $F\colon\cat A\to\smsetiso$ be strong symmetric monoidal. Let $\cat A$ act on $\poly$ via the strong monoidal composite $\Store\circ F\colon\cat A\to\poly$ as in \cref{lem.smc_action}. There is an identity-on-objects normal lax symmetric monoidal functor
\[
D_F\colon\para{\Store\circ F}{\poly}\to\pciso.
\]
\end{lemma}

\begin{proof}
The induced action sends $(Q,p)\mapsto \Store(F(Q))\otimes p=F(Q)\yon^{F(Q)}\otimes p$. The following action square, whose lower edge is the store action (\cref{def.store_action}), obviously commutes:
\[
\begin{tikzcd}[ampersand replacement=\&, column sep=60pt]
\cat A\times\poly\ar[r,"\Store(F(\blank))\otimes\blank"]\ar[d,"F\times\id_\poly"', ""{name=L}]\& \poly\ar[d, equal, ""' {name=R}]\\
\smsetiso\times\poly\ar[r,"\Store(\blank)\otimes\blank"']\& \poly.
\end{tikzcd}
\]
\Cref{prop.para_square} gives $\para{\Store\circ F}{\poly}\to\para{\smsetiso}{\poly}$, and composing with the isomorphism $\para{\smsetiso}{\poly}\cong\pciso$ from \cref{prop.pc_as_para} yields the claim. For normality, note that $F$ being strong monoidal gives $F(I_{\cat A})\cong 1$, so $D_F$ sends the unit $(I_{\cat A},\yon)$ to $\Store(F(I_{\cat A}))\otimes\yon\cong\yon$.
\end{proof}

The idea of \cref{lem.poly_to_pc} is very simple: that a polynomial map $F(Q)\yon^{F(Q)}\otimes p\to p'$ can be converted into a $[p,p']$-coalgebra with states $F(Q)$, and that this process is functorial.

\begin{definition}[Integrator]\label{def.integrator}
Let $\cat Q$ be a symmetric monoidal category and $p\colon\cat Q\to\poly$ a strong symmetric monoidal functor, the \emph{parameter interface}. A \emph{$p$-integrator} is a pair $\defineTerm{intg}=(\Fun S,\dyn)$
\[
\begin{tikzcd}[column sep=60pt]
	\cat Q\ar[d, equal]\ar[r, "\Fun S", ""' name=top]&
	\smsetiso\ar[d, "\Store"]\\
	\cat Q\ar[r, "p"', ""{name=bot}]&
	\poly
	\arrow[from=top, to=bot, Rightarrow, shorten=3pt, "\dyn"]
\end{tikzcd}
\]
where $\Fun S\colon\cat Q\to\smsetiso$, the \emph{state space} functor, is strong symmetric monoidal and $\defineTerm{dyn}\colon\Store\circ\Fun S\Rightarrow p$, the \emph{dynamics map}, is a monoidal natural transformation.
\end{definition}

By \cref{prop.coalg_as_poly_map}, any component $\dyn_a\colon\Store(\Fun S_a)\to p_a$ is equivalently a $p_a$-coalgebra structure on the state set $\Fun S_a$, and for any $a,a'$, the coalgebra associated to $a\otimes a'$ is the Dirichlet product of those associated to $a$ and $a'$, i.e.\ the two coalgebras running in parallel.

We call $\intg$ an \emph{integrator} for what its instances turn out to be. Indeed the integrators built below---configuration and phase---are numerical integrators of explicit-Euler type, namely Euler on configuration space and a phase-space analogue. They differ in both the state space $\Fun S$ (configuration versus phase space) and the dynamics $\dyn$.

\begin{proposition}\label{prop.integrator_to_pc}
Let $\cat Q$ be a symmetric monoidal category and $p\colon\cat Q\to\poly$ a strong symmetric monoidal functor. Every $p$-integrator $\intg$ induces a normal lax symmetric monoidal functor
\[\Psisem{\intg}\colon\para{p}{\poly}\to\pciso.\]
\end{proposition}

\begin{proof}
Let $\intg=(\Fun S,\dyn)$. Since the state space functor $\Fun S$ is strong symmetric monoidal, \cref{lem.poly_to_pc} provides a normal lax monoidal functor
\[
D_{\Fun S}\colon\para{\Store\circ\Fun S}{\poly}\to\pciso.
\]
By \cref{prop.para_strong_induced,prop.para_square}, the dynamics $\dyn\colon\Store\circ\Fun S\Rightarrow p$ induces a strong symmetric monoidal functor
\[
\Para_{\cat Q}^{\dyn}(\poly)\colon\para{p}{\poly}\to\para{\Store\circ\Fun S}{\poly}.
\]
Define $\defineTerm{Psisem}_{\intg}\coloneqq D_{\Fun S}\circ\Para_{\cat Q}^{\dyn}(\poly)$; it is normal lax monoidal because $\Para_{\cat Q}^{\dyn}(\poly)$ is strong.
\end{proof}

Let $\arr_D$ be the adaptive-arrangement operad of an adaptive-arrangement datum $D=(\cat{M},\cat Q,J,\Fun R)$.

\begin{theorem}[Dynamics functor]\label{thm.dynamics_functor}
Every polynomial interpretation datum $\interp=(\Fun c,\Fun R',\theta,\potalg)$ and $(\Fun{c}\circ J)$-integrator $\intg$ induce an operad functor
\[
\Phi_{\interp,\intg}\colon\arr_D\to\pc,
\]
which is normal (as a lax symmetric monoidal functor) when $z\cong\yon$.
\end{theorem}

\begin{proof}
Define $\Phi_{\interp,\intg}$ to be the composite of the lax monoidal functors below, given by \cref{thm.poly_interpretation,prop.integrator_to_pc}:
\[
\arr_D\To{\Phip{\interp}}
\para{\Fun{c}\circ J}{\poly}\To{\Psisem{\intg}}
\pciso\inj\pc.
\]
When $z\cong\yon$ all three factors are normal---$\Phip{\interp}$ by \cref{thm.poly_interpretation}, $\Psisem{\intg}$ by \cref{prop.integrator_to_pc}, and the inclusion strongly---so the composite is normal.
\end{proof}

Unwinding, we find that the functor $\Phi_{\interp,\intg}$ sends an object $\lensob M$ to the polynomial interface
\[
\Phi_{\interp,\intg}\lensob M=\Fun c(\outp M)\otimes\ihom{\Fun c(\inpt M),z}
\]
and a morphism $f=(Q,\outp f,\inpt f)\colon\lensob M\to\lensob N$ to the $\pc$-morphism $\Phi_{\interp,\intg}\lensob M\to\Phi_{\interp,\intg}\lensob N$ carried by the state set $\Fun S_Q$: the coalgebra (\cref{prop.coalg_as_poly_map}) whose structure map is the composite
\[
\Store(\Fun{S}_Q)=\Fun{S}^{ }_Q\yon^{\Fun S_Q}
\To{\dyn_Q}
\Fun c(JQ)
\To{\widetilde{\Phip{\interp}(f)}}
\ihom*{\Phi_{\interp,\intg}\lensob M,\Phi_{\interp,\intg}\lensob N}.
\]
Here $\dyn_Q$ is the integrator's dynamics (\cref{def.integrator}) and $\widetilde{\Phip{\interp}(f)}$ is the transpose under \eqref{eqn.dirichlet_adjunction} of the $\Fun c(JQ)$-parameterized polynomial map $\Phip{\interp}(f)$ of \cref{thm.poly_interpretation}. The integrator enters only through $\dyn_Q$, which trades the parameter $\Fun c(JQ)$ for the stored state $\Fun S_Q$, while the backward pass $\inpt f$ and the output map $\outp f$ are recorded inside $\Phip{\interp}(f)$.

When $z\not\cong\yon$, the unit comparison cell for the empty interface $I=\lensunit\in\arr_D$ corresponds to a non-invertible $\pc$-morphism $\yon\to z$. It will be a single-state coalgebra on $\ihom{\yon,z}\cong z$, which picks out the unit position $e_z\in z(1)$. 

\begin{remark}[Multi-stage integrators]\label{rmk.multistage}
The integrators we build are \emph{single-stage}: at each tick the system emits a position, receives a direction, and updates once. Many numerical integrators---Runge--Kutta methods, or velocity Verlet (leapfrog)~\cite[\S 2.2]{gauckler2017dynamics}---instead take several emit--receive rounds per tick, producing a coalgebra of the form $S\to p^{\circ\ell}(S)$, where $p^{\circ\ell}=p\circ\cdots\circ p$ is the $\ell$-fold composite. A correspondingly modified integrator (\cref{def.integrator}), with state space $\Fun S$ unchanged but dynamics of the form $\Store\circ\Fun S\Rightarrow(\Fun c\circ J)^{\circ\ell}$, should again yield an operad functor into the analogous $2$-category of $\ell$-round coalgebras. An LLM (Claude) has implemented the $\ell=2$ case (velocity Verlet) in companion code (\url{https://github.com/dspivak/dap}), where it is used to integrate the wave equation; we have not independently audited it, and leave the general treatment to future work. See also \cite{ngotiaoco2017compositionality}.\qedhere
\end{remark}

\begin{example}[A syntactic datum]\label{ex.syntactic_datum}
The spaces $\cat M$ of \cref{def.rewiring_datum} need not be manifolds; they can arise from a genuinely algebraic syntax. For example, one can take the Cartesian category $\Sigma$ presented by a signature of primitive operations---e.g.\ addition, multiplication, and $\tanh\colon\rr\to\rr$, subject to the commutative-$\rr$-algebra laws---but more generally any desired syntax that includes the theory of real vector spaces via some strong monoidal $J\colon\rvect\to\Sigma$ (the responsive vector spaces of \cref{sec.rvect}).

Then $D_\Sigma\coloneqq(\Sigma,\rvect,J,\rr\times\blank)$ is an adaptive-arrangement datum, and if each primitive has a smooth analogue, one has a strong monoidal $\iota\colon\Sigma\to\mfd$ fixing $(\rr,0,+)$. By the functoriality used in the proof of \cref{thm.poly_interpretation}---namely of $\Lcokl{-}{\rr}$ in the base (\cref{lem.lens_T_functoriality}) and of $\para{\rvect}{-}$ (\cref{prop.para_square})---this $\iota$ induces an operad functor $\arr_{D_\Sigma}\to\sarr$ into the smooth adaptive-arrangement operad of \cref{ch.smooth_rwd}, and composing with the smooth dynamics functor $\Phi$ of \cref{ch.smooth_dynamics}, we obtain semantics $\arr_{D_\Sigma}\to\pc$ for the algebraic syntax.

A feedforward network of affine layers and $\tanh$ is then a single arrangement of $\arr_{D_\Sigma}$, i.e.\ a finite syntactic term, not an arbitrary smooth map.
Under the configuration semantics, if $\sharpR_q=\etaLR\sharpEuc{}$ and the incoming covector is the differential of a loss, its image has exactly the gradient-descent update (\cref{sec.dl_warmup}).
\end{example}

%

\chapter{Smooth adaptive arrangements}\label{ch.smooth_rwd}

This section assembles the smooth adaptive-arrangement datum $\Sm$ and its syntax operad $\sarr$. We recall smooth-manifold notation in \cref{sec.manifolds_notation}, introduce the category $\rvect$ of responsive vector spaces in \cref{sec.rvect}, and assemble $\Sm$, whose adaptive-arrangement operad $\sarr$ consists of $\rvect$-parameterized manifold lenses with real-valued potentials, in \cref{sec.potentialized_lenses}. This $\sarr$ is the smooth adaptive-arrangement syntax, the particular $\cat{W}$ promised in \cref{sec.wd_operads}; its interpretation in $\poly$ and the resulting dynamics are the subject of \cref{ch.smooth_dynamics}.

\section{Smooth-manifold notation}\label{sec.manifolds_notation}

In this section we fix the differential-geometric notation used throughout the rest of this paper.

We work in the cartesian category $(\defineTerm{mfd},1,\times)$ of finite-dimensional smooth real manifolds and smooth maps; the monoidal unit $1\coloneqq\rr^0$ is the one-point manifold. For a manifold $M$, we write $TM:\mfd$ and $T^*M:\mfd$ for the tangent and cotangent bundles. Given a point $m\colon 1\to M$, we write $T_mM:\vect$ and $T^*_mM:\vect$ for the tangent space and cotangent space at $m$.

We often elide the inclusion $\inc\colon\defineTerm{vect}\to\mfd$, treating finite-dimensional real vector spaces as manifolds. The direct sum $(0,\oplus)$ structure on $\vect$ is (a coproduct and) a product, and both $\inc$ and $\absval{\blank}$ are product preserving, e.g.\ $\absval{V\oplus W}\cong\absval V\times\absval W$.

The tangent functor $T\colon\mfd\to\mfd$ preserves products, and the cotangent-bundle assignment (which is not functorial) does so on objects: $T(M\times N)\cong TM\times TN$ and $T^*(M\times N)\cong T^*M\times T^*N$; in particular
\begin{equation}\label{eqn.T_products}
T_{(m,n)}(M\times N)\cong T_mM\oplus T_nN
\qqand
T^*_{(m,n)}(M\times N)\cong T^*_mM\oplus T^*_nN.
\end{equation}

For a smooth map $f\colon M\to N$, write
\[
Tf\colon TM\to TN
\]
for the \emph{tangent map} (also called the \emph{global differential}), and
\[
T_m f\colon T_mM\to T_{f(m)}N
\qqand
(T_m f)^\top\colon T^*_{f(m)}N\to T^*_mM
\]
for the \emph{differential at $m$}---the restriction of $Tf$ to the fiber over $m\in M$---and its transpose, each of which is a map in $\vect$.%
\footnote{In the literature, what we denote $T_m f$ is often denoted $Df_m$.}

By standard manifold theory, the canonical global chart on a finite-dimensional real vector space $Q$ provides linear isomorphisms
\begin{equation}\label{eqn.tangent_vec}
T_qQ \;\cong\; Q \qqand T^*_qQ \;\cong\; Q^*
\end{equation}
at every point $q\in Q$; we use these identifications throughout, often without comment. Assembling these fibers over all $q$, the same chart trivializes the tangent and cotangent bundles of the manifold $Q$,
\begin{equation}\label{eqn.TT_def}
TQ\cong Q\oplus Q\qqand T^*Q\cong Q\oplus Q^*,
\end{equation}
so the tangent and cotangent bundles of a vector space are again vector spaces. We also will make frequent use of the canonical identification
\begin{equation}\label{eqn.canonical_dual_sum}
(Q\oplus Q^*)^*\cong Q^*\oplus Q,\qquad
\lambda\leftrightarrow(\xi,q)\quad\text{where}\quad
\lambda(q',\xi')=\xi(q')+\xi'(q).
\end{equation}

When the codomain is $\rr$, the tangent map of any $U\colon M\to\rr$ can be combined with the standard constant covector field $r\mapsto +1$ on $\rr$ to give a covector field on $M$, also called the \emph{differential} of $U$:%
\footnote{Covector fields are also called 1-forms.}
\begin{equation}\label{eqn.differential}
dU\colon M\to T^*M,\qquad dU|_m\coloneqq(T_m U)^\top(+1)\in T^*_mM.
\end{equation}
Equivalently, $dU|_m(\dot m)=(T_mU)(\dot m)$ for $\dot m\in T_mM$. When $M=Q$ is a vector space, we regard $dU|_q$ as an element of $Q^*$ under \eqref{eqn.tangent_vec}.

For a smooth map $f\colon M\times N\to P$ and a point $(m,n)\in M\times N$, we write
\begin{equation}\label{eqn.partial_differential}
\partial_m f\coloneqq T_m\bigl(f(\blank,n)\bigr)\colon T_mM\to T_{f(m,n)}P
\end{equation}
for the differential at $m$ of the partial application $f(\blank,n)\colon M\to P$, and similarly in any factor of a finite product: the subscript names the point at which we differentiate, and all other arguments are held fixed. When $P=\rr$, we regard $\partial_m f$ as a covector in $T^*_mM$ under \eqref{eqn.tangent_vec}, so that $\partial_m f=d\bigl(f(\blank,n)\bigr)\big|_m$ in the sense of \eqref{eqn.differential}.

\begin{warning}\label{warn.prime}
We will not ever use $f'$ notation to mean derivatives in this paper; ``priming'' a symbol (such as $f$) will only be used to indicate that $f'$ is somehow analogous to $f$.
\end{warning}

\begin{remark}[$C^1$ regularity]\label{rmk.C1_regularity}
We use smooth manifolds and smooth maps throughout the paper, but the latter is purely a convenience rather than a necessity. With one exception, we have checked that all definitions, theorem and proposition statements, and proofs in this paper remain valid when the manifolds are smooth but the maps of manifolds (the output and input maps $\outp f,\inpt f$ and the potential $U$ of \eqref{eqn.srw_morphism}, and the sharp map $\sharpR_Q$ of \cref{def.rvect}) \emph{are allowed to be $C^1$} rather than smooth.

The reason it works is that the semantics is first-order and set-theoretic: the cotangent functor $\cot$ (\cref{def.cot}) uses only the pointwise differentials $T_mf$ and their transposes; its functoriality is the chain rule, which is valid for $C^1$ maps. The differential $dU$ of a $C^1$ potential is merely continuous, but even that is unused. It enters only as the backward component of the polynomial map $\potd\circ\cotof{U}\colon\cotof{X}\to\yon$ \eqref{eqn.d_potential}, and maps in $\poly$ carry no regularity requirement: the coalgebra updates in $\pc$ are functions of sets.

The exception is the fullness assertion in the first item below; three further ancillary claims also require qualification:
\begin{enumerate}
\item The operad functor $\cat R\to\sarr$ of \cref{sec.related_cls} remains faithful but is no longer full; the arrow label $\mathrm{ff}$ becomes $\mathrm{faithful}$ in both copies of the comparison diagram \eqref{eqn.cls_diag}, and likewise in the proof of \cref{prop.euler_submersion_lenses}, whose obligations are only weakened.
CLS is a smooth framework, so the morphisms of $\cat R$ keep smooth $w$, $E$, and sharp, while $\sarr$ includes morphisms that aren't smooth.
\item The tangent functor $T\colon\mfd\to\mfd$ above ceases to be an endofunctor: if $f$ is $C^1$, the global tangent map $Tf$ may not be. Its pointwise content---the identifications \eqref{eqn.T_products}---and the pointwise differentials $T_mf$ are unaffected, the former because the manifolds remain smooth and the latter because $f$ is $C^1$; the constructions of this paper use only these.
\item Two statements about \emph{flow}---the final clause of \cref{rmk.symplectic_perpendicular} and the sentence ``In particular the continuous flow conserves $H$, by \cref{rmk.symplectic_perpendicular}'' in \cref{sec.closed_conservation}---require the Hamiltonian vector field to generate a flow. However, for a $C^1$ Hamiltonian that field is merely continuous, so the flow may not exist or be unique.
\item Newton's method (\cref{sec.newton_warmup}) defines its response as an inverse Hessian, but this is undefined for $C^1$ potentials, which is an issue for the construction of the Newton response structure.
\end{enumerate}
An example potential we have seen in practice is $U(q)\coloneqq\absval{q}^{3/2}$, which is $C^1$ but not $C^2$. Yet \cref{cor.functor} still assigns its arrangement a total, deterministic dynamical system, because the discrete semantics uses $dU$ pointwise and never solves an ODE.
\end{remark}

\section{Responsive vector spaces}\label{sec.rvect}

We will eventually define $\sarr$ as the operad underlying $\para{\rvect}{\plmfd}$. Here we define the parameterizing category $\rvect$, whose objects encode the data needed to turn covectors into vectors that serve as state updates. We also give a symplectic example $T^*Q$.

\begin{definition}\label{def.rvect}
A \emph{responsive vector space} is a finite-dimensional real vector space $Q$ together with a smooth map
\[
\sharpR_Q\colon Q\to\vect(Q^*,Q),
\]
the \hypertarget{term.sharp_map}{\emph{sharp}} map, where we write $\sharpR_q\coloneqq\sharpR_Q(q)\colon Q^*\to Q$ for each $q\in Q$.\footnote{The terminology mostly follows~\cite{capucci2024organizing}, where a smooth section of $[T^*X,TX]$ over a manifold $X$ is called a \emph{reaction}; the above is the vector-space case. Equivalently, $\sharpR_Q$ is a smooth bundle map $T^*Q\to TQ$ over $Q$, under the canonical identifications $T_qQ\cong Q$ and $T^*_qQ\cong Q^*$. It will be more convenient to use the above form.
}
We say the response is \emph{constant} if $\sharpR_q=\sharpR_0$ for all $q$. We say the response is \emph{symmetric} if $\xi(\sharpR_q(\xi'))=\xi'(\sharpR_q(\xi))$ for all $q\in Q$ and all $\xi,\xi'\in Q^*$, and we call the responsive vector space \emph{quadratic} if its response is constant and symmetric.%
\footnote{When the response is constant, symmetry says exactly that $B(\xi,\xi')\coloneqq\xi\bigl(\sharpR_0(\xi')\bigr)$ is a symmetric bilinear form on $Q^*$, i.e.\ a quadratic form on $Q^*$ (with positive definiteness not assumed).}
We call a responsive vector space \emph{nondegenerate} if $\sharpR_q$ is an isomorphism for every $q$; in that case we denote the pointwise inverse $\flat_q\coloneqq\sharpR_q\inv\colon Q\to Q^*$ and call it the \emph{flat} map, and we call $\pairing{v,v'}_q\coloneqq\flat_q(v)(v')$ the \emph{pairing at $q$}.

We define the category $\defineTerm{rvect}$ to have responsive vector spaces as objects---we write $\rv Q\coloneqq(Q,\sharpR_Q)$ for the object and $Q$ for its underlying vector space---and as morphisms $\rv Q_1\to\rv Q_2$ the linear isomorphisms $\phi\colon Q_1\To{\cong}Q_2$ that are \emph{sharp-equivariant}: 
for all $q\in Q_1$,
\begin{equation}\label{eqn.pairing_triangle}
\phi\circ\sharpR_q\circ\phi^*=\sharpR_{\phi(q)}
\end{equation}
as maps $Q_2^*\to Q_2$.%
\footnote{
We require every $\rvect$-morphism to be an isomorphism so that we can define the cotangent endofunctor $T^*\colon\rvect\to\rvect$ (\cref{prop.TT_endofunctors}): its action $T^*\phi=\phi\oplus(\phi^*)\inv$ \eqref{eqn.TT_phi} needs $\phi$ invertible.
}
In the nondegenerate case, \eqref{eqn.pairing_triangle} is equivalent to pointwise pairing-preservation: $\pairing{\phi v,\phi v'}_{\phi(q)}=\pairing{v,v'}_q$.

Where the sharp is understood, or is the canonical one (as on $T^*Q$, \cref{prop.canonical_symplectic_pairing}), we may abuse notation and write simply $Q$ for $\rv Q$.
\end{definition}

Our theory does not privilege nondegenerate responses in any way: the only nondegeneracy our results rely on is that of the canonical symplectic pairing on $T^*Q$ (\cref{prop.canonical_symplectic_pairing}), which holds regardless of the response $\sharpR_Q$. Degeneracy is what allows things like controllers to act on their environment without being retroactively back-driven by it. We refer to this condition as unilaterality; see \cref{def.arrangement_terminology}.
I believe this sort of back-driving is an instance of what Del Vecchio--Ninfa--Sontag call \emph{retroactivity}~\cite{delvecchio2008modular}, which they identify as an obstruction to modularity in biomolecular networks.\footnote{I thank Eduardo Sontag for a helpful discussion.}

\begin{example}\label{ex.tank_sharp}
Imagine a weirdly-shaped water tank for which the cross-sectional area is a function of height.
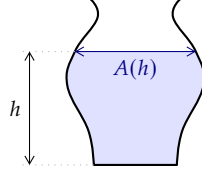
\[
\begin{tikzpicture}[font=\scriptsize,
	dim/.style={<->, >={Straight Barb[length=3pt,width=3.5pt]}, thin}]
	\begin{scope}
		\clip (-1.2,-.05) rectangle (1.2,1.5);
		\fill[blue!12]
			plot[smooth, tension=.7] coordinates {(-.55,0) (-.62,.45) (-.9,1.05) (-.8,1.5) (-.5,1.95) (-.6,2.2)}
			-- plot[smooth, tension=.7] coordinates {(.6,2.2) (.5,1.95) (.8,1.5) (.9,1.05) (.62,.45) (.55,0)}
			-- cycle;
	\end{scope}
	\draw[thick, line join=round]
		plot[smooth, tension=.7] coordinates {(-.6,2.2) (-.5,1.95) (-.8,1.5) (-.9,1.05) (-.62,.45) (-.55,0)}
		-- plot[smooth, tension=.7] coordinates {(.55,0) (.62,.45) (.9,1.05) (.8,1.5) (.5,1.95) (.6,2.2)};
	\draw[dim, blue!50!black] (-.8,1.5) -- node[below=-1pt] {$A(h)$} (.8,1.5);
	\draw[gray!60, dotted] (-.8,1.5) -- (-1.5,1.5);
	\draw[gray!60, dotted] (-.55,0) -- (-1.5,0);
	\draw[dim] (-1.4,0) -- node[left=-1pt] {$h$} (-1.4,1.5);
\end{tikzpicture}
\]
Whether we take its state to be the height or the volume of water, its response can depend on the shape of the tank at its current state, and hence be modeled by a non-constant sharp. We make a concrete choice in \cref{ex.tank_system}.
\end{example}

\begin{proposition}\label{prop.rvect_monoidal}
Direct sum gives $\rvect$ a symmetric monoidal structure $(0,\oplus)$, with the zero space as unit and direct-sum sharp $\sharpR_{Q\oplus W}$ given by
\begin{equation}\label{eqn.directsum_sharp}
\sharpR_{(v,w)}\coloneqq\sharpR_v\oplus\sharpR_w\colon Q^*\oplus W^*\to Q\oplus W.
\end{equation}
\end{proposition}

\begin{proof}
A pointwise direct sum of smooth maps is again smooth. Sharp-equivariance \eqref{eqn.pairing_triangle} is preserved by direct sums, and coherence is inherited from $(\vect,0,\oplus)$.\qedhere
\end{proof}

\begin{proposition}\label{prop.rvect_smc_action}
The composite $\rvect\to\vect\inj\mfd$ is strong symmetric monoidal, so $\rvect$ acts on $\mfd$ via $\rv Q\cdot M\coloneqq Q\times M$.
\end{proposition}

\begin{proof}
The forgetful $\rvect\to\vect$ preserves direct sums (which are products in $\vect$); the inclusion $\vect\inj\mfd$ preserves products.\qedhere
\end{proof}

\begin{proposition}\label{prop.canonical_symplectic_pairing}
For any finite-dimensional vector space $Q:\vect$, the cotangent bundle $T^*Q\cong Q\oplus Q^*$ \eqref{eqn.TT_def} is a responsive vector space under the \emph{canonical symplectic pairing}: at every basepoint $(q_0,\xi_0)\in T^*Q$,
\begin{equation}\label{eqn.canonical_form}
\pairing{(q,\xi),(q',\xi')}_{(q_0,\xi_0)} \;=\; \xi'(q)-\xi(q').
\end{equation}
The right-hand side does not depend on $(q_0,\xi_0)$, so the response is constant.
\end{proposition}

\begin{proof}
The form in \eqref{eqn.canonical_form} is bilinear. To see that it is nondegenerate, suppose $(q,\xi)$ pairs to $0$ with every $(q',\xi')$. Taking $q'=0$ gives $\xi'(q)=0$ for all $\xi'\in Q^*$, hence $q=0$; taking $\xi'=0$ gives $\xi(q')=0$ for all $q'\in Q$, hence $\xi=0$.\qedhere
\end{proof}

For the canonical symplectic pairing, under \eqref{eqn.canonical_dual_sum}, the associated sharp map is the constant map
\begin{equation}\label{eqn.canonical_sharp}
\defineTerm{sharp_symp}_{Q\oplus Q^*}\colon Q^*\oplus Q\to Q\oplus Q^*,\qquad(\xi,q)\mapsto(q,-\xi).
\end{equation}

\begin{remark}[Symplectic sharp produces perpendicular vectors]\label{rmk.symplectic_perpendicular}
For any covector $(\xi,q)\in Q^*\oplus Q$, the resulting vector $\sharpS_{Q\oplus Q^*}(\xi,q)=(q,-\xi)\in Q\oplus Q^*$ \eqref{eqn.canonical_sharp} is annihilated by it: pairing the two via \eqref{eqn.canonical_dual_sum},
\[
(\xi,q)\bigl(\sharpS_{Q\oplus Q^*}(\xi,q)\bigr) = \xi(q) + (-\xi)(q) = 0.
\]
Geometrically, the symplectic sharp ``rotates'' each covector to a vector ``perpendicular'' to it. In particular, when the covector is the differential $dH|_w$ of a smooth function $H\colon Q\oplus Q^*\to\rr$, the vector $\sharpS_{Q\oplus Q^*}(dH|_w)$ is tangent to the level set $\{w'\mid H(w')=H(w)\}$, so the flow it generates conserves $H$.\qedhere
\end{remark}

\begin{example}[Euclidean sharp]\label{ex.euclidean_sharp}
A finite-dimensional real vector space $Q$ equipped with an inner product---a symmetric, positive-definite bilinear form $\pairing{-,-}\colon Q\otimes Q\to\rr$---yields a nondegenerate responsive vector space with constant sharp
\begin{equation*}
\defineTerm{sharpEuc}\colon Q^*\To{\cong}Q,
\end{equation*}
the inverse of $\flat_\euc(v)\coloneqq\pairing{v,-}\colon Q\to Q^*$.\qedhere
\end{example}

\section{The smooth adaptive-arrangement datum \texorpdfstring{$\Sm$}{Sm}}\label{sec.potentialized_lenses}

We build $\Sm$ in two steps: first specialize the backward-monad construction of \cref{sec.backwards} to the writer monad $(\rr\times\blank)$ on $\mfd$, producing $\plmfd$, then parameterize $\plmfd$ over $\rvect$ to form $\sarr$.

Since $(\rr,0,+)$ is a commutative monoid in $\mfd$, \cref{lem.monoid_to_monad} gives the writer monad $(\rr\times\blank)$ on $\mfd$, and \cref{prop.backward_comonad} gives the corresponding colax monoidal comonad on $\Lens{\mfd}$. We let $(\plmfd,\lensunit,\ltens)$ denote its coKleisli category. Let's unpack that.

A map in $\plmfd$ is a tuple $f=(\outp f,\inpt f,U)\colon\lensob M\to\lensob N$ with
\[
\outp f\colon\outp M\to\outp N,\qquad
\inpt f\colon\outp M\times\inpt N\to\inpt M,\qquad
U\colon\outp M\times\inpt N\to\rr.
\]

There is an action of $\rvect$ on $\plmfd$ from \cref{lem.parameter_lens_action} (taking $\cat A\coloneqq\rvect$, $\cat C\coloneqq\mfd$, $J\coloneqq\inc$, $\Fun T\coloneqq\rr\times\blank$), and the resulting Para construction is symmetric monoidal by \cref{lem.para_monoidal}.

\begin{definition}[Smooth adaptive-arrangement datum; $\sarr$]\label{def.potlens}
The \emph{smooth adaptive-arrangement datum} is
\[
\Sm\coloneqq(\mfd,\rvect,\inc,\rr),
\]
an adaptive-arrangement datum (\cref{def.rewiring_datum}): $\mfd$ is cartesian monoidal, $(\rr,0,+)$ is a commutative monoid object in it---giving the writer potentials monad $\rr\otimes\blank$, which we abbreviate by its carrier $\rr$---$\rvect$ is symmetric monoidal (\cref{prop.rvect_monoidal}), and $\inc\colon\rvect\to\mfd$ is strong symmetric monoidal (\cref{prop.rvect_smc_action}). 

Its adaptive-arrangement operad $\defineTerm{srwd}$ is the underlying operad of the symmetric monoidal category $\para{\rvect}{\plmfd}$ (\cref{lem.para_monoidal}).
\end{definition}

An object in $\sarr$ is a pair $\lensob{M}$ of manifolds and a multimorphism
\[f\colon\lensob{M_1},\cdots,\lensob{M_K}\to\lensob{N}\] 
consists of a tuple $f\coloneqq(\rv Q,\outp f,\inpt f,U)$ with $\rv Q\coloneqq(Q,\sharpR_Q):\rvect$ a responsive vector space, and, writing $\lensob M\coloneqq\lensob{M_1}\ltens\cdots\ltens\lensob{M_K}$ for the tensored domain, so that $\outp M=\prod_{i=1}^K\outp M_i$ and $\inpt M=\prod_{i=1}^K\inpt M_i$, the components are smooth maps
\begin{equation}\label{eqn.srw_morphism}
\begin{gathered}
\outp f\colon Q\times\outp M\to\outp N,\qquad
\inpt f\colon Q\times\outp M\times\inpt N\to\inpt M,\\
U\colon Q\times\outp M\times\inpt N\to\rr.
\end{gathered}
\end{equation}
\eqref{eqn.arrangement} depicts this data in the case $K=1$.

The formalism does not prescribe how these data should be chosen for each component, but the dynamics functor then guarantees that these locally-specified behaviors assemble functorially under interconnection.

\begin{definition}[Vocabulary]\label{def.arrangement_terminology}
We say that $f\in\sarr(b_1,\dots,b_K; c)$, a smooth adaptive arrangement given by $f=(Q,\sharpR_Q,\outp f,\inpt f,U)$ as in \eqref{eqn.srw_morphism} is:
\begin{description}[labelindent=2em, leftmargin=!, labelwidth=\widthof{\itshape a closed system}, labelsep=1em, align=left, font=\normalfont\itshape]
  \item[closed] if its codomain is the monoidal unit $I=\lensunit$;
  \item[a system] if its domain is $I=\lensunit$ (i.e.\ the arity is $K=0$); 
  \item[a closed system] if it is a morphism $I\to I$ (sometimes called a \emph{scalar});
  \item[stateless] if its parameter $Q=\rr^0$ is trivial;
  \item[potential-free] if its potential $U=0$ vanishes;
  \item[static] if it is both stateless and potential-free;
  \item[quadratic] if its responsive parameter space $\rv Q$ is quadratic (\cref{def.rvect});
  \item[harmonic] if it is quadratic, $\outp M$ and $\inpt N$ are vector spaces, and its potential $U$ is a quadratic form;%
  \footnote{A \emph{quadratic form} on a vector space $W$ is a function $W\to\rr$ of the form $w\mapsto B(w\otimes w)$ for a symmetric linear map $B\colon W\otimes W\to\rr$. It is common to assume positive definiteness, but we will not need it. Bilinear forms, and hence quadratic forms, are closed under direct sum.}
  \item[linear] if it is harmonic, its other interfaces $\inpt M$ and $\outp N$ are also vector spaces, and its maps $\outp f$ and $\inpt f$ are linear;
  \item[unilateral] if whenever $\outp f(q,\outp m)=\outp n$ and $\xi\in T^*_{\outp n}\outp N$, we have $\sharpR_q\bigl((\partial_q\outp f)^\top\xi\bigr)=0$ \eqref{eqn.partial_differential}, i.e.\ if the following diagram commutes
  \[
  \begin{tikzcd}
  	T^*_{\outp n}\outp N\ar[d, "\bang"']\ar[r, "(T_{(q,\outp m)}\outp f)^\top"]&[25pt]
		Q^*\oplus T^*_{\outp m}\outp M\ar[r, "\pi_{Q^*}"]&
		Q^*\ar[d, "\sharpR_q"]\\
		0\ar[rr, "\bang"']&&
		Q
  \end{tikzcd}
\qedhere
  \]
\end{description}
\end{definition}

We have now settled the choice of composition syntax $\cat{W}$ deferred at the end of \cref{sec.wd_operads}: from this point forward, we work in $\cat{W}\coloneqq\sarr$.

\begin{proposition}\label{prop.suboperads}
The stateless, potential-free, static, quadratic, and unilateral morphisms each span a wide suboperad of $\sarr$, and the linear morphisms also span a suboperad (on objects $\binom{V^-}{V^+}$ for vector spaces $V^-$ and $V^+$).
\end{proposition}
\begin{proof}
For the linear case, restrict throughout to colors $\binom{V^-}{V^+}$ for vector spaces $V^-$ and $V^+$. The other suboperads are asserted to be wide.

Static is just the intersection of stateless and potential-free. The identity has $Q=0$ and $U=0$, so it is stateless and potential free, and any stateless arrangement is quadratic and unilateral. It is also linear since its output is identity and its input is a projection (see \cref{def.lens}). So it suffices to show that stateless, potential-free,
quadratic, linear, and unilateral maps are each closed under composition.
In fact, we may assume they are 1-ary $\lensob{M}\to\lensob{N}$ at the cost of additionally showing that they are each also closed under tensor product.

Under tensor product, everything adds: $Q_{f\ltens g}=Q\oplus R$, $U_{f\ltens g}=U+V$ by the productor \eqref{eqn.monoid_productor} of the monad $\rr\times\blank$, $\sharpR_{(q,r)}=\sharpR_q\oplus\sharpR_r$ by \eqref{eqn.directsum_sharp}, and $\outp{(f\ltens g)}=\outp f\times\outp g$.
So stateless is preserved because $\rr^0\oplus\rr^0=\rr^0$, potential-free is preserved because $0+0=0$,
quadratic is preserved because $\sharpR_q\oplus\sharpR_r$ is again constant and symmetric, linear is preserved because $\outp{(f\ltens g)}$ and $\inpt{(f\ltens g)}$ are products of linear maps and $U+V$ is again a quadratic form, and unilateral is preserved because
\[
\sharpR_{(q,r)}\circ\bigl(\partial_{(q,r)}\outp{(f\ltens g)}\bigr)^\top
=\Bigl(\sharpR_q\circ(\partial_q\outp f)^\top\Bigr)\oplus\Bigl(\sharpR_r\circ(\partial_r\outp g)^\top\Bigr)
\]
is a direct sum of maps that each vanish by unilaterality of $f$ and $g$.

Under composition $f\colon\lensob M\to\lensob N$ and $g\colon\lensob N\to\lensob O$, again the state spaces and potentials add, $Q_{f\then g}=R\oplus Q$, and potentials are given by
\[
U_{f\then g}(r,q,\outp m,\inpt o)=U\bigl(q,\outp m,\inpt g(r,\outp n,\inpt o)\bigr)+V(r,\outp n,\inpt o),
\qquad\outp n\coloneqq\outp f(q,\outp m),
\]
coming from the multiplication of the writer monad $\rr\times\blank$ in the coKleisli composition of \cref{prop.backward_comonad}. So stateless is preserved because $\rr^0\oplus\rr^0=\rr^0$, and potential-free is preserved because $0+0=0$.
Quadratic is preserved because the sharp $\sharpR_r\oplus\sharpR_q$ of the composite is again constant and symmetric.
Linearity is preserved as well: the maps $\outp{(f\then g)}$ and $\inpt{(f\then g)}$ are each composites of linear maps, the composite is quadratic as above, and each summand of the displayed potential is a quadratic form precomposed with a linear map, hence a quadratic form; thus their sum is too.%
\footnote{
The linearity of $\inpt g$ is what keeps this potential quadratic, which is why linear arrangements form a suboperad but harmonic ones don't.
}
Thus it suffices to prove that unilaterality is preserved under composition.

So suppose $f$ and $g$ are unilateral, let $\outp n\coloneqq\outp f(q,\outp m)$ and $\outp o\coloneqq\outp g(r,\outp n)$, and let $\xi\in T^*_{\outp o}\outp O$. The output map of the composite $R\times Q\times\outp M\To{R\times\outp f}R\times\outp N\To{\outp g}\outp O$ is $\outp{(f\then g)}(r,q,\outp m)=\outp g\bigl(r,\outp f(q,\outp m)\bigr)$, so the chain rule gives
\[
\bigl(\partial_{(r,q)}\outp{(f\then g)}\bigr)^\top(\xi)
=\Bigl((\partial_r\outp g)^\top(\xi),\;(\partial_q\outp f)^\top(\xi')\Bigr)\in R^*\oplus Q^*,
\qquad\xi'\coloneqq(\partial_{\outp n}\outp g)^\top(\xi)\in T^*_{\outp n}\outp N.
\]
Applying the sharp $\sharpR_{(r,q)}=\sharpR_r\oplus\sharpR_q$ of the composite, we obtain
\[
\sharpR_{(r,q)}\left(\bigl(\partial_{(r,q)}\outp{(f\then g)}\bigr)^\top(\xi)\right)
=\Bigl(\sharpR_r\bigl((\partial_r\outp g)^\top(\xi)\bigr),\;\sharpR_q\bigl((\partial_q\outp f)^\top(\xi')\bigr)\Bigr)
=(0,0),
\]
each component vanishing by unilaterality.
\end{proof}

\section{Examples: tanks, electricity, and predator--prey}\label{sec.arrangement_examples}
\begin{example}\label{ex.tank_system}
Recall the weirdly-shaped water tank from \cref{ex.tank_sharp}, with cross-sectional area $A\colon\rr\to\rr_{>0}$ as a function of height. Assume that the cumulative-volume map $h\mapsto\int_0^h A(s)\,ds$ is a diffeomorphism, write $h(v)$ for its inverse, and take the tank's state to be its water volume $v\in\rr$. We consider it as a system
$
\begin{tikzpicture}[oriented WD, bb small, bb port length = 3mm, font=\tiny, baseline=(tank.-20)]
	\node[bb={0}{1}] (tank) {tank};
\end{tikzpicture}
$
with an outgoing pipe, i.e.\ a potential-free arrangement $\fun{tank}\colon\lensunit\to\binom{1}{\rr}$. Let $c>0$ denote water's weight per unit volume, and define
\[
\fun{tank}\coloneqq(\rr,\sharpR_A,\outp{\fun{tank}},\bang,0),\qqwhere
\sharpR_A(v)(\xi)\coloneqq\frac{A(h(v))}{c}\xi,\qquad
\outp{\fun{tank}}(v)\coloneqq ch(v).
\]
We interpret its output as the pressure at the base of the tank. Applying $\frac{d}{dv}$ to both sides of $v=\int_0^{h(v)}A(s)\,ds$ gives
\[1=A(h(v))\frac{dh}{dv}\qqthus\frac{d}{dv}\outp{\fun{tank}}(v)=\frac{c}{A(h(v))}.\]
This is how we chose the sharp: its factor $A(h(v))/c$ cancels the factor $c/A(h(v))$ introduced by differentiating the output map.
\end{example}

\begin{example}\label{ex.tanks_connected}
We can connect two weirdly-shaped tank systems, with cross-sections $A_1,A_2$, of the form
$\fun{tank}_i=(\rr,\sharpR_{A_i},\outp{\fun{tank}_i},\bang,0)$
as in \cref{ex.tank_system} using a pipe
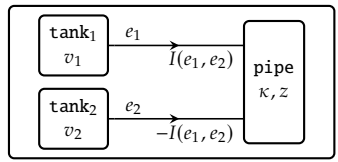
\[
\tankpic
\]
Here the pipe system $\fun{pipe}\colon\lensunit\to\binom{\rr^2}{1}$ is given by
$\fun{pipe}=(\rv 0,\bang,\bang,V_z)$,
i.e.\ stateless with potential
$V_z(e_1,e_2)\coloneqq\tfrac{\kappa}{2}((e_1-e_2)+cz)^2$, where $e_1,e_2$ are the pressures arriving at the pipe's two inputs, $\kappa$ is conductance, $c$ is water's weight per unit volume and $z$ is the elevation of the first tank's base above the second's, which we assume is known to the pipe. The $i$th input of $\fun{pipe}$ receives $e_i\coloneqq\outp{\fun{tank}_i}(v_i)$, and we write $I\colon\rr^2\to\rr$ for the function
\[
I(e_1,e_2)\coloneqq\kappa\bigl((e_1-e_2)+cz\bigr),
\]
so that $dV_z|_{(e_1,e_2)}=\bigl(I(e_1,e_2),\,-I(e_1,e_2)\bigr)$. We will see in \eqref{eqn.covector_triple} that, semantically, covectors flow backwards through wires, and that those covectors are given by the differential of the potential. Thus $\fun{pipe}$ returns $I(e_1,e_2)$ on its first port and $-I(e_1,e_2)$ on its second, as the picture shows.
The whole wiring diagram shown is a closed static arrangement of the form
\begin{equation}\label{eqn.tank_wire}
\fun{wire}\colon\binom{1}{\rr}\ltens\binom{1}{\rr}\ltens\binom{\rr^2}{1}\to\lensunit
\end{equation}
given by $\fun{wire}=(\rv 0,\bang,\id,0)$, i.e.\ the only content being that the lens feeds the outputs of the tanks to the inputs of the pipe by the isomorphism $\rr\times\rr\times 1\times 1\to1\times 1\times \rr^2$, which we're calling $\id$ by abuse of notation.

Composing, we get a map
$\fun{wire}(\fun{tank}_1,\fun{tank}_2,\fun{pipe})\colon\lensunit\to\lensunit$
whose state space is $\rr^2$ with sharp $\sharpR_{A_1}\oplus\sharpR_{A_2}$ (\cref{prop.rvect_monoidal}), and whose potential $\rr^2\to\rr$ is
$(v_1,v_2)\mapsto V_z\bigl(ch_1(v_1),ch_2(v_2)\bigr)$.

If we had instead let $\fun{wire}$ carry as state the relative elevation $\rv Z\coloneqq(\rr,\sharpR_Z)$ and fed that value to a more expressive pipe system $\lensunit\to\binom{\rr^3}{1}$ with potential $V(e_1,e_2,z)\coloneqq\tfrac{\kappa}{2}((e_1-e_2)+cz)^2$, then $\fun{wire}$ would no longer be static and the elevation difference could move to bring the two water surfaces level. A choice of degenerate $\sharpR_Z=0$ is like a weld holding $z$ fixed, as would putting $z$ into the state of a fourth---this time unilateral---component.
\end{example}

\begin{remark}[Electrical reading]\label{rmk.tank_electrical}
The well-known analogy between hydraulics and electrical circuits lets us consider the arrangement \eqref{eqn.tank_wire} as an electrical circuit. Each output port is drawn as the difference between two terminals, like a tank has an opening at the top and bottom. Making one ground $0\in\rr$ is like choosing one tank opening to be open to the air at some known atmospheric pressure:
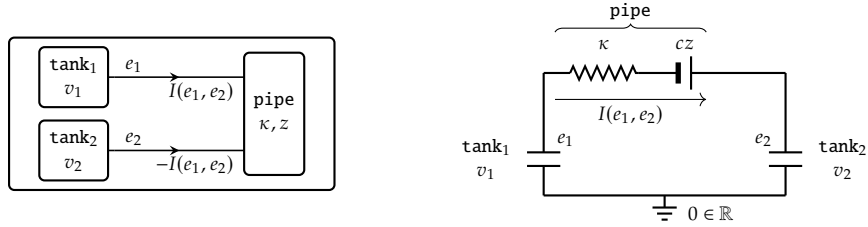
\[
\tankpic
\qquad\qquad
\begin{tikzpicture}[font=\tiny, thick, baseline=(current bounding box.center)]
	\coordinate (n1) at (0,1.9);
	\coordinate (n2) at (3.2,1.9);
	\draw (n1) -- (.35,1.9);
	\draw[decorate, decoration={zigzag, segment length=4.5pt, amplitude=2.5pt}] (.35,1.9) -- (1.25,1.9);
	\draw (1.25,1.9) -- (1.78,1.9);
	\draw[line width=2pt] (1.78,1.76) -- (1.78,2.04);
	\draw (1.95,1.68) -- (1.95,2.12);
	\draw (1.95,1.9) -- (n2);
	\draw[->, thin] (.15,1.55) -- (2.15,1.55) node[midway, below=-1.5pt] {$I(e_1,e_2)$};
	\draw[decorate, decoration={brace, amplitude=3pt}, thin] (.15,2.44) -- (2.15,2.44) node[midway, above=1pt] {$\fun{pipe}$};
	\node[above=0pt] at (1.87,2.07) {$cz$};
	\node[above=0pt] at (.80,2.07) {$\kappa$};
	\draw (n1) -- (0,.85);
	\draw (-.22,.85) -- (.22,.85);
	\draw (-.22,.67) -- (.22,.67);
	\draw (0,.67) -- (0,.28);
	\node[left=2pt, align=center] at (-.22,.76) {$\fun{tank}_1$\\$v_1$};
	\draw (3.2,1.9) -- (3.2,.85);
	\draw (2.98,.85) -- (3.42,.85);
	\draw (2.98,.67) -- (3.42,.67);
	\draw (3.2,.67) -- (3.2,.28);
	\node[right=2pt, align=center] at (3.42,.76) {$\fun{tank}_2$\\$v_2$};
	\node[right=1pt] at (0,1.02) {$e_1$};
	\node[left=1pt] at (3.2,1.02) {$e_2$};
	\draw (0,.28) -- (3.2,.28);
	\draw (1.6,.28) -- (1.6,.12);
	\draw (1.44,.12) -- (1.76,.12);
	\draw (1.50,.04) -- (1.70,.04);
	\draw (1.56,-.04) -- (1.64,-.04);
	\node[right=2pt] at (1.70,.04) {$0\in\rr$};
\end{tikzpicture}
\]
Every constituent of $\fun{wire}(\fun{tank}_1,\fun{tank}_2,\fun{pipe})$ has an analogue in the circuit, except $c$: water's weight per unit volume is the translator between the two columns:
\[
\begin{array}{ll}
\text{hydraulics} & \text{electricity}\\\hline
\text{the }i\text{th tank} & \text{the }i\text{th capacitor}\\
v_i\text{, the volume of water in the }i\text{th tank} & \text{the charge on the }i\text{th capacitor}\\
e_i\coloneqq\outp{\fun{tank}_i}(v_i)\text{, the pressure at its base} & \text{the voltage at the }i\text{th node}\\
\sharpR_{A_i}\text{, shape as $\Delta$ volume per unit pressure} & \text{capacitance as $\Delta$ charge per unit voltage}\\[4pt]
\text{the pipe and its elevation drop} & \text{a resistor and a voltage source, in series}\\
cz\text{, the pressure from that drop} & \text{the voltage from that source}\\
\kappa\text{, the conductance of the pipe} & \text{the conductance of the resistor}\\
I\text{, the flow through the pipe} & \text{the current through the source and resistor}\\[4pt]
\text{a wire} & \text{a node}\\
0\in\rr & \text{the ground}\\
\text{the outer box: closed} & \text{external terminals: none}
\end{array}
\qedhere
\]
\end{remark}

\begin{example}[Lotka--Volterra]\label{ex.lotka_volterra}
We can use the same system, except with oppositely-signed responses, to model predator-prey systems. Wire two species systems to a predation system as above:
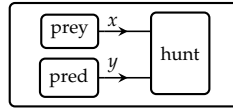
\[
\begin{tikzpicture}[oriented WD, bb small, font=\tiny]
	\node[bb port sep=3mm, bb={0}{0}] (hunt) {hunt};
	\node[bb={0}{1}, above left=2pt and 1.75 of hunt.180] (s1) {prey};
	\node[bb={0}{1}, below left=2pt and 1.75 of hunt.180] (s2) {pred};
	\node[bb={0}{0}, fit=(s1) (s2) (hunt)] (outer) {};
	\draw[ar] (s1_out1) -- (s1_out1-|hunt.west);
	\node[above right=-2pt and -3pt] at (s1_out1) {$x$};
	\draw[ar] (s2_out1) -- (s2_out1-|hunt.west);
	\node[above right=-2pt and -3pt] at (s2_out1) {$y$};
\end{tikzpicture}
\]
Each species system $\fun{prey},\fun{pred}\colon\lensunit\to\binom{1}{\rr}$ carries as state the logarithm of its population---noting that $\rr_{>0}$ is not a vector space but its logarithm is---and reports the population itself, $x=\outp{\fun{prey}}(u)\coloneqq e^u$ and $y=\outp{\fun{pred}}(v)\coloneqq e^v$.

The predation system $\fun{hunt}\colon\lensunit\to\binom{\rr^2}{1}$ is stateless, $\fun{hunt}=(\rv 0,\bang,\bang,V)$ with mass-action potential $V(x,y)\coloneqq\beta xy$ using an encounter rate of $\beta>0$. The systems are combined using $\fun{wire}=(\rv 0,\bang,\id,0)$ as in \eqref{eqn.tank_wire}. The species differ only in their responses, which as in \cref{ex.tank_sharp} are not constant,
\[
\sharpR_{\fun{prey}}(u)\colon\xi\mapsto e^{-u}\xi
\qqand
\sharpR_{\fun{pred}}(v)\colon\xi\mapsto-\tfrac{\delta}{\beta}e^{-v}\xi,
\]
so the prey descends $V$ while the predator ascends it; the factor $\delta/\beta$ in $\sharpR_{\fun{pred}}$ is the number of predators produced per prey consumed. Here $\xi$ is a flux of individuals per unit time,
and each response divides by the population size to obtain the change of log-population. The species' potentials are $U_{\fun{prey}}(u)\coloneqq-\alpha e^u$ and $U_{\fun{pred}}(v)\coloneqq-\tfrac{\gamma\beta}{\delta}e^v$, for a prey birth rate $\alpha>0$ and a predator death rate $\gamma>0$. The composite closed system then has state space $\rr^2$, sharp $\sharpR_{\fun{prey}}\oplus\sharpR_{\fun{pred}}$, and potential given by the sum of those above, with $V$ evaluated at the reported populations $x$ and $y$:
\[
U(u,v)=\beta e^{u+v}-\alpha e^u-\tfrac{\gamma\beta}{\delta}e^v.
\]
Applying the negative sharp to the differential of the potential gives a
state-space
vector field $(\dot u,\dot v)=-(\sharpR_{\fun{prey}}\oplus\sharpR_{\fun{pred}})(dU|_{(u,v)})$.
Since $x=e^u$ and $y=e^v$, we have $\dot x=x\dot u$ and $\dot y=y\dot v$; hence
we obtain
\[
\dot x=x(\alpha-\beta y)
\qqand
\dot y=y(\delta x-\gamma),
\]
the Lotka--Volterra predator--prey equations. In population coordinates the potential is $\beta xy-\alpha x-\tfrac{\gamma\beta}{\delta}y$, and its unique critical point is the coexistence equilibrium $(x,y)=(\gamma/\delta,\alpha/\beta)$.

We can make different ecological modeling choices. For example, making both responses positive makes each species descend the interaction potential $V$. Adding $\tfrac{\rho}{2}e^{2u}$ to $U_{\fun{prey}}$ slows the prey's growth as its population increases, modeling limited resources. Finally, one could add an input port to $\fun{prey}$ and make the potential depend on it \eqref{eqn.para_potential_lens_maps}, in which case an environmental signal could drive $\alpha$.
\end{example}

\begin{example}\label{ex.further_reach}
There are many more examples from electronics, mechanics, control theory, chemistry, biology, and neuroscience that appear to be modellable in this framework. 

Each construction below is implemented in the companion code (\url{https://github.com/dspivak/dap}) by Claude, which used the definitions here to make executable code that was then tested against idealized standard equations. The author understands the content of the claims as stated, and believes them to be reasonable and to have been checked, but \textit{has not verified all of them by hand.}

\begin{description}
	\item[Transistor.] A transistor is a system with interface $\binom{\rr^3}{1}$ whose three inputs receive the gate, drain, and source voltages; whose responsive vector space is $Q\coloneqq(\rr^2,\sharpR)$ with $\sharpR(\xi,\xi')\coloneqq(\xi',0)$; and whose potential $U\colon Q\times \rr^3\to\rr$ is
\[
U\bigl((q,q'),e_G,e_D,e_S\bigr)\coloneqq\tfrac{\kappa}{2}\bigl(e_G-e_S-q-q'\bigr)^2+W\bigl(q,\,e_D-e_S\bigr)
\]
for a constant $\kappa>0$ and a function $W\colon\rr^2\to\rr$ encoding the channel's current--voltage law (e.g.\ it distinguishes NMOS from PMOS). The idea is that in a state $(q,q')$, only the sum $q+q'$ is meaningful: it relaxes toward the gate--source voltage $e_G-e_S$, whereas $\sharpR$ freezes $q'$, ensuring that the summand involving the drain voltage $e_D$ is unable to change the state.
\item[NAND gate.]
A NAND gate is given by applying a static wiring $\binom{\rr^3}{1}^{\ltens 4}\ltens\binom{1}{\rr}^{\ltens 2}\ltens\binom{1}{\rr}\to\binom{\rr^2}{\rr}$ to four transistors---two NMOS and two PMOS \cite{wikipedia2026nand}---two capacitors, and a supply voltage held fixed by the degenerate $\sharpR=0$; when the two input voltages are each held at one of two fixed values, the output voltage converges to the value prescribed by the NAND truth table.
\item[Logic gates.] We obtain logic gates from the NAND gate by static wiring in the usual way. In our implementation, a chain of $K$ inverters takes roughly $K$ times as many ticks of the coalgebra to converge as one does, whereas $K$ side-by-side copies take the same number of ticks as one.
\item[Sequential circuits.]
Cross-coupling two NAND gates yields a latch \cite{wikipedia2026flipflop}---a system with two stable states, each preserved by the dynamics for the parameters we tried, i.e.\ a stored bit of memory---and, when combined with logic gates, sequential circuitry as well.
\item[Amplifier.]
An amplifier is a system with interface $\binom{\rr}{\rr}$: it has the transistor's responsive vector space $(\rr^2,\sharpR)$, it reports its state $(q,q')\mapsto q$, and its potential $U(q,q',e)\coloneqq\tfrac{\kappa}{2}(Ae-q-q')^2$ makes $q$ relax toward the amplification $A\gg 0$ times the input $e$.
Moreover, the amplifier is unilateral: any covector returned at its output port lands in the kernel of $\sharpR$, so whatever is wired to the output cannot move the state.
\item[Feedback amplifier.]
Black's feedback amplifier \cite{black1934stabilized} is the system given by applying a
static arrangement $b\colon\binom{\rr}{\rr}\to\binom{\rr}{\rr}$  to the amplifier above. The output map of $b$ is identity and its input map is $(v,e_{\tn{src}})\mapsto e_{\tn{src}}-\beta v$, where $v$ is emitted from the amp and $e_{\tn{src}}$ is the source voltage input to the system.  Black invented this device for long-distance telephony, where precise amplification is necessary. At equilibrium the system reports $\tfrac{A}{1+A\beta}\approx\tfrac1\beta$ times the source, so the amplification is governed by $\beta$---realized by a comparatively precise passive resistor---rather than by $A$, which was generally quite imprecise.
\item[LC circuit.]
An LC circuit is a system that resonates with incoming signals, e.g.\ radio, formed by a stateless arrangement of an inductor and a capacitor
\[
\lensunit\To{\fun{ind}\ltens\fun{cap}}\binom{1}{\rr}\ltens\binom{1}{\rr}\To{f}\binom{1}{\rr}
\]
The inductor has responsive vector space $(\rr,\sharpR_{\fun{ind}})$, $\sharpR_{\fun{ind}}(\xi)\coloneqq-\xi/L$ for an inductance $L>0$ and potential $-\tfrac{R}{2}i^2$ for a small resistance $R>0$, the sign matching that of $\sharpR_{\fun{ind}}$; its output map reports the state $i\in\rr$. The capacitor is as in \cref{rmk.tank_electrical} reporting a state $v\in\rr$. The arrangement $f$ is stateless with potential $U(i,v)\coloneqq iv$ and output map $(i,v)\mapsto v$: the result is an oscillation as in the predator-prey model (\cref{ex.lotka_volterra})---caused by the opposite-signed sharps---with period roughly $2\pi\sqrt{LC}$ ticks, decaying whenever $1/C<R<2\sqrt{L/C}$. A radio signal enters as a covector---a current drive---returned at the terminal; the oscillation is strongest when the drive frequency matches the circuit's own, tuning to the signal from the matching radio station.%
\footnote{
LC circuits can also be integrated using the phase integrator, if we replace the two systems above with a single one whose state is the charge $q\coloneqq Cv$, responsive vector space $(\rr,\xi\mapsto\xi/L)$, and potential $q^2/2C$. Doing so keeps the energy bounded whenever $4LC>1$.
}
\item[DC motor.]
A DC motor converts electrical current into mechanical rotation. It has the shape of the LC circuit: an electrical component and a mechanical one with states the current $i$ and the angular velocity $\omega$; sharps $\xi/L$ and $-\xi/J$ (inductance $L$, inertia $J$); cell potentials $\tfrac{R}{2}i^2-Vi$ (winding resistance $R$, source voltage $V$) and $-(\tfrac{b}{2}\omega^2+\tau\omega)$ (friction $b$, load torque $\tau$); and coupler potential $Ki\omega$ (motor constant $K$). One tick is exactly the Euler step of the motor's standard equations~\cite[Ex.~2.5]{vanderschaft2014port}.
\item[Molecular dynamics.]
Molecular dynamics simulates a molecule or material as $d$-many point particles with pairwise interactions. Each particle is a potential-free system
$\lensunit\to\binom{1}{\rr^{3}}$ with responsive vector space $(\rr^{3},\sharpR)$,
$\sharpR(\xi)\coloneqq\xi/m$ for a mass $m>0$, reporting its
state, the position $x\in\rr^3$;
each interacting pair is a stateless system
$\lensunit\to\binom{\rr^{6}}{1}$  whose potential is the Lennard--Jones energy $4\epsilon\bigl((\sigma/r)^{12}-(\sigma/r)^6\bigr)$, $r$ being the distance between the two positions it receives and $\sigma>0$ a length scale; pairs are wired to particles along the interaction graph as in \cref{sec.graph_laplacian}.
Under the configuration dynamics the particles descend the emergent total energy---structure relaxation---a pair settling at separation exactly
$2^{1/6}\sigma$. Actual molecular dynamics, however, requires a better integrator: the two-round integrator of \cref{rmk.multistage} applied to this arrangement is exactly velocity Verlet, which is its standard integration loop.
\item[Impedance control.]
Impedance control~\cite{hogan1985impedance} makes a robot arm behave as though it were attached by a spring to its target: push it off the target and it pushes back proportionally. The spring term is a stateless system $\lensunit\to\binom{\rr^2}{1}$ whose two inputs receive a reported position $q$ and a run-time setpoint $s$, and whose potential is $\tfrac{k}{2}(q-s)^2$, where $k$ is stiffness. Wired to a system reporting its state $q$, it pulls $q$ toward $s$; the force returned at the setpoint port is equal and opposite, and in most control systems would be discarded by making the setpoint unilateral.
\item[Kuramoto oscillators.]
A heart's pacemaker cells and the generators on a power grid are modeled as synchronizing by the same mechanism, the Kuramoto equations. Each oscillator is a system $\lensunit\to\binom{1}{\rr}$ reporting its state, the phase $\theta\in\rr$, with responsive vector space $(\rr,\sharpR)$, $\sharpR(\xi)\coloneqq\xi$, and potential $-\omega\theta$ (natural frequency $\omega$); each edge of an interaction graph is a stateless system $\binom{\rr^2}{1}$ with potential $-K\cos(\theta-\theta')$ in the two reported phases (coupling strength $K$). Differentiating, the covector returned to an oscillator is $-\omega+K\sum\sin(\theta-\theta')$, summed over its neighbors, so one tick is $\theta\mapsto\theta+\omega+K\sum\sin(\theta'-\theta)$---the Euler step of the Kuramoto equations.%
\footnote{The phase integrator applied to the same arrangement yields the undamped swing equation of power grids: a generator is a Kuramoto oscillator with inertia.}
\item[Petri nets.]
A chemical reaction network can be modeled by a Petri net
$(S,T,m,n,r,x^0)$ with arc multiplicities $m,n\colon T\times S\to\nn$, rate constants $r\colon T\to\rr_{\geq 0}$, and initial concentrations $x^0\colon S\to\rr_{\geq0}$: a transition $t$ consumes $m_{ts}$ and produces $n_{ts}$ tokens of the species $s$, and fires at the mass-action rate $f_t(x)\coloneqq r(t)\prod_sx_s^{m_{ts}}$ in the concentrations $x$.
Each transition $t\in T$ is a system $\lensunit\to\binom{\rr^{S_t}}{\rr}$ whose inputs receive the concentrations at the set $S_t$ of species it touches; it has the transistor's responsive vector space $(\rr^2,\sharpR)$, $\sharpR(\xi,\xi')\coloneqq(\xi',0)$; it reports the first coordinate of its state $(\epsilon_t,\epsilon'_t)\mapsto\epsilon_t$, the extent of reaction; and
its potential is $U_t\bigl((\epsilon_t,\epsilon'_t),x\bigr)\coloneqq-\epsilon'_t f_t(x)$. The transition-systems are wired together by the static arrangement $\bigltens_{t\in T}\binom{\rr^{S_t}}{\rr}\to\lensunit$ whose input map is affine, as for the feedback amplifier above: it sends the reported extents to the concentration $x_s\coloneqq x^0_s+\sum_t\bigl(n_{ts}-m_{ts}\bigr)\epsilon_t$ at each species $s$.
Differentiating, the sharp discards $\partial U/\partial\epsilon_t$ and delivers $\partial U/\partial\epsilon'_t=-f_t(x)$, so one tick is $\epsilon_t\mapsto\epsilon_t+f_t(x)$, the Euler step of mass action at any molecularity. Note that the chemical reading requires each concentration $x_s$ to remain nonnegative, a condition the update does not enforce.%
\footnote{Nonnegativity can be obtained by construction by re-encoding the same net in log coordinates, at the cost of exact conservation laws, and likewise for the replicator tick below; see the companion code for details.}
Note that $U$ does not represent free energy and, as in the transistor, $\epsilon'_t$ represents nothing---it is frozen and only there to channel covectors to the real state $\epsilon$.
\item[Replicator dynamics.]
Evolutionary game theory tracks the fractions $x_1,\dots,x_n$ of a population playing each of $n$ strategies: a strategy grows when it earns more than the population average. This is a single closed system with state $x\in\rr^n$, state-dependent Shahshahani sharp $\sharpR_x(\xi)\coloneqq\bigl(\operatorname{diag}(x)-xx^\top\bigr)\xi$, and potential $U(x)\coloneqq-\tfrac12x^\top Ax$ for a symmetric payoff matrix $A$ (a potential game). Differentiating, one tick is $x_i\mapsto x_i+x_i\bigl((Ax)_i-x^\top Ax\bigr)$---the Euler step of the replicator equation, each share growing by its payoff excess over the average---and when $\sum_ix_i=1$ the update preserves it; as with the Petri tick above, nonnegativity is a condition the update does not enforce.
\item[Hopfield networks.]
A Hopfield network~\cite{hopfield1984neurons} with suitable weights stores memories as attractors of a neural dynamics: started near a stored pattern, the state falls into it. Each neuron is a system $\lensunit\to\binom{1}{\rr}$ with state $q\in\rr$, state-dependent sharp $\sharpR_q(\xi)\coloneqq\cosh^2(q)\,\xi$, and potential $U(q)\coloneqq q\tanh q-\log\cosh q-I\tanh q$ (bias current $I$), reporting the activation $V\coloneqq\tanh q$; each synapse is a stateless system $\binom{\rr^2}{1}$ with potential $-W\,VV'$ in the two reported activations (weight $W\in\rr$). One tick of configuration dynamics reproduces the Euler step of Hopfield's equations, and the emergent composite potential is exactly Hopfield's energy function.
\end{description}
Each component in each of these examples is specified in its own native variables by a potential and a sharp; any dynamic arrangement of such components is again such a system. All the above components can be mixed and matched---combined without worry---as long as one is willing for them to receive covectors at their output ports in addition to signals at their input ports; indeed, the coupling terms including loading effects are derived from the shared potentials rather than written by hand as in block-diagram tools such as Simulink.

A few additional examples are worked out in complete detail---dynamics included---in \cref{ch.applications}.
\end{example}

\chapter{Smooth adaptive dynamics}\label{ch.smooth_dynamics}

This section turns the syntax $\sarr$ of \cref{ch.smooth_rwd} into dynamics. By the framework of \cref{ch.framework}, this takes two semantic ingredients---a polynomial interpretation and an integrator---which compose as
\[
\sarr\To{\Phip{\interpsm}}\para{\cot}{\poly}\To{\Psisem{\intg}}\pc.
\]
The framework's dynamics functor $\Phi_{\interp,\intg}$ (\cref{thm.dynamics_functor}) records both ingredients; specialized here it is $\Phi_{\interpsm,\intg}$. Since the interpretation is fixed to the smooth one $\interpsm$ throughout this section, we drop it from the notation and write $\Phi_\intg$ for $\Phi_{\interpsm,\intg}$ (and $\Phiconf,\Phiphase$ for the configuration and phase integrators of \cref{sec.integrator_semantics}).

\Cref{sec.cot} builds the cotangent functor $\cot\colon\mfd\to\poly$, which is involved in both. \Cref{sec.smooth_interpretation} assembles the polynomial interpretation $\Phip{\interpsm}$, casting each interface as a polynomial and each adaptive arrangement as a $\cotof{Q}$-parameterized polynomial map. \Cref{sec.integrator_semantics} constructs the two integrators used in this paper---configuration and phase---each inducing a semantics $\Psisem{\intg}\colon\para{\cot}{\poly}\to\pc$. \Cref{sec.dynamics_functor} composes the two to recover the operad functors $\Phiconf,\Phiphase\colon\sarr\to\pc$ (\cref{cor.functor}), and \cref{sec.configuration_dynamics,sec.phase_dynamics} unpack the differential geometry underlying each in full detail.

\section{Cotangent bundles as polynomials}\label{sec.cot}

The category $\mfd$ of smooth manifolds provides the source of polynomials relevant to mechanics and learning. Each manifold has a cotangent bundle, and each smooth map has a transpose going backwards on cotangent fibers. This keeps $T^*$ from being functorial as a map to $\mfd$, but it is exactly what lets it be a (surprisingly well-behaved) functor to $\poly$.

\begin{definition}[Cotangent functor]\label{def.cot}
The \emph{cotangent functor} $\defineTerm{cot}\colon\mfd\to\poly$ sends a smooth manifold $M$ to
\begin{equation*}
\cotof{M}\coloneqq\sum_{m\in M}\yon^{T^*_mM}
\end{equation*}
where $T^*_mM$ is the cotangent space at $m$. A smooth map $f\colon M\to N$ is sent to the $\poly$ map $\cotof{f}\colon\cotof{M}\to\cotof{N}$ whose positions and directions are
\[
  \cotof{f}_1\coloneqq f
  \qqand
  \bk{\cotof{f}}{m}\coloneqq(T_m f)^\top
 \]
 as maps $M\to N$ and $T^*_{f(m)}N\to T^*_mM$, respectively.
\end{definition}

This is indeed functorial: the identity map $\id_M\colon M\to M$ has $T_m(\id_M)=\id_{T_mM}$, whose transpose is $\id_{T^*_mM}$, so $\cotof{\id_M}=\id_{\cotof{M}}$. For composable smooth maps $f\colon M\to N$ and $g\colon N\to P$, the chain rule gives $T_m(g\circ f)=T_{f(m)}g\circ T_m f$, so
\[
(T_m(g\circ f))^\top=(T_m f)^\top\circ(T_{f(m)}g)^\top=\bk{\cotof{f}}{m}\circ\bk{\cotof{g}}{f(m)},
\]
which is the backward part of $\cotof{g}\circ\cotof{f}$.

\begin{proposition}\label{prop.cot_monoidal}
The cotangent functor $\cot\colon(\mfd,1,\times)\to(\poly,\yon,\otimes)$ is strong symmetric monoidal:
\[
\cotof{1}\cong\yon
\qqand
\cotof{M\times N}\cong\cotof{M}\otimes\cotof{N}.
\]
\end{proposition}

\begin{proof}
The one-point manifold $1=\{0\}$ has $T^*_01=\{0\}$, so $\cotof{1}=\sum_{\{0\}}\yon^{\{0\}}\cong\yon$, i.e.\ the unit is preserved. By \eqref{eqn.T_products}, $\absval{T^*_{(m,n)}(M\times N)}\cong\absval{T^*_mM\oplus T^*_nN}\cong\absval{T^*_mM}\times\absval{T^*_nN}$, so we have
\begin{equation}\label{eqn.cot_times}
\cotof{M\times N}=\sum_{(m,n)\in M\times N}\yon^{T^*_{(m,n)}(M\times N)}\cong\sum_{m\in M}\yon^{T^*_mM}\otimes\sum_{n\in N}\yon^{T^*_nN}=\cotof{M}\otimes\cotof{N}
\end{equation}
using the definition of the Dirichlet product \eqref{eqn.dirichlet}.

This isomorphism $\cotof{M\times N}\cong\cotof{M}\otimes\cotof{N}$ is natural: given a product map $f\times g\colon M\times N\to M'\times N'$, the tangent map is block-diagonal: $T_{(m,n)}(f\times g)=T_m f\oplus T_n g$. Hence $(T_{(m,n)}(f\times g))^\top=(T_m f)^\top\times(T_n g)^\top$.
%
\end{proof}

\begin{remark}
In fact, $\cot$ also preserves coproducts,
\[
\cotof{M}+\cotof{N}\cong\cotof{M+N}.
\]
Hence $\cot\colon(\mfd,0,+,1,\times)\to(\poly,0,+,\yon,\otimes)$ is a distributive monoidal functor. We will not use this fact in this paper.
\end{remark}

\Cref{prop.cot_monoidal} uses only the underlying sets of the cotangent fibers, e.g.\ in \eqref{eqn.cot_times}. The linear structure is not lost, however: \cref{lem.cot_comonoid} shows that the diagonal and terminal maps in $\mfd$ induce fiberwise addition and zero on $\cotof{M}$.

Since $(\mfd,1,\times)$ is cartesian, every manifold $M$ is a comonoid via the diagonal $\Delta_M\colon M\to M\times M$ and $\bang_M\colon M\to 1$. And since $\cot$ is strong symmetric monoidal (\cref{prop.cot_monoidal}), it preserves comonoid structures. 

\begin{lemma}\label{lem.cot_comonoid}
For any smooth manifold $M$, the polynomial $\cotof{M}$ is a comonoid in $(\poly,\yon,\otimes)$. The comultiplication and counit
\[
\cotof{\Delta_M}\colon\cotof{M}\to\cotof{M}\otimes\cotof{M}
\qqand
 \cotof{\bang_M}\colon\cotof{M}\to\yon
\]
send positions by $m\mapsto(m,m)$ and $m\mapsto *$, respectively. On directions, they endow each cotangent fiber $T^*_mM$ with a \emph{commutative monoid} structure
\[
T^*_mM\times T^*_mM\to T^*_mM,\quad (\xi,\xi')\mapsto\xi+\xi'
\qqand \rr^0\to T^*_mM,\quad *\mapsto 0,
\]
i.e.\ fiberwise addition and zero. Thus the $\otimes$-comonoid structure on $\cotof{M}$ recovers the additive structure of each cotangent space.%
\end{lemma}

\begin{proof}
On positions $\cotof{\Delta_M}$ and $\cotof{\bang_M}$ are given $m\mapsto(m,m)$ and $m\mapsto *$. On directions, $T_m\Delta_M$ is the diagonal $T_mM\to T_mM\times T_mM$, $v\mapsto(v,v)$; its transpose is addition $T^*_mM\times T^*_mM\to T^*_mM$. The map $T_m(\bang_M)$ is $T_mM\to\rr^0$, whose transpose is the map $\rr^0\to T^*_mM$ that selects the zero covector.
\end{proof}

\begin{proposition}[Hopf monoid structure on cotangent bundles of Lie groups]\label{prop.cot_hopf}
Let $G$ be a Lie group with unit $\eta\colon 1\to G$, multiplication $\mu\colon G\times G\to G$, and inversion $\iota\colon G\to G$. Then $\cotof{G}$ is a Hopf monoid in $(\poly,\yon,\otimes)$: the monoid structure is
\[
\cotof{\eta}\colon\yon\to\cotof{G},\qquad \cotof{\mu}\colon\cotof{G}\otimes\cotof{G}\to\cotof{G},
\]
the comonoid structure is as in \cref{lem.cot_comonoid}, and the antipode is $\cotof{\iota}\colon\cotof{G}\to\cotof{G}$.
\end{proposition}

\begin{proof}
It is well known that strong monoidal functors preserve monoids, comonoids, and bimonoids: they preserve both the defining maps and the axioms relating them. For the same reason they also preserve the antipode structure and axioms, hence the Hopf monoid structure. See also~\cite{spivak2023hopf}.
\end{proof}

In fact, we will only need that if $G$ is a commutative monoid object in $\mfd$ then $\cotof{G}$ is a commutative monoid object in $\poly$.

\section{The smooth polynomial interpretation}\label{sec.smooth_interpretation}

We now interpret the syntax $\sarr$ in $\poly$ along the cotangent functor $\cot$ (\cref{sec.cot}). For the semantic effect and effect comparison we take the canonical choices that come with $\cot$: the writer monad $\cotof{\rr}\otimes\blank$ and the comparison $\theta$ from its strong monoidality (\cref{lem.monoid_to_monad}). The one ingredient we actually deliberate over is the potential algebra, so we pin it down first; \cref{def.smooth_setup} then assembles all four into the interpretation $\Phip{\interpsm}\colon\sarr\to\para{\cot}{\poly}$.

The potential algebra is a $(\cotof{\rr}\otimes\blank)$-monoid $(z,e,m,\alpha)$ in $\poly$, as needed for lens internalization (\cref{thm.Theta_T_alpha}). With a view toward our mechanics and learning applications, we take $z\coloneqq\yon$. It has a unique monoid structure, but $\alpha\colon\cotof{\rr}\to\yon$ is extra data.

\begin{lemma}\label{lem.alpha_constant}
Consider the monad $\Fun{T}\coloneqq(\cotof{\rr}\otimes\blank)$ on $\poly$. The $\Fun{T}$-monoid structures $(z,e,m,\alpha)$ for which $z=\yon$ are in bijection with real numbers $r:\rr$.
\end{lemma}

\begin{proof}
There is a unique $\otimes$-monoid structure on $\yon$. Maps $\alpha\colon\cotof{\rr}\to\yon$ are in bijection with (not-necessarily continuous) covector fields $\omega\colon\rr\to T^*\rr$ on $\rr$, where $\alpha$ is identified with $\omega(q)\coloneqq\bk{\alpha}{q}(*)$. The first condition of \eqref{eqn.T_monoid} is vacuous. Thus the only condition on $\alpha$ necessary to make it a $\Fun{T}$-monoid is the second condition, which requires that the following equation holds for all $q_1,q_2\in\rr$
\[
\left(\omega(q_1),\omega(q_2)\right)=
\left(\omega(q_1+q_2),\omega(q_1+q_2)\right).
\]
Setting $q_2\coloneqq 0$, we find $\omega(0)=\omega(q_1)$ for all $q_1$. In other words, $\omega(\blank)=r$ is constant at some $r:\rr$. And conversely, any constant covector field is a $\Fun{T}$-monoid.
\end{proof}

\begin{remark}[General form]\label{rmk.alpha_general}
\Cref{lem.alpha_constant} is an instance of a general fact: $\cotof{\rr}$ can be replaced by any commutative Hopf monoid $h$ in $\poly$. The $(h\otimes\blank)$-monoid structures on $\yon$ are then in bijection with $h[e_h]$, the directions of $h$ at its unit position $e_h$. For $h=\cotof{\rr}$ this is $T^*_0\rr\cong\rr$, recovering \cref{lem.alpha_constant}; more generally, for $h\coloneqq\cotof{G}$ with $G$ a commutative Lie group, $h[e_h]=T^*_{e_G}G=\mathfrak{g}^*$, the dual of its Lie algebra. The Hopf structure is essential here: recovering the direction at an arbitrary position $i$ from the one at $e_h$ pairs $i$ with its inverse (so that the product is $e_h$), so it is the antipode that makes the value at $e_h$ determine the whole structure.
\end{remark}

Specializing \cref{lem.alpha_constant} to $m\coloneqq\cotof{\rr}$ and $z\coloneqq\yon$, we choose the homomorphism corresponding to $r\coloneqq+1$ and denote the corresponding algebra map as
\begin{equation}\label{eqn.d_potential}
\defineTerm{potd}\colon\cotof{\rr}\to\yon,\qquad\bk{\potd}{q}(*)\coloneqq +1\in T^*_q\rr.
\end{equation}
We call it $\potd$ because this is the choice for which potentials contribute their usual differentials. Indeed, for a smooth map $U\colon X\to\rr$, the composite $\potd\circ\cotof{U}\colon\cotof{X}\to\yon$ is trivial on positions and tracing through directions,
\[
\bk{\potd\circ\cotof{U}}{q}(*)=(T_q U)^\top(+1)=dU|_q
\]
exactly as in \eqref{eqn.differential}. That is, under the identification of maps $\cotof{X}\to\yon$ with covector fields on $X$, the composite $\potd\circ\cotof{U}$ corresponds to the differential $dU$.

\begin{definition}[Smooth polynomial interpretation datum]\label{def.smooth_setup}
We define the \emph{smooth polynomial interpretation datum} for $\Sm$ (\cref{def.potlens}) to be the tuple
\[
\defineTerm{interpsm}\coloneqq(\cot,\cotof{\rr}\otimes\blank,\theta,(\yon,\potd)),
\]
an instance of \cref{def.polynomial_interpretation} whose
\begin{enumerate}
  \item \emph{interface functor} is the cotangent functor $\cot\colon(\mfd,1,\times)\to(\poly,\yon,\otimes)$, strong symmetric monoidal (\cref{prop.cot_monoidal});
  \item \emph{semantic effect} is the writer monad $\cotof{\rr}\otimes\blank$ on $\poly$, monoidal since $\cotof{\rr}$ is a commutative monoid (\cref{prop.cot_hopf});
  \item \emph{effect comparison} is the canonical $\theta\colon\cot(\rr\otimes\blank)\cong\cotof{\rr}\otimes\cot(\blank)$ induced by the strong monoidality of $\cot$ (\cref{lem.monoid_to_monad});
  \item \emph{potential algebra} is the unit $\yon$, equipped with the $(\cotof{\rr}\otimes\blank)$-monoid structure given by $\potd\colon\cotof{\rr}\to\yon$, i.e.\ the choice $r=+1$ of \eqref{eqn.d_potential}.
  \qedhere
\end{enumerate}
\end{definition}

Applying \cref{thm.poly_interpretation} to the smooth polynomial interpretation datum (\cref{def.smooth_setup}) yields a normal lax symmetric monoidal functor $\Phip{\interpsm}\colon\sarr\to\para{\cot}{\poly}$. It acts on objects by
\[\Phip{\interpsm}\lensob{M}\coloneqq\cotof{\outp{M}}\otimes\ihom{\cotof{\inpt{M}},\yon}.\] 
And $\Phip{\interpsm}$ sends a smooth adaptive arrangement $f\coloneqq(Q,\sharpR_Q,\outp f,\inpt f,U)$ as in \eqref{eqn.srw_morphism}
\[\lensob{M_1}\ltens\cdots\ltens\lensob{M_K}\To[\quad]{f}\lensob{N}\]
to a $\cotof{Q}$-parameterized polynomial map of the form
\begin{equation}\label{eqn.phi_prime_morphism}
\cotof{Q}\otimes
\Phip{\interpsm}\lensob{M_1}\otimes\cdots\otimes\Phip{\interpsm}\lensob{M_K}
\To{\Phip{\interpsm}(f)}
\Phip{\interpsm}\lensob{N}.
\end{equation}

This completes the polynomial interpretation $\Phip{\interpsm}\colon\sarr\to\para{\cot}{\poly}$. To run it as dynamics in $\pc$, we compose with an integrator $\Psisem{\intg}\colon\para{\cot}{\poly}\to\pc$, constructed in \cref{sec.integrator_semantics} and combined with $\Phip{\interpsm}$ in \cref{sec.dynamics_functor}.

\section{Integrator semantics}\label{sec.integrator_semantics}

With the polynomial interpretation in hand, the abstract account of \cref{sec.dynamics_functor_abstract} reduces dynamics to a single datum: a $p$-integrator $(\Fun S,\dyn)$ (\cref{def.integrator}), by which \cref{prop.integrator_to_pc} yields a functor $\para{p}{\poly}\to\pc$. This section supplies the two integrators that we'll use. \Cref{sec.integrators} builds the Euler step $\chi$ on which both rest; \cref{sec.config_integrator} reads off the configuration integrator; \cref{sec.cotangent_endofunctor} develops the cotangent endofunctor; \cref{sec.readouts_one_forms} introduces readouts and 1-forms; and \cref{sec.phase_integrators} assembles the phase integrator from them.

\subsection{The Euler step}\label{sec.integrators}

The polynomial interpretation $\Phip{\interpsm}$ of \cref{sec.smooth_interpretation} produces maps of the form $\cotof{Q}\otimes p\to p'$, or equivalently $\cotof{Q}\to[p,p']$. Our goal in this section is to turn them into $[p,p']$-coalgebras in a functorial way. To do so, \cref{prop.coalg_as_poly_map} says we just need a natural map $\Store(S)\to\cotof{Q}$. Such a map can be thought of as an \emph{integrator}, such as Euler integration, which takes a gradient vector and updates the state.

We specialize to $\cat Q=\rvect$ and $p=\cot$, writing $\para{\cot}{\poly}$ for the Para construction of the action $Q\cdot\blank\coloneqq\cotof{Q}\otimes\blank$ induced by $\cot$. By \cref{prop.integrator_to_pc}, every $\cot$-integrator $\intg$ induces a normal lax symmetric monoidal functor
\[\Psisem{\intg}\colon\para{\cot}{\poly}\to\pciso.\]

\begin{notation}\label{not.rvect_polynomials}
Going forward we denote the composite $\rvect\To{\inc}\mfd\To{\cot}\poly$ simply as $\cot\colon\rvect\to\poly$ and use \eqref{eqn.tangent_vec} and \eqref{eqn.canonical_dual_sum} to make the following identifications:
\[
\cotof Q=Q\yon^{Q^*}\qqand\cotof{T^*Q}=(Q\oplus Q^*)\yon^{Q^*\oplus Q}.
\]
\end{notation}

\begin{proposition}\label{prop.rvect_polynomial}
There is a monoidal natural transformation $\chi\colon\Store\circ\absval{\blank}\to\cot\circ\inc$ of functors $\rvect\to\poly$ whose component $\defineTerm{chiQ}_{\rv Q}\colon\Store(\absval{Q})\to\cotof{Q}$ at $\rv Q=(Q,\sharpR_Q)$ is the polynomial map given on positions by the identity and on directions at $q\in Q$ by
\begin{equation}\label{eqn.directions_sharp}
T^*_qQ\To{\cong}Q^*\To{\sharpR_q}Q\To{q+\blank}Q,
\qquad
\xi\longmapsto q+\sharpR_q(\xi).
\end{equation}
\end{proposition}

\begin{proof}
On positions, both composite functors $\Store\circ\absval{\blank}$ and $\cot\circ\inc$ send $Q$ to its underlying set, so we can take the on-positions component of $\chiQ{\rv Q}$ to be the identity. At each $q\in Q$, \eqref{eqn.directions_sharp} is the composite of an isomorphism \eqref{eqn.tangent_vec}, the linear map $\sharpR_q$, and translation by $q$---the flat \emph{exponential} step $q\mapsto q+\sharpR_q(\xi)$, which on linear $Q$ coincides with a single Euler step---giving a well-defined polynomial map $\Store(\absval{Q})\to\cotof{Q}$.

For naturality at a $\rvect$-isomorphism $\phi\colon\rv Q\to\rv W$, the on-positions square commutes by definition. On directions at $q:Q$, we must check that the two backward composites $T^*_{\phi(q)}W\to Q$ agree:
\begin{align*}
\xi\;&\Mapsto{\phi^*}\;\phi^*\xi\;\Mapsto{\sharpR_q}\;\sharpR_q(\phi^*\xi)\;\Mapsto{q+\blank}\;q+\sharpR_q(\phi^*\xi), && \text{[$\cotof{\phi}\circ\chiQ{\rv Q}$]}\\
\xi\;&\Mapsto{\sharpR_{\phi(q)}}\;\sharpR_{\phi(q)}(\xi)\;\Mapsto{\phi(q)+\blank}\;\phi(q)+\sharpR_{\phi(q)}(\xi)\;\Mapsto{\phi^{-1}}\;q+\phi^{-1}(\sharpR_{\phi(q)}\xi) && \text{[$\chiQ{\rv W}\circ\Store(\phi)$]}.
\end{align*}
The two are equal iff $\sharpR_q\circ\phi^*=\phi^{-1}\circ\sharpR_{\phi(q)}$, which is the sharp-equivariance condition \eqref{eqn.pairing_triangle} at $q$. Monoidality follows from the direct-sum convention $\sharpR_{Q_1\oplus Q_2}=\sharpR_{Q_1}\oplus\sharpR_{Q_2}$ (\cref{prop.rvect_monoidal}).\qedhere
\end{proof}

\begin{remark}[Generalization to manifolds]\label{rmk.rmfdc_generalization}
The dynamics functor (\cref{thm.dynamics_functor}) depends on $\rvect$ only through the monoidal natural transformation $\chi$ of \cref{prop.rvect_polynomial}. Any symmetric monoidal category of manifolds admitting such a $\chi\colon\Store\circ\absval{\blank}\to\cot\circ\inc$ would serve equally well.
One possible source would be
\emph{responsive manifolds with connection}: triples $(M,\nabla_M,\sharpR_M)$ of a smooth manifold, a geodesically complete linear connection, and a smooth bundle map $\sharpR_M\colon T^*M\to TM$. The geodesic exponential $\exp^{\nabla_M}_m(\sharpR_m(\xi))$
would replace
the Euler step \eqref{eqn.directions_sharp}, and on flat $Q$ recovers it exactly.
We expect examples to include
complete Riemannian manifolds via their Levi-Civita connections, and Lie groups once equipped with a complete left-invariant Riemannian metric. We stay with responsive vector spaces for concreteness.
\end{remark}

\subsection{The configuration integrator}\label{sec.config_integrator}

The \emph{configuration integrator} $\conf\coloneqq(\absval{\blank},(-1)\circ\chi)$ has state space $Q\mapsto\absval{Q}$ and \emph{descends}: it feeds the negative gradient into the Euler step $\chi$ of \cref{prop.rvect_polynomial}. Its dynamics is the composite
\[
\Store\circ\absval{\blank}\To{\chi}\cot\To{-1}\cot,
\]
where $-1\colon\cot\Rightarrow\cot$ negates the incoming covector and is the identity on positions, so on directions at $q\in Q$ it is
\begin{equation}\label{eqn.conf_integrator}
\Store(\absval Q)\to\cotof Q,\qquad \xi\longmapsto\chiQ{\rv Q}(-\xi)=q-\sharpR_q(\xi).
\end{equation}
The minus is the integrator's: an incoming covector $\xi\in Q^*$ steps the state \emph{against} $\sharpR_q(\xi)$, by $q\mapsto q-\sharpR_q(\xi)$, which is what makes the configuration dynamics descend rather than ascend. It rides on the gradient, not on the response---$\sharpR_Q$ is used exactly as the syntax supplies it, so the one responsive vector space $\rv Q$ drives this integrator and the phase integrator below alike. Both legs are monoidal natural transformations---$\chi$ by \cref{prop.rvect_polynomial} and $-1$ trivially---so the composite is one as well, hence a $\cot$-integrator.

\subsection{The cotangent endofunctor}\label{sec.cotangent_endofunctor}

We now turn to a physics-based semantics. We consider integrators of the form
\[
\begin{tikzcd}[column sep=60pt]
	\rvect\ar[d, equal]\ar[r, "\absval{T^*\blank}", ""' name=top]&
	\smsetiso\ar[d, "\Store"]\\
	\rvect\ar[r, "\cot"', ""{name=bot}]&
	\poly
	\arrow[from=top, to=bot, Rightarrow, shorten=3pt, "\dyn"]
\end{tikzcd}
\]
On objects, the cotangent endofunctor sends a responsive vector space $(Q,\sharpR_Q):\rvect$ to its cotangent bundle $T^*Q\cong Q\oplus Q^*$ \eqref{eqn.TT_def}, equipped with the canonical symplectic pairing of \cref{prop.canonical_symplectic_pairing}. Thus for these integrators, the stored state is a point $(q,\xi)$ in the phase space $T^*Q$, with \emph{position} $q$ and \emph{momentum} $\xi$. Our first job is to show (\cref{prop.TT_endofunctors,prop.TT_monoidal}) that this makes sense, i.e.\ that $T^*\colon\rvect\to\rvect$, and hence $\absval{T^*\blank}$, is a strong monoidal functor. We then build the phase integrator $\phase$ from the canonical symplectic structure on $T^*Q$.

The cotangent bundle is naturally a responsive vector space via the canonical symplectic pairing, whose sharp map is \eqref{eqn.canonical_sharp}.
For a $\rvect$-morphism $\phi\colon Q\To{\cong} W$, define
\begin{equation}\label{eqn.TT_phi}
T^*\phi\coloneqq\phi\oplus(\phi^*)\inv.
\end{equation}

\begin{proposition}\label{prop.TT_endofunctors}
Equipping $T^*Q$ with the canonical symplectic pairing of \cref{prop.canonical_symplectic_pairing}, the assignments \eqref{eqn.TT_def}, \eqref{eqn.TT_phi} define an endofunctor $T^*\colon\rvect\to\rvect$.
\end{proposition}

\begin{proof}
To define an endofunctor $T^*\colon\rvect\to\rvect$ we must check that $T^*\phi$ is a morphism of $\rvect$---a linear isomorphism that is moreover sharp-equivariant---and that $T^*$ preserves identities and composites. That $T^*\phi=\phi\oplus(\phi^*)\inv$ is a linear isomorphism is immediate, and preservation of identities and composites is inherited from $\vect$; the substance is sharp-equivariance, which we now verify. 

The defining property of $\phi^*$ states that $(\phi^*\xi')(q) = \xi'(\phi q)$ for all $\xi'\in W^*\text{ and }q\in Q.$ Setting $\xi'\coloneqq\psi(\xi)$, where $\psi\coloneqq(\phi^*)\inv\colon Q^*\to W^*$, and using $\phi^*\circ\psi=\id_{Q^*}$, we note that
\begin{equation}\label{eqn.dualeval}
(\psi\xi)(\phi q) = (\phi^*(\psi\xi))(q) = \xi(q).
\end{equation}
To verify sharp-equivariance \eqref{eqn.pairing_triangle}, we equivalently check pairing-preservation (both pairings are nondegenerate). Fix any basepoint $(q_0,\xi_0)\in T^*Q$ and calculate
\begin{align*}
\pairing{T^*\phi(q_1,\xi_1),T^*\phi(q_2,\xi_2)}_{T^*\phi(q_0,\xi_0)}
&= (\psi\xi_2)(\phi q_1)-(\psi\xi_1)(\phi q_2) && \text{[\eqref{eqn.TT_phi}, \eqref{eqn.canonical_form}]}\\
&= \xi_2(q_1)-\xi_1(q_2) && \text{[\eqref{eqn.dualeval}]}\\
&= \pairing{(q_1,\xi_1),(q_2,\xi_2)}_{(q_0,\xi_0)} && \text{[\eqref{eqn.canonical_form}]};
\end{align*}
both sides are independent of $(q_0,\xi_0)$ by constancy of the canonical symplectic pairing.
\end{proof}

\begin{proposition}\label{prop.TT_monoidal}
The endofunctor $T^*\colon\rvect\to\rvect$ is strong symmetric monoidal with respect to $(0,\oplus)$.
\end{proposition}

\begin{proof}
The unitor is the identity since $T^*0=0\oplus 0^*=0$. The productor is the swap
\begin{equation*}
(T^*)_2\colon T^*Q_1\oplus T^*Q_2\To{\cong}T^*(Q_1\oplus Q_2),\quad ((q_1,\xi_1),(q_2,\xi_2))\mapsto((q_1,q_2),(\xi_1,\xi_2)).
\end{equation*}
For sharp-equivariance \eqref{eqn.pairing_triangle}, the canonical symplectic sharp on $T^*(Q_1\oplus Q_2)$ acts by the map $(\xi_1,\xi_2,q_1,q_2)\mapsto(q_1,q_2,-\xi_1,-\xi_2)$ by \eqref{eqn.canonical_sharp}. Conjugating by $(T^*)_2$ gives the map $(\xi_1,q_1,\xi_2,q_2)\mapsto(q_1,-\xi_1,q_2,-\xi_2)$, which is the direct-sum sharp $\sharpS_{T^*Q_1}\oplus\sharpS_{T^*Q_2}$ from \eqref{eqn.directsum_sharp}.\qedhere
\end{proof}

\begin{lemma}\label{lem.TT_projection_diagonal_monoidal}
The bundle projections and diagonals
\[
\pi_Q\colon T^*Q\to Q,
\qquad
\Delta_{T^*Q}\colon T^*Q\to T^*Q\times T^*Q,
\]
are components of monoidal natural transformations as follows between functors $\rvect\to\mfd$:
\[
\pi\colon\inc\circ T^*\Rightarrow\inc,\qquad
\Delta_{T^*}\colon\inc\circ T^*\Rightarrow(\inc\circ T^*)\times(\inc\circ T^*).
\]
\end{lemma}

\begin{proof}
For $\pi$, naturality in $\phi\colon Q\To{\cong} W$ is the identity $\pi_W\circ T^*\phi=\phi\circ\pi_Q$, and monoidality is the identity $\pi_{Q\oplus W}=\pi_Q\times\pi_W$ under the product comparisons $T^*(Q\oplus W)\cong T^*Q\times T^*W$ and $Q\oplus W\cong Q\times W$. Diagonals are natural, and $\Delta_{T^*(Q\oplus W)}=\Delta_{T^*Q}\times\Delta_{T^*W}$ under $T^*(Q\oplus W)\cong T^*Q\times T^*W$.\qedhere
\end{proof}

Whiskering $\pi$ by $T^*$ we obtain a monoidal natural transformation
\begin{equation}\label{eqn.pi_whiskered}
\pi_{T^*}\colon\inc\circ T^*\circ T^*\Rightarrow\inc\circ T^*,\qquad (\pi_{T^*})_{\rv Q}=\pi_{T^*Q}\colon T^*(T^*Q)\to T^*Q.
\end{equation}

\subsection{Readouts and one-forms}\label{sec.readouts_one_forms}

Two kinds of data on the phase space will combine into integrators: \emph{readouts}, which set the position the forward pass reports, and \emph{1-forms}, which feed the backward pass. We discuss them in turn.

\begin{definition}\label{def.monoidal_readout}
A \emph{monoidal readout on $T^*$} is a monoidal natural transformation 
\[\rho\colon\inc\circ T^*\Rightarrow\inc\]
between functors $\rvect\to\mfd$.
\end{definition}

In other words, a monoidal readout is a family of smooth maps $\rho_{\rv Q}\colon T^*Q\to Q$, natural in $\rvect$-isomorphisms and monoidal over direct sums, $\rho_{\rv Q\oplus\rv W}=\rho_{\rv Q}\oplus\rho_{\rv W}$.

\begin{proposition}\label{prop.readouts_vector_space}
Under pointwise addition and scalar multiplication, the monoidal readouts on $T^*$ form a real vector space.
\end{proposition}

\begin{proof}
Given monoidal readouts $\rho,\rho'$ and a scalar $c\in\rr$ we may set $(\rho+\rho')_{\rv Q}(q,\xi)\coloneqq\rho_{\rv Q}(q,\xi)+\rho'_{\rv Q}(q,\xi)$ and $(c\,\rho)_{\rv Q}(q,\xi)\coloneqq c\,\rho_{\rv Q}(q,\xi)$, the operations taken in the real vector space $Q$. These are again smooth, natural in $\rvect$-isomorphisms (which act linearly), and monoidal, e.g.\
\[
(\rho+\rho')_{\rv Q\oplus\rv W}=\rho_{\rv Q\oplus\rv W}+\rho'_{\rv Q\oplus\rv W}=(\rho_{\rv Q}\oplus\rho_{\rv W})+(\rho'_{\rv Q}\oplus\rho'_{\rv W})=(\rho+\rho')_{\rv Q}\oplus(\rho+\rho')_{\rv W}.
\qedhere\]
\end{proof}

\begin{lemma}\label{lem.readout_examples}
For any $\alpha\in\rr$, the \emph{time-$\alpha$ exponential} $\exp^\alpha_{\rv Q}(q,\xi)\coloneqq q+\alpha\,\sharpR_q(\xi)$ is a monoidal readout. The projection $\pi=\exp^0$ from \cref{lem.TT_projection_diagonal_monoidal} and the \emph{exponential} $\exp\coloneqq\exp^1=q+\sharpR_q(\xi)$ are special cases.
\end{lemma}

\begin{proof}
We check the three conditions of \cref{def.monoidal_readout} for $\exp^\alpha_{\rv Q}(q,\xi)=q+\alpha\,\sharpR_q(\xi)$. Smoothness is clear, since $\sharpR$ is a smooth bundle map. For naturality at an $\rvect$-isomorphism $\phi\colon\rv Q\to\rv W$, we compute
\[
\phi\big(q+\alpha\,\sharpR^Q_q(\xi)\big)
=\phi(q)+\alpha\,\phi\big(\sharpR^Q_q(\xi)\big)
=\phi(q)+\alpha\,\sharpR^W_{\phi q}\big((\phi^*)\inv\xi\big)
=\exp^\alpha_{\rv W}\big(T^*\phi(q,\xi)\big),
\]
the first equality by linearity of $\phi$, the second by sharp-equivariance \eqref{eqn.pairing_triangle}, and the third by the action of $T^*\phi=\phi\oplus(\phi^*)\inv$ on $(q,\xi)$. For $\oplus$-monoidality, $\sharpR_{Q\oplus W}=\sharpR_Q\oplus\sharpR_W$ (\cref{prop.rvect_monoidal}) and the affine map $q\mapsto q+\alpha\,\sharpR_q(\xi)$ is applied coordinatewise, so $\exp^\alpha_{\rv Q\oplus\rv W}=\exp^\alpha_{\rv Q}\oplus\exp^\alpha_{\rv W}$.
\end{proof}

\begin{lemma}\label{lem.readout_to_poly}
Every monoidal readout $\rho$ on $T^*$ induces a monoidal natural transformation
\[
\rho\colon \cot\circ T^*\Rightarrow\cot ,
\]
with on-positions $\rho_{\rv Q}$ and on-directions the injection $\xi'\mapsto(\xi',0)$.
\end{lemma}

\begin{proof}
At each $\rv Q$, the on-positions $\rho_{\rv Q}\colon T^*Q\to Q$ and the backward maps $\xi'\mapsto(\xi',0)$, from $Q^*$ to $Q^*\oplus Q\cong T^*_{(q,\xi)}(T^*Q)$, constitute a polynomial map $\cotof{T^*Q}\to\cotof{Q}$.

For naturality at an $\rvect$-isomorphism $\phi\colon Q\to W$, compare $\cotof{\phi}\circ\rho_{\rv Q}$ with $\rho_{\rv W}\circ\cotof{T^*\phi}$. On positions these are $\phi\circ\rho_{\rv Q}$ and $\rho_{\rv W}\circ T^*\phi$, equal because $\rho$ is natural (\cref{def.monoidal_readout}). On directions at $(q,\xi)$, a covector $\xi'\in W^*$ pulls back along both maps to $(\phi^*\xi',0)$. The $\oplus$-monoidality is straightforward.
\end{proof}

Whereas a readout is about the forward pass, the following notion is about the backward pass.

\begin{definition}\label{def.monoidal_one_form}
A \emph{monoidal 1-form on $T^*$} is a monoidal natural transformation
\[\omega\colon\inc\circ T^*\Rightarrow\inc\circ T^*\circ T^*\] 
that is a section of the projection $\pi_{T^*Q}\colon T^*(T^*Q)\to T^*Q$, from \eqref{eqn.pi_whiskered}.
\end{definition}

In other words, a monoidal 1-form on $T^*$ consists of a smooth family of sections $\omega_{\rv Q}\colon T^*Q\to T^*(T^*Q)$, natural in $\rvect$-isomorphisms, that is monoidal over direct sums: $\omega_{\rv Q\oplus\rv W}=\omega_{\rv Q}\oplus\omega_{\rv W}$.
\begin{proposition}\label{prop.one_forms_vector_space}
Under fiberwise addition and scalar multiplication, the monoidal 1-forms on $T^*$ form a real vector space.
\end{proposition}

\begin{proof}
The pointwise argument of \cref{prop.readouts_vector_space} applies verbatim, with the operations taken in each fiber $T^*_{(q,\xi)}(T^*Q)\cong Q^*\oplus Q$ in place of $Q$.
\end{proof}

\begin{example}[Kinetic 1-form]\label{ex.kinetic_one_form}
The \emph{kinetic 1-form} of $(Q,\sharpR_Q)$ is the section
\begin{equation}\label{eqn.kinetic_one_form}
\beta_{\rv Q}\colon T^*Q\to T^*(T^*Q),\qquad
(q,\xi)\mapsto(0,\sharpR_q(\xi)),
\end{equation}
using the identification $Q^*\oplus Q\cong T^*_{(q,\xi)}(T^*Q)$ from \eqref{eqn.T_products} and \eqref{eqn.tangent_vec}.%
\footnote{We borrow the name from physics: when $\sharpR$ is constant and symmetric, $\beta_{\rv Q}$ is the differential of the kinetic energy $\tfrac12\xi(\sharpR_0(\xi))$.}
It is a monoidal 1-form on $T^*$: naturality follows from sharp-equivariance \eqref{eqn.pairing_triangle}, and monoidality holds because $\sharpR_{Q\oplus W}=\sharpR_Q\oplus\sharpR_W$ (\cref{prop.rvect_monoidal}).
\end{example}

More generally, when a 1-form $\omega_{\rv Q}(q,\xi)=(a,b)$ is fed through the symplectic sharp $\sharpS_{T^*Q}$ and added to the state---as it will be under the phase integrator \eqref{eqn.phase_integrator} below---it contributes $(+b,-a)$ to the update, since $\sharpS_{T^*Q}$ swaps coordinates and negates the second \eqref{eqn.canonical_sharp}. Thus the $Q$-component of $\omega$ drives position and the $Q^*$-component drives momentum with a sign flip: the kinetic 1-form ends up adding velocity $\sharpR_q(\xi)$ to position.

  \begin{example}[Damping 1-form]\label{ex.damping_one_form}
 Although it plays no further role, it is worth recording where damping enters the formalism. Define the \emph{damping 1-form} to be the section
 \begin{equation}\label{eqn.damping}
 \zeta_{\rv Q}\colon T^*Q\to T^*(T^*Q),\qquad
 (q,\xi)\mapsto(\xi,\,0):
 Q^*\oplus Q\cong T^*_{(q,\xi)}(T^*Q),
 \end{equation}
 which is a monoidal 1-form on $T^*$ by linearity and the $\oplus$-monoidality of the momentum coordinate. For a \emph{damping coefficient} $c\in[0,1]$, the scalar multiple $c\,\zeta$ is again a monoidal 1-form (\cref{prop.one_forms_vector_space}); it removes a fraction $c$ of the momentum at each step.
 \end{example}

\begin{lemma}\label{lem.one_form_to_poly}
Every monoidal 1-form $\omega$ on $T^*$ induces a monoidal natural transformation
\[
\omega\colon \cot\circ T^*\Rightarrow\yon .
\]
\end{lemma}

\begin{proof}
A section $\omega\colon M\to T^*M$ defines a polynomial map $\cotof{M}\to\yon$ whose position map is unique and whose backward map at $m$ is the function $1\to T^*_mM$ selecting $\omega(m)$. Applying this at $M=T^*Q$ gives the components $\omega_{\rv Q}\colon\cotof{T^*Q}\to\yon$. Naturality and monoidality are exactly the naturality and monoidality over $\oplus$ in \cref{def.monoidal_one_form}.%
\footnote{A monoidal natural transformation $\omega\colon\cot\circ T^*\to\yon$ may not be a monoidal 1-form because the former does not enforce smoothness, whereas the latter does.}
\qedhere
\end{proof}

We use the same symbol to denote both a monoidal 1-form and the polynomial map it induces. By \cref{def.monoidal_one_form} a 1-form $\omega$ is a section $\inc\circ T^*\Rightarrow\inc\circ T^*\circ T^*$, and by \cref{lem.one_form_to_poly} it also induces a map $\omega\colon\cot\circ T^*\Rightarrow\yon$; both are determined by the same datum---the cotangent vector $\omega_{\rv Q}(q,\xi)\in T^*_{(q,\xi)}(T^*Q)\cong Q^*\oplus Q$ selected at each phase point---so no confusion should arise. For the kinetic 1-form this datum is $\beta_{\rv Q}(q,\xi)=(0,\sharpR_q(\xi))$.

The following lemma combines monoidal readouts and 1-forms into a package we will use shortly.

\begin{lemma}\label{lem.combined_update}
A monoidal readout $\rho$ and a monoidal 1-form $\omega$ combine into a monoidal natural transformation $\nu_{\rho,\omega}\colon\cot\circ T^*\Rightarrow\cot$, given by the composite
\[
\begin{tikzcd}[column sep=2.5em]
\cot\circ T^*
  \ar[rrr, bend left=12, "\nu_{\rho,\omega}"]
  \ar[r, "\cotof{\Delta_{T^*\blank}}"']
& (\cot\circ T^*)\otimes(\cot\circ T^*)
  \ar[r, "\rho\otimes\omega"']
& \cot\otimes\yon
  \ar[r, "\cong"']
& \cot.
\end{tikzcd}
\]
On positions $\nu_{\rho,\omega}$ is $\rho_{\rv Q}$; on directions at $(q,\xi)$ it sends $\xi'\mapsto(\xi',0)+\omega_{\rv Q}(q,\xi)$.
\end{lemma}

\begin{proof}
Each arrow is a monoidal natural transformation: $\cotof{\Delta_{T^*\blank}}$ by \cref{lem.TT_projection_diagonal_monoidal,prop.cot_monoidal}, the readout $\rho$ by \cref{lem.readout_to_poly}, the 1-form $\omega$ by \cref{lem.one_form_to_poly}, and the unitor $\cot\otimes\yon\cong\cot$ because $\yon$ is the monoidal unit. On positions, $\cotof{\Delta_{T^*\blank}}$ is the diagonal and $\rho\otimes\omega$ applies $\rho_{\rv Q}$ in the first factor while $\omega$ lands in $\yon$, so the composite reads out $\rho_{\rv Q}$. On directions at $(q,\xi)$, a covector $\xi'\in Q^*$ pulls back through $\rho\otimes\omega$ to $(\xi',0)$ and $\omega_{\rv Q}(q,\xi)$ in the two factors, which the diagonal adds.
\end{proof}

\subsection{The phase integrator}\label{sec.phase_integrators}

The \emph{phase integrator} combines two of the transformations built above: the Euler step $\chiQ{T^*\blank}$ on the cotangent bundle (\cref{prop.rvect_polynomial}) and $\defineTerm{nu}\coloneqq\nu_{\exp,\beta}\colon\cot\circ T^*\Rightarrow\cot$. It is the instance of \cref{lem.combined_update} assembled from the exponential readout $\exp$ (\cref{lem.readout_examples}) and the kinetic 1-form $\beta$ (\cref{ex.kinetic_one_form}). On positions $\nuro$ reports the \emph{presented position}
\begin{equation}\label{eqn.presented_position}
\tilde q\coloneqq q+\sharpR_q(\xi).
\end{equation}
On directions it pairs the incoming gradient $\xi'$ with the stored velocity $\sharpR_q\xi$,
\begin{equation}\label{eqn.nu}
\bk{\nu}{(q,\xi)}\colon\xi'\mapsto(\xi',\sharpR_q\xi)\in Q^*\oplus Q\cong T^*_{(q,\xi)}(T^*Q).
\end{equation}

As promised, define $\dynphase$ to be the composite of this readout with the Euler step,
\[
\defineTerm{dynphase}\colon\quad\Store\circ\absval{T^*\blank}\To{\chiQ{T^*\blank}}\cot\circ T^*\To{\nuro}\cot.
\]
Since $\absval{T^*\blank}$ is a composite of strong monoidal functors (\cref{prop.TT_monoidal}), and since $\chiQ{T^*\blank}$ and $\nuro$ are monoidal-natural (\cref{prop.rvect_polynomial,lem.combined_update}) we define the phase integrator to be the pair
\begin{equation}\label{eqn.phase_integrator}
\phase\coloneqq(\absval{T^*\blank},\dynphase).
\end{equation}

The two factors play distinct roles. The map $\nuro$ assembles the covector $(\xi',\sharpR_q\xi)$ \eqref{eqn.nu}, and this is the only place the response $\sharpR_Q$ enters. The Euler step $\chiQ{T^*\blank}$ \eqref{eqn.directions_sharp} then applies the canonical symplectic sharp on $T^*Q$ \eqref{eqn.canonical_sharp}, the same for every responsive vector space: on directions it adds to the current state $(q,\xi)$ the sharp of that covector,
\[
\sharpS_{T^*Q}\bigl(\xi',\sharpR_q(\xi)\bigr)=\bigl(\sharpR_q(\xi),\,-\xi'\bigr),
\]
so the update is
\begin{equation}\label{eqn.phase_update}
\absval{T^*Q}\times Q^*\to\absval{T^*Q},\qquad ((q,\xi),\xi')\mapsto(q+\sharpR_q(\xi),\,\xi-\xi').
\end{equation}
This has one step of Hamilton's equations as a special case%
: the position moves by the velocity $\sharpR_q\xi$, landing on $\tilde q$, and the momentum moves by $-\xi'$, the incoming gradient at $\tilde q$. See \cref{sec.closed_conservation}.

\section{The dynamics functors}\label{sec.dynamics_functor}

The smooth adaptive arrangements of \cref{sec.potentialized_lenses} become smooth adaptive dynamics by composing the polynomial interpretation $\Phip{\interpsm}$ (\cref{sec.smooth_interpretation}) with an integrator $\intg$ (\cref{sec.integrator_semantics}); we unpack this composite for the smooth instance. The interpretation $\Phip{\interpsm}$ already carries most of the construction: the resulting functor $\Phi_\intg$ agrees with it on objects,
\[\Phi_\intg\lensob{M}=\Phip{\interpsm}\lensob{M},\]
and \cref{thm.dynamics_functor} supplies only what the integrator contributes. It first rearranges \eqref{eqn.phi_prime_morphism} to
\[
\cotof{Q}\to
\left[
  \Phi_\intg\lensob{M_1}\otimes\cdots\otimes\Phi_\intg\lensob{M_K},
  \Phi_\intg\lensob{N}
\right],
\]
then use the integrator to replace the $\cotof{Q}$ with something of the form $S\yon^S$ for $S=F(Q)$, and finally, to unfold that to a coalgebra via \cref{prop.coalg_as_poly_map}
\[
S\to\left[
  \Phi_\intg\lensob{M_1}\otimes\cdots\otimes\Phi_\intg\lensob{M_K},
  \Phi_\intg\lensob{N}
\right](S).
\]
The integrator is what fixes the relationship between $S$ and $Q$: in our two main cases of interest $S$ is the configuration space $\absval{Q}$ or the phase space $\absval{T^*Q}$; see \eqref{eqn.our_phis}.

\begin{corollary}\label{cor.functor}
Every $\cot$-integrator $\intg$ induces an operad functor
\[
\Phi_\intg\colon\sarr\to\pc
\]
with $\Phi_\intg\lensunit\cong\yon$; hence it carries closed systems to $\yon$-coalgebras.
\end{corollary}

\begin{proof}
This is \cref{thm.dynamics_functor} for the smooth adaptive-arrangement datum $\Sm$ (\cref{def.potlens}) and the smooth polynomial interpretation datum $\interpsm$ (\cref{def.smooth_setup}), whose potential algebra has carrier $z=\yon$.
\qedhere
\end{proof}

Specializing $\Phi_\intg$ to the configuration integrator $\conf$ \eqref{eqn.conf_integrator} and the phase integrator $\phase$ \eqref{eqn.phase_integrator} gives our two main functors of interest, the \emph{configuration} and \emph{phase dynamics functors}
\begin{equation}\label{eqn.our_phis}
\defineTerm{Phiconf},\ \defineTerm{Phiphase}\colon\sarr\to\pc.
\end{equation}

\begin{remark}[Varying the parameterizing category]\label{rmk.adam}

Given an adaptive-arrangement datum $D=(\cat M,\cat Q,J,\Fun R)$, composing $J$ with any strong monoidal functor $\iota\colon\cat{Q}'\to\cat{Q}$ induces a new adaptive-arrangement datum $D'$. This operation doesn't affect a polynomial interpretation datum, and any $(\Fun c\circ J)$-integrator restricts to a $(\Fun c\circ J\circ\iota)$-integrator, so \cref{thm.dynamics_functor} goes through to give a functor $\arr_{D'}\to\pc$. However, there may be more integrators for $J\circ\iota$ than there were for $J$. In particular, removing morphisms from $\cat{Q}$ allows for more integrators because the naturality condition is weakened.

In $\rvect$, the sharp-equivariant isomorphisms between Euclidean parameter spaces (\cref{ex.euclidean_sharp}) are exactly the orthogonal maps. But several optimizers used in practice do not require naturality with respect to orthogonal maps. Any method that acts coordinatewise---e.g.\ an elementwise square or a per-coordinate step
size---singles out a basis and admits no natural transformation over $\rvect$. Such a method is not outside the framework; it just lives over a smaller $\cat Q$. For example, we could restrict our attention to vector spaces equipped with a Euclidean sharp and take as morphisms any of the following nested classes of isometries:%
\footnote{With the Euclidean response structure fixed, this is as far as the restriction can go while keeping $\cat Q$ symmetric monoidal: the braidings are themselves coordinate permutations, so dropping them too would leave $\cat Q$ without its symmetry.}
\[
  \text{orthogonal}\;\supset\;\text{signed coordinate permutations}\;\supset\;\text{coordinate permutations}.
\]
We conjecture that \textsc{Adam}~\cite{kingma2014adam} is an integrator when $\cat{Q}$ is restricted to the subcategory whose objects are Euclidean vector spaces $\rr^n$ equipped with Euclidean sharp and whose morphisms are signed permutations. That is, that \textsc{Adam}'s behavior does not depend on the order of the coordinate axes, nor on the direction of each axis. 

The monoidal requirement of integrators cannot be negotiated like the naturality can. The monoidality of $\intg$ corresponds to the compositionality of the semantics, that each box can compute its own update. This suggests seeking other optimizers as integrators over appropriately chosen parameter categories.
\end{remark}

\section{Configuration dynamics \texorpdfstring{$\termraw{Phiconf}$}{Phi conf}}\label{sec.configuration_dynamics}

The dynamics functor $\Phiconf\colon\sarr\to\pc$ is the simpler of the two. A state is just an element $q\in Q$ of the parameterizing manifold---there is no additional ``momentum'' component---and each step updates it by a descent step on $Q$. This is the functor we apply to recover Newton's method (\cref{sec.newton_warmup}) and gradient descent (\cref{sec.dl_warmup}).

Recall from \cref{thm.dynamics_functor} that $\Phiconf=\Psisem{\conf}\circ\Phip{\interpsm}$. The functor $\Psisem{\conf}$ replaces a $\cotof{Q}$-parameter with $\Store(S)$, where $S\coloneqq\absval{Q}$:
\[
\Store(S)\To{(-1)\circ\chiQ{\rv Q}}\cotof{Q}.
\]
A state consists of a position $q\in Q$, and an incoming covector $\xi'\in T^*_qQ$ updates that state to $q-\sharpR_q(\xi')$ by \eqref{eqn.conf_integrator}.

However, the majority of what takes place under $\Phiconf$ is still compressed within the $\Phip{\interpsm}$ factor from \cref{def.smooth_setup,thm.poly_interpretation}. So we now unpack $\Phiconf$ in total, tracing each term and sign in the resulting formulas back to the definitions that produced it.

Let $\inpt M,\outp M:\mfd$ be manifolds, and write
\[\Omega(M)\coloneqq\poly(\cotof M,\yon)\]
for the set of \emph{covector fields} on $M$: a polynomial map $\omega\colon\cotof M\to\yon$ is exactly a choice of covector $\omega(m)\coloneqq\bk{\omega}{m}(*)\in T^*_mM$ at each position $m\in M$, i.e.\ a (not-necessarily continuous) section of the cotangent projection $\pi\colon T^*M\to M$.
A smooth map $f\colon M\to N$ induces a \emph{pullback} $f^*\coloneqq\poly(\cotof f,\yon)\colon\Omega(N)\to\Omega(M)$, precomposition with $\cotof f$, which sends $\omega$ to $(f^*\omega)(m)=(T_m f)^\top\omega(f(m))$.

\paragraph{The action of $\termraw{Phiconf}$ on objects.}
The functor $\Phiconf$ acts on objects as follows:
\begin{equation}\label{eqn.Phi_on_obs}
\Phiconf\lensob M=
\cotof{\outp M}\otimes\ihom{\cotof{\inpt M},\yon}\cong
\sum_{(\outp m,\omega):\outp M\times\Omega(\inpt M)}\yon^{\left(T^*_{\outp{m}}\outp M\right)\times\inpt M}.
\end{equation}
A position in $\Phiconf\lensob M$ is a pair $(\outp m,\omega)$ where $\outp m\in\outp{M}$ is a point and $\omega\in\Omega(\inpt M)$ is a covector field. A direction at $(\outp m,\omega)$ is a pair $(\xi,\inpt m)$ where $\xi:T^*_{\outp{m}}\outp M$ is a covector and $\inpt m\in\inpt M$ is a point.

So at every time step, a box of the form
\boxpic{$\inpt{M}$}{$\outp{M}$}
outputs both a point $\outp m$ of $\outp M$ and a covector field $\omega$ on $\inpt M$.
It then waits for input consisting of a covector at $\outp m$ and a point of $\inpt M$.

\paragraph{The action of $\termraw{Phiconf}$ on morphisms.}
Describing the action of $\Phiconf$ on morphisms will extend to the end of this section, \cref{sec.configuration_dynamics}. Let $(Q,\sharpR_Q):\rvect$ and let
$f\colon Q\cdot\lensob M\to\lensob N$
be an arrangement in $\sarr$, i.e.\ a tuple $(\outp f,\inpt f,U)$; specializing \eqref{eqn.srw_morphism} to $K=1$:
\begin{equation}\label{eqn.para_potential_lens_maps}
\outp f\colon Q\times\outp M\to\outp N,\qquad
\inpt f\colon Q\times\outp M\times\inpt N\to\inpt M,\qquad
U\colon Q\times\outp M\times\inpt N\to\rr.
\end{equation}
It is sent by $\Phiconf$ to a $[\Phiconf\lensob M,\Phiconf\lensob N]$-coalgebra whose carrier is the underlying set of $Q$,
\begin{equation}\label{eqn.state_space_conf}
S\coloneqq\absval{Q};
\end{equation}
we denote its points by $q\in S=|Q|$. The state update below uses the (possibly basepoint-varying) sharp $\sharpR_q$ on $Q$ from \cref{def.rvect}.

We now unpack the coalgebra function itself,
\[
S\To{\Phiconf(\rv Q,f)}
\left[\Phiconf\lensob M,\Phiconf\lensob N\right](S).
\]
Given a state $q\in S$ and a position $a\coloneqq(\outp m,\omega_M):\outp M\times\Omega(\inpt M)$, the readout position $b=(\outp n,\omega_N)$ is the on-positions component, at parameter $q$, of the polynomial map $\Phip{\interpsm}(f)$ \eqref{eqn.phi_prime_morphism}. Unwinding \cref{thm.poly_interpretation}, this map applies $\Theta_\potalg$ \eqref{eqn.theta_morphism} to the lens $\binom{\inpt f}{\outp f}$, so its $\outp N$-coordinate is carried by the forward map $\outp f$,
\begin{equation}\label{eqn.outpn}
\outp n\coloneqq\outp f(q,\outp m),
\end{equation}
and its covector-field coordinate $\omega_N\colon\inpt N\to T^*\inpt N$ is carried by the map $\psi_f$ \eqref{eqn.psi_f}, whose backward map $\inpt f$ contributes a pullback of $\omega_M$ and whose potential algebra $\potd$ contributes the differential of $U$:
\begin{equation}\label{eqn.omega_n}
\omega_N(\inpt n)\coloneqq(\inpt f(q,\outp m,\blank))^*\omega_M\big|_{\inpt n}\;+\;d(U(q,\outp m,\blank))\big|_{\inpt n}.
\end{equation}
The first term is the pullback of $\omega_M$ along the partial map $\inpt f(q,\outp m,\blank)\colon\inpt N\to\inpt M$; the second is the differential of the potential $U(q,\outp m,\blank)\colon\inpt N\to\rr$ at $\inpt n$, in the sense of \eqref{eqn.differential}.%
\footnote{The $+1$ that turns the second summand into $dU$ comes from \eqref{eqn.d_potential}; the $+$ between summands is the cotangent splitting $T^*(\rr\times\inpt M)\cong T^*\rr\oplus T^*\inpt M$ (\cref{sec.manifolds_notation}).} Thus the readout position is the pair,
\[
b\coloneqq\bigl(\outp n,\,\omega_N\bigr)\;\in\;\outp N\times\Omega(\inpt N).
\]

Now given a direction $(\xi_N,\inpt n):T^*_{\outp n}\outp N\times\inpt N=\Phiconf\lensob N[b]$, the coalgebra function $\Phiconf(\rv Q,f)$ at $(q,a)$ produces a direction in $T^*_{\outp m}\outp M\times\inpt M$ by the on-directions component of the same map $\Phip{\interpsm}(f)$, i.e.\ of $\Theta_\potalg(f)$ \eqref{eqn.theta_morphism}. The input point is pulled back along the backward map $\inpt f$,
\begin{equation}\label{eqn.inptm}
\inpt m\coloneqq\inpt f(q,\outp m,\inpt n),
\end{equation}
and $\xi_N$, $\omega_M$, and the potential are pulled back along the maps of \eqref{eqn.para_potential_lens_maps} to give
\begin{equation}\label{eqn.covector_triple}
(\xi_Q,\,\xi_M,\,\xi_{\inpt N})\coloneqq (T_{(q,\outp m)}\outp f)^\top\xi_N \;+\; (T_{(q,\outp m,\inpt n)}\inpt f)^\top\omega_M(\inpt m) \;+\; dU\big|_{(q,\outp m,\inpt n)}\;
\end{equation}
where $(\xi_Q,\,\xi_M,\,\xi_{\inpt N}):Q^*\oplus T^*_{\outp m}\outp M\oplus T^*_{\inpt n}\inpt N$.%
\footnote{Each summand is a cotangent pullback: $\xi_N$ along $\outp f$, $\omega_M(\inpt m)$ along $\inpt f$, and $+1$ along $U$ (via $\potd$, \eqref{eqn.d_potential}); the $+$'s are the cotangent splitting of $Q\times\outp M\times\inpt N$ (\cref{sec.manifolds_notation}), morally the domain of the maps in \eqref{eqn.para_potential_lens_maps}.}
Note that the first summand $(T_{(q,\outp m)}\outp f)^\top\xi_N$ naturally lands in $Q^*\oplus T^*_{\outp m}\outp M$ and has been extended by $0$ in the remaining $\inpt N$-factor, so by \eqref{eqn.omega_n} we have $\omega_N(\inpt n)=\xi_{\inpt N}$.

All of the work has now been done: the coalgebra at $(q,a)$ on $(\xi_N,\inpt n)$ returns the direction
\[
(\xi_M,\,\inpt m):T^*_{\outp m}\outp M\times\inpt M=\Phiconf\lensob M[a]
\]
and updates the state (from its original $q$) to
\begin{equation}\label{eqn.state_update_gradient}
q-\sharpR_q(\xi_Q),
\end{equation}
which is the on-directions component of the descent update $(-1)\circ\chiQ{\rv Q}$ \eqref{eqn.conf_integrator}---the Euler step $\chi$ fed the negative gradient---applied to the covector $\xi_Q\in Q^*$ of \eqref{eqn.covector_triple}. In other words, the state is updated by summing three covectors pulled back along the maps in \eqref{eqn.para_potential_lens_maps}---$\xi_N$ along $\outp f$, $\omega_M(\inpt m)$ along $\inpt f$, and $+1$ along $U$ (yielding $dU$)---projecting to the $Q^*$-component, sharpening to a vector via $\sharpR_q$, and subtracting it from the original state.

All in all, one could tersely denote the coalgebra map as follows:
\[
\Phiconf(\rv Q,f)\colon
q\mapsto
(\outp m,\omega_M)\mapsto
\Big(
	(\outp n,\omega_N),
	(\xi_N,\inpt n)\mapsto\\
	\big(
		(\xi_M,\inpt m),
		q-\sharpR_q(\xi_Q)
	\big)
\Big),
\]
where $\outp n$, $\omega_N$, $\xi_M$, $\inpt m$, and $\xi_Q$ are as in \eqref{eqn.outpn}, \eqref{eqn.omega_n}, \eqref{eqn.inptm}, and \eqref{eqn.covector_triple}.

We can now fulfill the promise from \cref{sec.related_cls}.

\begin{proposition}[Eulerized submersion lenses]\label{prop.euler_submersion_lenses}
For $\cat R$ as in \cref{sec.related_cls}, there is a category $\cat S$, whose objects are submersions, and a diagram
\begin{equation}\label{eqn.cls_diag}
\begin{tikzcd}[column sep=large, row sep=large]
\mathbf{OpenErg}\ar[d, "{\mathrm{collapse}\circ T^*}"'] &
\cat{R}\ar[l, "{\mathrm{wide}}"']\ar[r, "{\mathrm{ff}}"]\ar[d] &
\sarr\ar[d, "\Phiconf"]\\
\mathbf{COrg} &
\cat{S}\ar[l, "{\mathrm{wide}}"]\ar[r] &
\pc
\end{tikzcd}
\end{equation}
where the middle vertical functor sends $M$ to the submersion $T^*M\surj M$, and the lower right functor sends a submersion $p\colon E\surj B$ to the polynomial $\sum_{b:B}\yon^{E_b}$.
\end{proposition}

\begin{proof}
Let $\defineTerm{SLens}$ denote the symmetric monoidal category whose objects are submersions $p\colon E\surj B$ and whose morphisms $p\to p'$, for $p'\colon E'\surj B'$, are smooth maps
\[
\begin{tikzcd}
E\ar[d, ->>, "p"']&
B\times_{B'}E'\ar[l, "\alpha"']\ar[r]\ar[d, ->>]&
E'\ar[d, ->>, "p'"]\\
B\ar[r, equal]&
B\ar[r, "r"']&
B'\ar[ul, phantom, very near end, "\lrcorner"]
\end{tikzcd}
\]
The monoidal product is defined by $(E\surj B)\otimes(E'\surj B')\coloneqq(E\times E'\surj B\times B')$, and the unit is $1\surj 1$.
There is a strong symmetric monoidal functor $P\colon\SLens\to\poly$ given on objects by
\[
P(E\surj B)\coloneqq\sum_{b:B}\yon^{E_b}.
\]
It sends a submersion lens $(r,\alpha)\colon p\to p'$ to the polynomial map with forward component $r$ and backward component $\alpha$, and is strong monoidal by construction.

The strong monoidal functor
\[
\defineTerm{vectiso}\to\SLens,
\qquad
Q\longmapsto (TQ\surj Q)
\]
induces an action of $\vectiso$ on $\SLens$; let $\cat{S}\coloneqq\para{\vectiso}{\SLens}$. The objects of $\cat S$ are submersions, and a morphism $p\to p'$ in $\cat S$ consists of a tuple $(Q,r,u,\alpha)$, where $Q$ is a finite-dimensional vector space, $r\colon Q\times B\to B'$ is the forward map on bases, and $(u,\alpha)\colon(Q\times B)\times_{B'}E'\to Q\times E$ is the backward map, assembled into the submersion lens
\[
\begin{tikzcd}
Q\times E\ar[d, ->>, "Q\times p"']&
(Q\times B)\times_{B'}E'\ar[l, "{(u,\alpha)}"']\ar[r]\ar[d, ->>]&
E'\ar[d, ->>, "p'"]\\
Q\times B\ar[r, equal]&
Q\times B\ar[r, "r"']&
B'\ar[ul, phantom, very near end, "\lrcorner"]
\end{tikzcd}
\]

Let $\absval{\blank}\colon\vectiso\to\smsetiso$ be the underlying-set functor. For $Q:\vectiso$ and $p\colon E\surj B$, define a polynomial map\footnote{The Euler step $\chi'_{Q,p}$ is a sharp-free analogue of $\chi$ (\cref{prop.rvect_polynomial}); we do not spell out the relation here.}
\[
\chi'_{Q,p}\colon \Store(\absval{Q})\otimes P(p)\to P\bigl((TQ\surj Q)\otimes p\bigr).
\]
Both sides have positions $(q,b)\in Q\times B$, so take $\chi'_{Q,p}$ on positions to be the identity. Using the canonical identification $TQ\cong Q\oplus Q$ \eqref{eqn.TT_def}, the map $\chi'_{Q,p}\colon T_qQ\times E_b\to Q\times E_b$ on directions at $(q,b)$ is given by
\[
(\dot q,e)\mapsto (q+\dot q,e).
\]
These maps are natural in vector-space isomorphisms and submersion lenses, and monoidal by componentwise addition in direct sums. Thus \cref{prop.para_square} induces a normal lax symmetric monoidal functor
\[
\cat{S}\coloneqq\para{\vectiso}{\SLens}\to\para{\smsetiso}{\poly}.
\]
Composing with the identification $\para{\smsetiso}{\poly}\cong\pciso$ of \cref{prop.pc_as_para} and the inclusion $\pciso\inj\pc$ gives a normal lax symmetric monoidal functor $\cat S\to\pc$. Unwinding this composite sends the above morphism $(Q, r, u,\alpha)$ to the $\ihom{P(p),P(p')}$-coalgebra with state set $Q$ whose action is
\begin{equation}\label{eqn.S_coalgebra}
q\longmapsto b\longmapsto
\left(
r(q,b),\;
a'\longmapsto\bigl(\alpha(q,b,a'),\,q+u(q,b,a')\bigr)
\right).
\end{equation}

The functor $\cat S\to\mathbf{COrg}$ is more straightforward: an $\cat S$-morphism is a vector space $Q$ together with a lens $(TQ\surj Q)\otimes p\to p'$, hence a $\mathbf{COrg}$ morphism with state manifold $Q$. It is identity-on-objects, hence wide.

For $\cat{R}$ we take the category $\cat R$ from \cref{sec.related_cls}. We define $\cat R\to\cat S$ on objects by $M\mapsto T^*M\surj M$. On morphisms, send $(\rv Q,w,E)\colon M\to N$ to the $\cat S$-morphism with parameter space $Q$, readout $r\coloneqq w$, and for the rest of the data write
\[
(\xi_Q,\xi_M)\coloneqq dE|_{(q,m)}+(T_{(q,m)}w)^\top(\xi_N)\in Q^*\oplus T_m^*M
\]
and set $\alpha(q,m,\xi_N)\coloneqq\xi_M$ and $u(q,m,\xi_N)\coloneqq-\sharpR_q(\xi_Q)$, the descent velocity. Functoriality follows from the chain rule.

The right square in \eqref{eqn.cls_diag} commutes as follows. Under the fully faithful functor $\cat R\to\sarr$, the morphism $(\rv Q,w,E)$ is sent to the smooth adaptive arrangement with output map $w$, unique input map, and potential $E$. Since the input interface is terminal, the only nonzero terms in \eqref{eqn.covector_triple} are
\[
dE|_{(q,m)}+(T_{(q,m)}w)^\top(\xi_N)=(\xi_Q,\xi_M).
\]
Thus \eqref{eqn.state_update_gradient} sends the state $q$ to $q-\sharpR_q(\xi_Q)$ and returns the internal covector $\xi_M$. The $\cat S$-image of $(\rv Q,w,E)$ has readout $r=w$, $\alpha(q,m,\xi_N)=\xi_M$, and $u(q,m,\xi_N)=-\sharpR_q(\xi_Q)$, so \eqref{eqn.S_coalgebra} reads
\[
q\mapsto m\mapsto
\Big(
w(q,m),\;
\xi_N\mapsto\bigl(\xi_M,\,q-\sharpR_q(\xi_Q)\bigr)
\Big),
\qquad
(\xi_Q,\xi_M)=dE|_{(q,m)}+(T_{(q,m)}w)^\top(\xi_N),
\]
which is exactly the $\Phiconf$-coalgebra \eqref{eqn.state_update_gradient} of the same diagram. Hence the right square commutes. The descent vs. ascent sign here sits in the configuration integrator \eqref{eqn.conf_integrator}---seen as the minus in the velocity $u=-\sharpR_q(\xi_Q)$---while the response $\sharpR_Q$ stays sign-free; CLS make the opposite choice, folding the sign into the response so that the velocity is $\sharpR_Q(dE)$ directly~\cite[\S\S 2.1--2.2]{capucci2024organizing}. The remaining compatibility with $\mathbf{OpenErg}$ is obtained by restricting to $\cat R$ the CLS functor $\mathbf{OpenErg}\to\mathbf{OpenODE}$ and then applying their ``collapse'' construction~\cite[Thm.~22 and Prop.~21]{capucci2024organizing}.\qedhere
\end{proof}

Applications of $\Phiconf$ appear in \cref{sec.newton_warmup} (Newton's method) and \cref{sec.dl_warmup} (gradient descent and backpropagation).

\section{Phase-space dynamics \texorpdfstring{$\termraw{Phiphase}$}{Phi phase}}\label{sec.phase_dynamics}

We remind the reader of \cref{warn.prime}: primes (e.g.\ $\xi', q'$) never denote derivatives in this paper.

Recall from \cref{thm.dynamics_functor} that $\Phiphase=\Psisem{\phase}\circ\Phip{\interpsm}$. The functor $\Psisem{\phase}$ replaces a $\cotof{Q}$-parameter with $\Store(S)$, where $S\coloneqq\absval{T^*Q}$:
\[
\Store(S)\To{\dynphase}\cotof{Q}.
\]
A state consists of a position and a momentum $(q,\xi)\in T^*Q$, and an incoming covector $\xi'$ updates that state to $(q+\sharpR_q(\xi),\,\xi-\xi')$
 by \eqref{eqn.phase_update}. The readout and returned direction of the coalgebra $\Phiphase(\rv Q,f)$ have the same formulas as the configuration case---they are the on-positions and on-directions components of the shared factor $\Phip{\interpsm}$, formulas \eqref{eqn.outpn}, \eqref{eqn.omega_n}, \eqref{eqn.inptm}, \eqref{eqn.covector_triple}, which depend only on the parameter position---but the phase integrator evaluates them at the presented position $\tilde q=q+\sharpR_q(\xi)$ \eqref{eqn.presented_position}, where $\Phiconf$ evaluates them at the stored $q$. We unpack all this below.

\subsection{The phase coalgebra}\label{sec.phase_coalgebra}

\paragraph{The action of $\termraw{Phiphase}$ on objects.}
The functor $\Phiphase$ acts on objects exactly as $\Phiconf$ in \eqref{eqn.Phi_on_obs}: a position in $\Phiphase\lensob M$ is a pair $(\outp m,\omega)$ and a direction at $(\outp m,\omega)$ is a pair $(\xi,\inpt m)$, with the same types as in the configuration case (see \eqref{eqn.Phi_on_obs}).

\paragraph{The action of $\termraw{Phiphase}$ on morphisms.}
Describing the action of $\Phiphase$ on morphisms will extend to the end of this section, \cref{sec.phase_dynamics}. Let $(Q,\sharpR_Q):\rvect$ and let
$f\colon Q\cdot\lensob M\to\lensob N$
be an arrangement in $\sarr$, i.e.\ a tuple $(\outp f,\inpt f,U)$ as in \eqref{eqn.para_potential_lens_maps}.
It is sent by $\Phiphase$ to a $[\Phiphase\lensob M,\Phiphase\lensob N]$-coalgebra whose carrier is the underlying set of $T^*Q$,
\begin{equation}\label{eqn.state_space}
S\coloneqq\absval{T^*Q}\cong Q\oplus Q^*;
\end{equation}
we write $s=(q,\xi)\in S$ for position $q\in Q$ and momentum $\xi\in Q^*$. Although the coalgebra remembers only the set $S$, the state update below uses the canonical symplectic pairing on $T^*Q$. Its sharp map is
\begin{equation}\label{eqn.sharp_S}
\sharpS_S\colon S^*\to S,\qquad (\xi',q')\mapsto(q',-\xi'),
\end{equation}
specializing \eqref{eqn.canonical_sharp}.

We now unpack the coalgebra function itself,
\[
S\To{\Phiphase(\rv Q,f)}
\left[\Phiphase\lensob M,\Phiphase\lensob N\right](S).
\]
Given a state $s\coloneqq(q,\xi)\in S$ and a position $a\coloneqq(\outp m,\omega_M):\outp M\times\Omega(\inpt M)$, the readout position $b=(\outp n,\omega_N)\in\outp N\times\Omega(\inpt N)$ is the on-positions component of $\Phip{\interpsm}(f)$ at the parameter position $\tilde q$, with $\outp n$ and $\omega_N$ as in \eqref{eqn.outpn} and \eqref{eqn.omega_n} with $\tilde q$ in place of $q$. 

Given a direction $(\xi_N,\inpt n):T^*_{\outp n}\outp N\times\inpt N=\Phiphase\lensob N[b]$, the on-directions component of $\Phip{\interpsm}(f)$ returns the direction $(\xi_M,\inpt m):T^*_{\outp m}\outp M\times\inpt M=\Phiphase\lensob M[a]$, with $\inpt m$ as in \eqref{eqn.inptm} and the triple $(\xi_Q,\,\xi_M,\,\xi_{\inpt N})$ as in \eqref{eqn.covector_triple}, again at $\tilde q$.
These have the same form as the configuration case (\cref{sec.configuration_dynamics}): $\Phiphase$ and $\Phiconf$ share the factor $\Phip{\interpsm}$, whose on-positions and on-directions components depend only on the parameter position; the phase integrator presents it as $\tilde q$ \eqref{eqn.presented_position} where the configuration integrator presents the stored $q$. The remaining difference is in the state update: starting from $s=(q,\xi)$, the state is updated to
\begin{equation}\label{eqn.state_update}
s+\sharpS_S(\xi_Q,\sharpR_q(\xi))=\bigl(q+\sharpR_q(\xi),\,\xi-\xi_Q\bigr).
\end{equation}

We will now break down \cref{eqn.state_update}, which is the phase integrator formula \eqref{eqn.phase_update} with $\xi'\coloneqq\xi_Q$; note that the updated position is exactly the presented one, $q+\sharpR_q(\xi)=\tilde q$. Its left-hand side records where the three ingredients of this update come from. The sum $s+\sharpS_S(\blank)$ is the Euler step $\chi$ on the phase space $T^*Q$ \eqref{eqn.directions_sharp}; the use of $\sharpS_S$ comes from the canonical symplectic pairing on $T^*Q$, \eqref{eqn.canonical_sharp}, which is constant in $s$ regardless of variation in $\sharpR_Q$; and the pair $(\xi_Q,\sharpR_q(\xi))\in T^*_sS$ comes from $\nuro$ \eqref{eqn.nu} at $(q,\xi)$, with $\xi_Q$ computed at $\tilde q$.

The right-hand side of \cref{eqn.state_update} is just a computation: \eqref{eqn.canonical_sharp} gives $\sharpS_S(\xi_Q,\sharpR_q(\xi))=(\sharpR_q(\xi),-\xi_Q)$, then componentwise addition with $s=(q,\xi)$ produces $(q+\sharpR_q(\xi),\,\xi-\xi_Q)$.

All in all, one could tersely denote the coalgebra map as follows:
\begin{multline*}
\Phiphase(\rv Q,f)\colon
(q,\xi)\mapsto
(\outp m,\omega_M)\mapsto
\Big(
	(\outp n,\omega_N),
	(\xi_N,\inpt n)\mapsto\\
	\big(
		(\xi_M,\inpt m),
		(q+\sharpR_q(\xi),\,\xi-\xi_Q)
	\big)
\Big),
\end{multline*}
where $\outp n$ is defined in \eqref{eqn.outpn}, $\omega_N$ in \eqref{eqn.omega_n}, $\inpt m$ in \eqref{eqn.inptm}, and $\xi_M$ and $\xi_Q$ in \eqref{eqn.covector_triple}, all evaluated at the presented position $\tilde q$.

\subsection{Closed systems and symplecticity}\label{sec.closed_conservation}

When $f$ is a closed system (\cref{def.arrangement_terminology}), the potential reduces to a map $U\colon Q\to\rr$ and the two pullback terms of \eqref{eqn.covector_triple} vanish, leaving $\xi_Q=dU|_{\tilde q}$, the differential at the presented position $\tilde q=q+\sharpR_q(\xi)$ \eqref{eqn.presented_position}. The image $\Phiphase(f)$ is then a $\yon$-coalgebra, i.e.\ a function $\absval{T^*Q}\to\absval{T^*Q}$, namely
\begin{equation}\label{eqn.closed_update}
(q,\xi)\;\longmapsto\;\bigl(q+\sharpR_q(\xi),\;\xi-dU|_{\tilde q}\bigr).
\end{equation}
This step is the phase dynamics $\dynphase\colon\Store\circ\absval{T^*\blank}\Rightarrow\cot$ specialized to a closed system: its readout reports the presented position $\tilde q$, and its state update is \eqref{eqn.state_update}.

Equation \eqref{eqn.closed_update} is the closed phase step for an arbitrary response and potential. We now ask when it is \emph{symplectic}, i.e.\ when the step preserves the canonical symplectic pairing on $T^*Q$ \eqref{eqn.canonical_form}. We reach that regime by three hypotheses---imposed cumulatively, each marked in \textbf{bold} below---after which \cref{prop.closed_conservation} proves it symplectic.

If \textbf{the sharp is constant}, $\sharpR_q=\sharpR_0$, then the presented position \eqref{eqn.presented_position} is $\tilde q=q+\sharpR_0(\xi)$ and \eqref{eqn.closed_update} becomes
\[
(q,\xi)\longmapsto\bigl(q+\sharpR_0(\xi),\,\xi-dU|_{\tilde q}\bigr)=\bigl(\tilde q,\,\xi-dU|_{\tilde q}\bigr).
\]
This is one Euler step of the vector field
\[
\dot q\coloneqq \sharpR_0(\xi),
\qquad
\dot\xi\coloneqq-dU|_q,
\]
advancing the position along the velocity and then evaluating the force at the new position $\tilde q$. For instance, take $\sharpR_0(\xi)=\xi/m$ for a scalar mass $m>0$; then $\dot q=\xi/m$ and $\dot\xi=-dU|_q$, the familiar (velocity, force) form of Newton's second law, with $\xi=m\dot q$ the usual momentum.

When moreover \textbf{the sharp is symmetric}, $\xi'(\sharpR_0(\xi))=\xi(\sharpR_0(\xi'))$ for all $\xi,\xi'\in Q^*$, this vector field is \emph{Hamiltonian}, meaning it is the symplectic sharp applied to the differential of some function $H$, namely
\[
H(q,\xi)\coloneqq \tfrac12\xi(\sharpR_0(\xi))+U(q).
\]
When $\sharpR_0$ is positive-definite---as for $\sharpR_0(\xi)=\xi/m$ above---this is the \emph{energy function} of a classical mechanical system~\cite[\S 18.2]{cannasdasilva2008lectures}---kinetic energy $\tfrac12\xi(\sharpR_0(\xi))$ plus potential energy $U$. Indeed, the symmetry condition and linearity of $\sharpR$ give $dH|_{(q,\xi)}=(dU|_q,\sharpR_0(\xi))=\bk{\nuro}{(q,\xi)}(dU|_q)\in Q^*\oplus Q$ by \eqref{eqn.nu}, and \eqref{eqn.sharp_S} recovers the vector field---as advertised---as the symplectic sharp applied to the differential of $H$:
\[
\sharpS_S(dH|_{(q,\xi)})=(\sharpR_0(\xi),-dU|_q)=(\dot{q},\dot{\xi}).
\]
In particular the continuous flow conserves $H$, by \cref{rmk.symplectic_perpendicular}.
The step above is then the \emph{symplectic} Euler step of this Hamiltonian system~\cite[\S 2.1]{gauckler2017dynamics}.

The discrete step \eqref{eqn.closed_update} need not conserve $H$ exactly.
However, when the arrangement is \emph{harmonic} (\cref{def.arrangement_terminology})---quadratic response, \textbf{quadratic potential}---it does preserve the canonical symplectic pairing.
With $\rv Q$ quadratic as above, let $\kappa\colon Q\to Q^*$ be \emph{symmetric}, $(\kappa q)(q')=(\kappa q')(q)$ for all $q,q'\in Q$, and take the potential
\[U(q)\coloneqq\tfrac12(\kappa q)(q).\]
In this case, $dU|_q=\kappa q$ is linear,
\footnote{This is computed  using symmetry of $\kappa$: $dU|_q(q')=\tfrac12\bigl((\kappa q')(q)+(\kappa q)(q')\bigr)=(\kappa q)(q')$.} 
and we will see in \cref{prop.closed_conservation} that the step \eqref{eqn.closed_update} factors as $T^*Q\To{\fun{drift}}T^*Q\To{\fun{kick}}T^*Q$, where
\begin{equation}\label{eqn.drift_kick}
\fun{drift}(q,\xi)\coloneqq\bigl(q+\sharpR_0(\xi),\;\xi\bigr)
\qqand
\fun{kick}(q,\xi)\coloneqq\bigl(q,\;\xi-\kappa q\bigr)
\end{equation}
are the \emph{drift} and \emph{kick} steps of the symplectic Euler method~\cite[\S 2.1]{gauckler2017dynamics}, the drift advancing the position and the kick applying the force.

\begin{proposition}[Symplecticity of the closed phase step]\label{prop.closed_conservation}
Let $f=(Q,\sharpR_0,\bang,\bang,U)\colon I\to I$ be a harmonic closed system (\cref{def.arrangement_terminology}) whose potential is $U(q)=\tfrac12(\kappa q)(q)$ for a symmetric linear map $\kappa\colon Q\to Q^*$. Then, for the drift and kick of \eqref{eqn.drift_kick},
\[\Phiphase(f)=\fun{kick}\circ\fun{drift}.\]
In particular, the update \eqref{eqn.closed_update} is an endomorphism of the responsive vector space $T^*Q$ in $\rvect$, i.e.\ a linear isomorphism preserving the canonical symplectic pairing \eqref{eqn.canonical_form}.
\end{proposition}

\begin{proof}
Since $dU|_q=\kappa q$ and $\tilde q=q+\sharpR_0(\xi)$, the step \eqref{eqn.closed_update} sends $(q,\xi)$ to $\bigl(q+\sharpR_0(\xi),\,\xi-\kappa(q+\sharpR_0(\xi))\bigr)$, which is $\fun{kick}\bigl(\fun{drift}(q,\xi)\bigr)$. Each of $\fun{drift}$ and $\fun{kick}$ is a linear isomorphism, inverted by negating the second summand of its formula. 

For the second claim, $\rvect$-morphisms compose (\cref{def.rvect}), so since the step is $\fun{kick}\circ\fun{drift}$ it suffices to show that $\fun{drift}$ and $\fun{kick}$ are each $\rvect$-morphisms $T^*Q\to T^*Q$. Because the canonical symplectic pairing is nondegenerate and the same at every basepoint, being such a morphism---i.e.\ satisfying sharp-equivariance \eqref{eqn.pairing_triangle}---is equivalent to preserving that pairing (\cref{def.rvect}), which we check from \eqref{eqn.canonical_form}. For the drift, for all $(q,\xi),(q',\xi')\in Q\oplus Q^*$:
\begin{align*}
\pairing{\fun{drift}(q,\xi),\,\fun{drift}(q',\xi')}
&=\xi'\bigl(q+\sharpR_0(\xi)\bigr)-\xi\bigl(q'+\sharpR_0(\xi')\bigr)\\
&=\pairing{(q,\xi),\,(q',\xi')}+\xi'\bigl(\sharpR_0(\xi)\bigr)-\xi\bigl(\sharpR_0(\xi')\bigr),
\end{align*}
and the last two terms cancel by symmetry of the sharp. For the kick:
\[
\pairing{\fun{kick}(q,\xi),\,\fun{kick}(q',\xi')}
=(\xi'-\kappa q')(q)-(\xi-\kappa q)(q')
=\pairing{(q,\xi),\,(q',\xi')}-(\kappa q')(q)+(\kappa q)(q'),
\]
and the last two terms cancel by symmetry of $\kappa$. Each of $\fun{drift}$ and $\fun{kick}$ therefore preserves the pairing, hence lies in $\rvect(T^*Q,T^*Q)$, and so does their composite.
\end{proof}

\chapter{Applications}\label{ch.applications}

In this concluding section we illustrate the two dynamics functors on four concrete examples.\footnote{These constructions are precise enough for a language model to turn into runnable code: point one at this paper with the prompt at \url{https://github.com/dspivak/dap}, and you should obtain a program you can run and modify. The implementation there---written entirely by Claude from the definitions in \cref{ch.smooth_rwd,ch.smooth_dynamics}---has tests reproducing all four examples below, though we have not independently audited it.} \Cref{sec.newton_warmup} applies $\Phiconf$ to a closed system in $\sarr$ to recover Newton's method. \Cref{sec.dl_warmup} applies $\Phiconf$ to a feedforward arrangement to recover gradient descent, with backpropagation appearing as the lens backward pass. \Cref{sec.wave_equation} applies $\Phiphase$ to a chain of harmonic-oscillator particles to recover the discrete wave equation, computed two ways to illustrate functoriality. \Cref{sec.graph_laplacian} generalizes this---again using $\Phiphase$---to any finite directed graph of harmonic particles, recovering its graph Laplacian as a sanity check on the formalism.

\section{Newton's method}\label{sec.newton_warmup}

Newton's method is a \emph{closed system} in the sense of \cref{def.arrangement_terminology}: an arrangement $f\colon \lensunit\to \lensunit$ in $\sarr$. Both $\outp f$ and $\inpt f$ in \eqref{eqn.para_potential_lens_maps} are forced to be the unique map $\bang$, and the only data is a responsive parameter space $(Q,\sharpR_Q)$ and a potential $U\colon Q\to\rr$. Since $\Phiconf\lensunit\cong\yon$ (\cref{cor.functor}), under $\Phiconf$ it determines a closed system in $\pc$, i.e.\ a discrete dynamical system $Q\to[\yon,\yon](Q)=Q$.

Newton's method for \emph{minimization}---given an objective $U$, find a critical point, i.e.\ a point where the gradient $dU$ vanishes---is such a closed system~\cite[Ch.~9]{boyd2004convex}. In the presentation below, the standing hypothesis is that the Hessian of $U$ is invertible at every point of $Q$; 
this holds in particular when the Hessian is everywhere positive definite, in which case the critical point found is the global minimum. %

Let $U\colon Q\to\rr$ be a smooth map on a finite-dimensional real vector space $Q$. Its differential $dU\colon Q\to Q^*$ is a smooth map sending $q\in Q$ to the gradient covector $dU|_q\in Q^*$. Wanting it to vanish, we approximate $dU$ to first order about $q$; the linear part of that approximation is the differential of $dU$ at $q$, itself a linear map
\begin{equation*}
T_q(dU)\colon Q\to Q^*,
\end{equation*}
called the \emph{Hessian} of $U$ at $q$. In a basis $e_1,\ldots,e_n$ of $Q$, the linear map $T_q(dU)$ is given by the matrix of second partial derivatives
\[
\bigl(T_q(dU)\bigr)_{ij}=\partial_i\partial_j U(q),
\]
which is symmetric by equality of mixed partials. By the standing hypothesis it is invertible at every $q$, hence an isomorphism $T_q(dU)\colon Q\To{\cong}Q^*$, and we obtain a nondegenerate responsive vector space (\cref{def.rvect}) by putting
\[
\sharpR^{U}_q\coloneqq\bigl(T_q(dU)\bigr)^{-1}\colon Q^*\To{\cong}Q,
\]
so $\rv Q\coloneqq (Q,\sharpR^{U})\in\rvect$. We thus have a closed system in $\sarr$ with potential $U$,
\[
f=\bigl(\rv Q,\ \outp f=\bang,\ \inpt f=\bang,\ U\bigr)\colon I\to I,
\]
and \emph{Newton's method for $U$} is the discrete dynamical system it determines,
\begin{equation}\label{eqn.newton}
\Phiconf(f)=\Phiconf\bigl((Q,\sharpR^{U}),\bang,\bang,U\bigr)\colon Q\to Q.
\end{equation}

This coalgebra is carried by the state set $\absval Q$; in \eqref{eqn.covector_triple} the $\xi_N$- and $\omega_M$-summands vanish, leaving only $dU|_q$, and the state update \eqref{eqn.state_update_gradient} reads
\[
q\longmapsto q-\sharpR^{U}_q(dU|_q)=q-\bigl(T_q(dU)\bigr)^{-1}(dU|_q),
\]
one step of Newton's method for finding a critical point of $U$. For example if $Q=\rr$ and $U(q)=(q-a)^2+b$, the Hessian is $2$, so $\sharpR^{U}_q=\frac12$, and the update rule takes $q\mapsto q-\frac12(2q-2a)=a$, which is the critical point.

For a non-quadratic example, take $Q=\rr$ and $U(x)=e^x-x$. Its Hessian $T_x(dU)=e^x$ is positive at every $x$, so the standing hypothesis holds; the gradient is $dU|_x=e^x-1$, and the sharp is the inverse Hessian $\sharpR^{U}_x=(e^x)^{-1}=e^{-x}$. The update rule then reads
\[
x\longmapsto x-\sharpR^{U}_x(dU|_x)=x-e^{-x}(e^x-1)=x-1+e^{-x}.
\]
Its unique fixed point is the minimum at $x=0$.

In general, the response structure $\sharpR^{U}_q=\bigl(T_q(dU)\bigr)^{-1}$ locally approximates $U$ by a quadratic, and the configuration integrator's descent step jumps to that quadratic's critical point. When $U$ is exactly quadratic, the Hessian $T_q(dU)$ is constant, so this quadratic model lands on the true critical point in one step. In dimension $>1$, this could be a maximum, minimum, or saddle. 

This is genuinely different from gradient descent: there the response structure $\sharpR=\etaLR\,\id$ is a fixed metric on the parameter space, chosen independently of the potential, whereas here $\sharpR^{U}_q=\bigl(T_q(dU)\bigr)^{-1}$ is derived from the potential itself. It is this dependence on the second derivative of $U$ that makes the method second-order, and that lets the same descent step find any critical point rather than only moving downhill toward minima.

\section{Gradient descent and backpropagation}\label{sec.dl_warmup}

By \emph{deep learning} we mean the training of an artificial neural network by gradient descent with backpropagation~\cite{rumelhart1986learning}. Applying \cref{prop.lens_inclusion} with $\cat C=\mfd$ and $\Fun T=\rr\times\blank$ gives a strong symmetric monoidal functor $\inc\colon\mfd\to\plmfd$, $M\mapsto\binom{1}{M}$. Whiskering with $\para\rvect{\blank}$ and composing with $\Phiconf$ yields a normal lax symmetric monoidal functor
\[\para{\rvect}{\mfd}\To{\para{\rvect}{\inc}}\sarr\To{\Phiconf}\pc,\]
whose composite we denote by
\begin{equation}\label{eqn.lrn}
\lrn\colon\para{\rvect}{\mfd}\too\pc.
\end{equation} 
Training a neural network is a case of our configuration dynamics construction, in the case that input manifolds and potentials are trivial: $\inpt M=1$ and $U=0$.

{\emergencystretch=2em
We now unfold this construction in coordinates, taking the source and target manifolds---the input and output layers of the neural network---to be $\rr^m$ and $\rr^n$.
\par}

\paragraph{A parameterized map as a smooth adaptive arrangement.}
Suppose we are given a smooth parameterized map
\begin{equation}\label{eqn.model_map}
F\colon Q\times\rr^m\to\rr^n,
\end{equation}
  thought of as a feedforward neural network with parameter space $(Q,\sharpR_Q):\rvect$---i.e., $Q$ comes equipped with a smooth map $\sharpR_Q\colon Q\to\vect(Q^*,Q)$ (\cref{def.rvect}), whose value $\sharpR_q\colon Q^*\to Q$ at $q\in Q$ will act as a (possibly position-dependent) learning rate. This $F$ is regarded by $\para{\rvect}{\inc}$ as a smooth adaptive arrangement
\begin{equation}\label{eqn.learning_lens}
f=\bigl(\rv Q,\ \outp f=F,\ \inpt f=\bang,\ U=0\bigr)\colon\binom{1}{\rr^m}\to\binom{1}{\rr^n},
\end{equation}
with output map $\outp f=F\colon Q\times\rr^m\to\rr^n$, trivial input map $\inpt f=\bang$, and zero potential.

\paragraph{The interfaces.}
On objects, $\lrn(\rr^m)$ is given by
\[
\lrn(\rr^m)=\Phiconf\binom{1}{\rr^m}
=\cotof{\rr^m}\otimes\ihom{\cotof{1},\yon}
\cong\cotof{\rr^m}=\sum_{x:\rr^m}\yon^{T^*_x\rr^m},
\]
because $\cotof{1}\cong\yon$ and $\ihom{\yon,\yon}\cong\yon$. Hence $\Phiconf(f)$ is a coalgebra of the form $Q\to [\cotof{\rr^m},\cotof{\rr^n}](Q)$.

\paragraph{The coalgebra.}
Because we are using the configuration integrator, the state set is $\absval{Q}$ by \eqref{eqn.state_space_conf}; an element is a vector of the ``weights and biases'' of the artificial neural network. We unpack the coalgebra step at fixed $q\in Q$, given an input position $x\in\rr^m$ from upstream and a covector $\xi\in T^*_{F(q,x)}\rr^n$ from downstream; where these come from during the training process is taken up in the loss construction below and \cref{ex.training_data}.

At state $q\in Q$ and input data $x\in\rr^m$, the readout point is
\begin{equation*}
\outp n\coloneqq\outp f(q,x)=F(q,x),
\end{equation*}
which is \eqref{eqn.outpn} in this special case.

Now we feed the external interface a direction $\xi\in T^*_{F(q,x)}\rr^n$. Since $\inpt f$ is the unique map to $1$ and $U=0$, the second and third summands in \eqref{eqn.covector_triple} vanish. Thus the only nonzero contribution is the cotangent pullback along $F$,
\begin{equation}\label{eqn.dl_pullback}
(\xi_Q,\xi_M)=(T_{(q,x)}F)^\top\xi\in Q^*\oplus T^*_x\rr^m,
\end{equation}
which breaks up into a parameter covector $\xi_Q$ and a covector $\xi_M$ that's returned to the $\binom{1}{\rr^m}$ interface. In terms of partial derivatives,
\[
\xi_Q=\bigl(T_q(F(\blank,x))\bigr)^\top\xi
\qqand
\xi_M=\bigl(T_x(F(q,\blank))\bigr)^\top\xi.
\]
By \eqref{eqn.state_update_gradient}, the parameter update is therefore
\begin{equation}\label{eqn.dl_gradient_update}
q\longmapsto q-\sharpR_q(\xi_Q).
\end{equation}
Altogether, suppressing the unique $1$-components, the coalgebra is
\[
\Phiconf(F)(q)(x)
=\Bigl(F(q,x),\;\xi\mapsto\bigl(\xi_M,\,q-\sharpR_q(\xi_Q)\bigr)\Bigr),
\]
where $(\xi_Q,\xi_M)$ is defined by \eqref{eqn.dl_pullback}.

\paragraph{Losses and backpropagation.}
Recall the network $f\colon\binom{1}{\rr^m}\to\binom{1}{\rr^n}$: responsive parameter $(Q,\sharpR_Q)\colon\rvect$ modeling  the weights, with $\sharpR_Q$
the (possibly position-dependent) learning rate,
forward map $\outp f=F\colon Q\times\rr^m\to\rr^n$, trivial input map, and zero potential (though see \cref{def.reg_potential}). With the potential zero, the learning signal comes from whatever the network is wired into. We supply it with a \emph{loss}: a stateless $2$-ary closed arrangement (\cref{sec.potentialized_lenses})
\begin{equation}\label{eqn.loss}
\fun{loss}\colon\binom{1}{\rr^n}\ltens\binom{1}{\Lambda}\to I,
\end{equation}
with trivial maps and potential $U\colon\rr^n\times\Lambda\to\rr$ comparing a prediction $y\in\rr^n$ against a label $\lambda\in\Lambda$ (for regression, $\Lambda=\rr^n$ and $U(y,\lambda)=\tfrac12\lVert y-\lambda\rVert^2$). Define $g\colon\binom{1}{\rr^m}\ltens\binom{1}{\Lambda}\to I$ to be the composite
\begin{equation}\label{eqn.dl_composite}
\binom{1}{\rr^m}\ltens\binom{1}{\Lambda}\To{f\,\ltens\,\binom{1}{\Lambda}}\binom{1}{\rr^n}\ltens\binom{1}{\Lambda}\To{\fun{loss}}I,
\end{equation}
a closed arrangement with potential $U(F(q,x),\lambda)$ and open inputs a datum $x\in\rr^m$ and its label $\lambda\in\Lambda$. Applying $\Phiconf$, functoriality splits the result along the composition,
\[
\Phiconf(g)=\Phiconf(\fun{loss})\circ\Phiconf\bigl(f\ltens\tbinom{1}{\Lambda}\bigr),
\]
whose second factor runs $\lrn(F)$ on the input while carrying the label through. The loss factor $\Phiconf(\fun{loss})$ supplies the output covector $\xi=dU(\blank,\lambda)|_{F(q,x)}$---the differential of $U(\blank,\lambda)\colon\rr^n\to\rr$ at the prediction---which $\lrn(F)$ pulls back as in \eqref{eqn.dl_pullback}. 
The parameter covector $\xi_Q$ drives the descent update \eqref{eqn.dl_gradient_update}. On a Euclidean parameter space (\cref{ex.euclidean_sharp}), with the constant sharp scaled by a \emph{learning rate} $\etaLR>0$,
\begin{equation}\label{eqn.learning_sharp}
\sharpR_q=\etaLR\sharpEuc{},
\end{equation}
this becomes ordinary gradient descent at rate $\etaLR$:
\begin{equation}\label{eqn.gradient_descent}
q\longmapsto q-\sharpR_q(\xi_Q)=q-\etaLR\,\bigl(T_qF(\blank,x)\bigr)^\top\xi.
\end{equation}
The returned $\xi_M$ is propagated backward to the previous interface, if any; for a composite network these repeated $\xi_M$'s are the usual backpropagated gradient signal. Backpropagation is thus the chain rule, packaged as the functoriality of $\Phiconf$.
See \cite{shapiro2022dynamic}.

\begin{example}[Training data as an environment]\label{ex.training_data}
The composite $g$ \eqref{eqn.dl_composite} leaves two interfaces open---the input $\rr^m$ and the label $\Lambda$---which we close with an \emph{environment}: a $\pc$-map
\[
e\colon\yon\to\Phiconf\bigl(\tbinom{1}{\rr^m}\ltens\tbinom{1}{\Lambda}\bigr)
\]
whose job is to present the training stream. At any state, it emits a datum $(x,\lambda)\in\rr^m\times\Lambda$---an input together with its label---and, after the backward pass, receives a backward covector on each leg, on which its next state may depend. The composite
\[
\yon\To{e}\Phiconf\bigl(\tbinom{1}{\rr^m}\ltens\tbinom{1}{\Lambda}\bigr)\To{\Phiconf(g)}\yon
\]
is a closed discrete dynamical system on $\absval{Q}\times Q_e$, with $Q_e$ the carrier of $e$: at each step the training environment $e$ emits $(x,\lambda)$, the network computes $y=F(q,x)$, the loss seeds the output covector $\xi=dU(\blank,\lambda)|_{y}\in T^*_y\rr^n$, which pulls back to the parameter gradient $\xi_Q$; the parameter descends by \eqref{eqn.dl_gradient_update}, and the two backward covectors return to $e$---the input-gradient $\xi_M$ and the label-gradient $\xi_\Lambda=\partial_\lambda U(y,\lambda)$---after which $e$ advances.

When $e$ ignores these covectors it is an externally specified supervised stream, i.e.\ a fixed dataset of pairs $(x,\lambda)$. When its transitions are based on them, the stream adapts to the model's backpropagated covectors. For instance, when $\Lambda=\rr^n$, for the squared-error loss $U(y,\lambda)=\tfrac12\lVert y-\lambda\rVert^2$, we have $\xi_\Lambda=\lambda-y$, so $e$---which knows $\lambda$---can read off the prediction $y$ and adapt to it.
\end{example}

\paragraph{Regularization.}
The architecture above \eqref{eqn.learning_lens} carried the zero potential, and the entire learning signal arrived from the downstream loss. But using a potential can be useful, it costs no new machinery---\eqref{eqn.covector_triple} already accounts for it---and it is precisely where regularization comes in.

\begin{definition}[Regularizing potential]\label{def.reg_potential}
A \emph{regularizing potential} on a map $f\colon\binom{1}{\rr^m}\to\binom{1}{\rr^n}$ in $\sarr$ is a choice of potential $U\colon Q\times\rr^m\to\rr$ in \eqref{eqn.para_potential_lens_maps}. It is \emph{weight-only} if $U(q,x)=R(q)$ is independent of the input $x$.
\end{definition}

The simplest weight-only potential recovers a standard ingredient of training. 

\begin{proposition}[Weight decay]\label{prop.weight_decay}
Switching the potential of the arrangement \eqref{eqn.learning_lens} from $0$ to the weight-only $R(q)\coloneqq\tfrac\gamma2\lVert q\rVert^2$, on a Euclidean parameter space with $\sharpR_q=\etaLR\,\sharpEuc{}$ (\cref{ex.euclidean_sharp}, \eqref{eqn.learning_sharp}), changes the gradient-descent update \eqref{eqn.gradient_descent} to
\[
q\longmapsto (1-\etaLR\gamma)\,q-\etaLR\,\bigl(T_qF(\blank,x)\bigr)^\top\xi, 
\]
i.e.\ it gives weight decay, multiplying the weights by a per-step factor of $1-\etaLR\gamma$.
\end{proposition}

\begin{proof}
Take $U=R$. Since $\inpt f=\bang$, the input summand of \eqref{eqn.covector_triple} vanishes, and since $R$ is weight-only, $\partial_qR=\gamma q$ and $\partial_xR=0$. Hence the $Q^*$-component of \eqref{eqn.covector_triple} is the sum
\[
\xi_Q=(T_qF(\blank,x))^\top\xi+\gamma q,
\]
the parameter covector of \eqref{eqn.dl_pullback} plus $\gamma q$, and its $\outp M$-component $\xi_M$ is exactly that of \eqref{eqn.dl_pullback}, since $\partial_xR=0$. Substituting $\xi_Q$ into \eqref{eqn.state_update_gradient}, with $\sharpR_q=\etaLR\sharpEuc{}$ \eqref{eqn.learning_sharp} and $\sharpEuc{}\colon Q^*\cong Q$ the identity, we obtain the desired equation:
\[
q\mapsto q-\sharpR_q(\xi_Q)=q-\etaLR\bigl((T_qF(\blank,x))^\top\xi+\gamma q\bigr)=(1-\etaLR\gamma)\,q-\etaLR\,(T_qF(\blank,x))^\top\xi.\qedhere
\]
\end{proof}

\section{The wave equation}\label{sec.wave_equation}

We first illustrate $\Phiphase\colon\sarr\to\pc$ by deriving the wave equation. We describe a single particle (with harmonic potential) as a $0$-ary smooth adaptive arrangement $\Part\colon\lensunit\to\boxob$, and we arrange such particles into a chain via a $K$-ary smooth (static) arrangement $\fun{wire}_K\colon\boxob\ltpow{K}\to\boxob$. In \cref{sec.spring_first_pass} we compose the particles and the wiring in $\sarr$, apply $\Phiphase$, and recover the discrete wave equation. In \cref{sec.spring_second_pass} we apply $\Phiphase$ to the particle and wiring separately, compose the resulting coalgebras in $\pc$, and verify term-by-term agreement as a check on signs, indices, and routing.

We begin with a single particle whose state is a transverse displacement and momentum, which we put into a ``box'' that delivers the displacement to its neighbors and updates its state based on some fixed parameters $\kappa,m>0$ and the displacements of its neighboring particles. We wire these boxes together in series
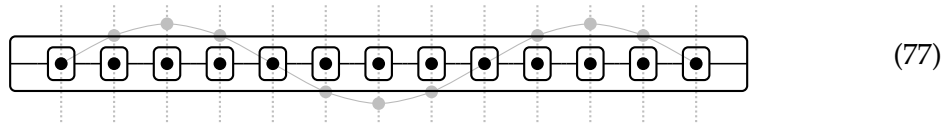
\begin{equation}\label{eqn.wave_wd}
\begin{tikzpicture}[oriented WD, bb port length=4pt, bb port sep=1.5, bb min width=.35cm, bbx=.5cm, bby=.8ex, baseline=(box1), particle/.style={path picture={\fill (path picture bounding box.center) circle (2.4pt);}}]
  \def\N{13}
  \node[bb={1}{1}, particle] (box1) {};
  \foreach \i [evaluate=\i as \prev using int(\i-1)] in {2,...,\N} {
    \node[bb={1}{1}, particle, right=.7 of box\prev] (box\i) {};
    \draw (box\prev_out1) -- (box\i_in1);
  }
  \node[bb={0}{0}, fit={(box1) (box\N)}] (Y) {};
  \node[coordinate] at (Y.west|-box1_in1) (Y_in1) {};
  \node[coordinate] at (Y.east|-box\N_out1) (Y_out1) {};
  \draw (Y_in1) -- (box1_in1);
  \draw (box\N_out1) -- (Y_out1);
  \begin{scope}[on background layer]
    \def\P{1.5}
    \xdef\wave{}
    \foreach \i [evaluate=\i as \ang using (\i-1)/(\N-1)*\P*360] in {1,...,\N} {
      \draw[gray!50, thick, densely dotted] ([yshift=22pt]box\i.center) -- ([yshift=-22pt]box\i.center);
      \coordinate (p\i) at ([yshift={15pt*sin(\ang)}]box\i.center);
      \xdef\wave{\wave (p\i)}
    }
    \draw[gray!55] plot[smooth] coordinates {\wave};
    \foreach \i in {1,...,\N} {\fill[gray!50] (p\i) circle (2.4pt);}
  \end{scope}
\end{tikzpicture}
\end{equation}
and aim to show that the resulting state dynamics evolves according to the discrete wave equation:%
\footnote{Taking a continuum limit with the inter-particle spacing and time step going to zero---with $m$ and $\kappa$ scaled appropriately---should recover the usual wave equation. See \cite[App.~B]{demanet2015waves}.}
\begin{equation}\label{eqn.discrete_wave}
m\ddot q_i=\kappa(q_{i-1}-2q_i+q_{i+1}),
\end{equation}
where for $t\geq1$ we use $\ddot q_i(t)$ to denote
\begin{equation}\label{eqn.doubledot}
\ddot q_i(t)\coloneqq q_i(t+1)-2q_i(t)+q_i(t-1).
\end{equation}

We will work in the full suboperad of $\sarr$ spanned by a single object $\boxob\coloneqq\binom{\rr}{\rr}$, corresponding to a box whose input port and output port are both $\rr$. That is, we'll consider arrangements of the form
\[
f\in\sarr(\boxob\ltpow{K},\boxob)
\]
for varying $K$. Note that $\boxob\ltpow{0}=\lensunit$ is the monoidal unit.

We derive the wave equation, then audit the derivation via functoriality. First, we describe the particle and the wiring in $\sarr$, compose them there, and apply $\Phiphase$ to the resulting $0$-ary smooth adaptive arrangement $\boxob\ltpow{0}\to\boxob$ to read off the dynamics. Then, as a bookkeeping check, we apply $\Phiphase$ to the particle and to the wiring separately, compose in $\pc$, and verify that the same coalgebra is obtained.

\subsection{First pass: composing in \texorpdfstring{$\termraw{srwd}$}{SRWD}}\label{sec.spring_first_pass}

First we give a $0$-ary smooth adaptive arrangement $\Part\colon\lensunit\to\binom{\rr}{\rr}$, for the dynamical system that goes in each internal box, representing the particle of mass $m>0$. Then we give a $K$-ary smooth static arrangement for each $K>0$ that describes the (static) arrangement shown in \eqref{eqn.wave_wd}. Then we calculate what happens when we compose.

\paragraph{The particle.}

To define a particle $\Part\colon\lensunit\to\binom{\rr}{\rr}$
\[
\begin{tikzpicture}[oriented WD, 
	bb port length=4pt, bb min width=.35cm, bbx=.4cm, bby=.8ex, 
	particle/.style=
		{path picture={\fill (path picture bounding box.center) circle (2.4pt);}}]
	\node[bb={1}{1}, particle] {};
\end{tikzpicture}
\]
as a $0$-ary smooth adaptive arrangement in $\sarr$ \eqref{eqn.srw_morphism}, we need a responsive vector space $(Q,\sharpR_Q)$, a smooth output map $\outp f\colon Q\to\rr$, and a smooth potential $U\colon Q\times\rr\to\rr$; the input map is trivial. We take the parameter space to be $\rr\in\rvect$ with constant sharp $\sharpR_q(\xi)\coloneqq\xi/m$ (independent of $q\in\rr$); we read the value $q\in\rr$ as the particle's transverse displacement from equilibrium. We take the output map to be the identity, $\outp f(q)\coloneqq q$, and the potential to be
\[U(q,q')\coloneqq\tfrac{\kappa}{2}(q-q')^2,\]
the \emph{harmonic} potential, assigning energy to the relative displacement of the particle and its neighbor.

\begin{remark}[Other particles]\label{rmk.other_particles}
In $\sarr$-maps of the form $\lensunit\to\binom{\rr}{\rr}$ with parameter $\rr$, the potential alone selects the kind of particle.
\begin{enumerate}
	\item $U(q,q')\coloneqq\tfrac\kappa2(q-q'-\ell)^2$ models a bond whose preferred displacement difference is $\ell$.
	\item $U(q,q')\coloneqq\tfrac\kappa2(q-q')^2+\tfrac\mu2 q^2$ models a particle also pinned toward $q=0$ by an on-site spring.
\end{enumerate}
We will continue to work with the particle above.
\end{remark}

\begin{remark}\label{rmk.symmetric_bonds}
In the above representation, the particle's position is being sent out the output port $\outp f=\id_\rr$, whereas it seems that nothing is being communicated through the particle's input port $\rr$, which appears problematic for the symmetry of heat or wave propagation. In fact, we will see that the potential's differential is communicated through the input port, so in the end everything will symmetrize. However, we could achieve that earlier by refactoring the particle.

Indeed, the map $\Part\colon\lensunit\to\binom{\rr}{\rr}$ is actually the composite of two systems---a pure mass with no potential and a stateless bond---nested inside a static wiring diagram $g\colon\binom{1}{\rr}\ltens\binom{\rr^2}{1}\to\binom{\rr}{\rr}$, given by the following wiring diagram (see \cref{sec.plenspot_ops})
\[
g\coloneqq\begin{tikzpicture}[oriented WD,
	bb port length=4pt, bb min width=.35cm, bbx=.4cm, bby=.8ex]
	\node[bb port sep=1, bb={2}{0}] (bond) {$\Bond$};
	\node[dot, above left=1.5 and 1 of bond_in1] (dot) {};
	\node[bb={0}{1}, bb port sep=1, left=1.25 of dot] (realpart) {$\RealPart$};
	\node[bb={0}{0}, fit=(dot) (bond) (realpart)] (outer) {};
	\coordinate (outer_in1) at (outer.west|-bond_in2);
	\coordinate (outer_out1) at (outer.east|-dot);
	\draw (outer_in1) -- (bond_in2);
	\draw (dot) -- (outer_out1);
	\draw (realpart_out1) -- (dot);
	\draw (dot) to (bond_in1);
\end{tikzpicture}
\]
So $\Part=g(\RealPart,\Bond)$ where $\RealPart\colon\lensunit\to\binom{1}{\rr}$ is potential-free with state $Q=\rr$ and $\outp\RealPart=\id$, and $\Bond\colon\lensunit\to\binom{\rr^2}{1}$ is stateless with potential $U(q_1,q_2)=\frac\kappa2(q_1-q_2)^2$. We will not use these pictures $\RealPart$, $\Bond$ again.

So what we have called the particle $\Part$ already carries its own bond: it is a pure mass (vertex) with a stateless bond attached on one side. If one wishes, the asymmetry in the presentation to follow can be avoided by working this way: replace edges with stateless bonds and vertices with potential-free (pure) particles.%
\footnote{I thank Ryan Tavenner for this observation.}

A stateless bond is, formally, the same construct as the training loss \eqref{eqn.loss}: a stateless $2$-ary closed arrangement whose only content is a potential $U$ of its two arguments. Indeed the squared-error loss $\tfrac12\lVert y-\lambda\rVert^2$ is exactly a harmonic bond that ties the prediction $y$ to its label $\lambda$. So each bond is like a ``training loss'' provided to each mass by its neighbor.
\end{remark}

\paragraph{The wiring diagram.}
For the wiring diagram, we start in $\Lens{\finset\op}$; see \cref{ex.lens_finsetop}. The $K$-ary series composition shown in \eqref{eqn.wave_wd} corresponds to a lens $\varphi_K\colon\binom{K}{K}=\binom{1}{1}\ltens\cdots\ltens\binom{1}{1}\to\binom{1}{1}$. Writing $\{1,\ldots,K\}$ for the inner finite set $K$, and the outer box $\binom{1}{1}$ as $\binom{\{K\}}{\{0\}}$---its output port $\{K\}$ sitting to the right of the inner boxes and its input port $\{0\}$ to the left---the lens is a pair of functions of type
\[
\{1,\ldots,K\}\From{\outp f}\{K\}
\qqand
\{1,\ldots,K\}+\{0\}\From{\inpt f} \{1,\ldots,K\},
\]
where the codomain of $\inpt f$ is the disjoint union of the internal-box outputs $\{1,\ldots,K\}$ and the external-box input $\{0\}$. Here $\outp f$ is the inclusion $\{K\}\hookrightarrow\{1,\ldots,K\}$, reading the external output off the last internal output, and $\inpt f(n)\coloneqq n-1$ feeds each internal input from the preceding wire: the first internal input from the external input $0$, and input $n>1$ from output $n-1$.

Applying $\rr^\blank$ from \cref{lem.lens_pow}, we obtain the lens $\binom{\rr^K}{\rr^K}\to\binom{\rr}{\rr}$ given on coordinates by the functions
\[
\outp f(q_1,\ldots,q_K)\coloneqq q_K
\qqand
\inpt f\bigl((q_1,\ldots,q_K),\,q_0\bigr)\coloneqq(q_0,\,q_1,\,\ldots,\,q_{K-1}).
\]
{\emergencystretch=2em
From this lens we obtain a static arrangement (taking $Q'\coloneqq\rr^0$ and the $0$-potential $U'\coloneqq0\colon\rr^K\times\rr\to\rr$): the wiring
\par}
\[
\fun{wire}_K\colon\boxob\ltpow{K}\to\boxob,\qquad (\outp f,\inpt f,0)
\]

Composing, we obtain a $0$-ary smooth adaptive arrangement, $\fun{wire}_K(\Part,\ldots,\Part)\colon\boxob\ltpow{0}\to\boxob$, with parameter $Q\coloneqq Q'\oplus\rr^{\oplus K}=\rr^K$ (one $\rr$ per particle) and constant sharp
\begin{equation}\label{eqn.sharp_chain}
\sharpR_{\vec q}\colon Q^*\to Q,\qquad \xi\mapsto\xi/m,
\end{equation}
acting per-particle on each $\rr$-summand (the inverse-mass pairing), independent of $\vec q\in Q$. The defining maps are $\inpt f\coloneqq\bang$, vacuous, and 
\[
\outp f(q_1,\ldots,q_K)\coloneqq q_K
\qqand
U\bigl((q_1,\ldots,q_K),\,q_0\bigr)\coloneqq\sum_{i=1}^{K}\tfrac{\kappa}{2}(q_i-q_{i-1})^2.
\]
Here, the sum in the above formula for $U$ comes from the productor \eqref{eqn.monoid_productor} of the monoidal monad $\rr\times\blank$ on $\mfd$. Composition with $\fun{wire}_K$ then substitutes $q'_i\mapsto q_{i-1}$ via $\inpt f$ (with $q_0$ the external input).

\paragraph{The discrete dynamics.}

The discrete dynamics appear when we apply the functor $\Phiphase\colon\sarr\to\pc$. The object $\boxob$ is sent to
\[
\Phiphase(\boxob)=\sum_{(q,\omega):\rr\times\Omega(\rr)}\yon^{T^*_q\rr\times \rr}.
\]
The map $\fun{wire}_K(\Part,\ldots,\Part)$ is sent to a coalgebra
\[
\Phiphase(\fun{wire}_K(\Part,\ldots,\Part))\colon S\To{}\Phiphase(\boxob)(S)
\]
with state set $S\coloneqq T^*\rr^K$, with elements written $s=((q_1,\ldots,q_K),(\xi_1,\ldots,\xi_K))$ for positions and momenta, and with the function given by
\begin{equation*}
(\vec q,\vec\xi)\mapsto\Bigl(\bigl(\tilde q_K,\,\omega_1\bigr),\;(\xi_{\fun{ext}},q_0)\mapsto\bigl(\vec{\tilde q},\,\vec\xi-\xi_Q\bigr)\Bigr),
\end{equation*}
where $\vec{\tilde q}\coloneqq\vec q+\tfrac{1}{m}\vec\xi$ is the presented position \eqref{eqn.presented_position}, computed with the sharp \eqref{eqn.sharp_chain}, and each ingredient is read off from the unpacking in \cref{sec.phase_dynamics}:
\begin{itemize}
\item The output position $\tilde q_K$ comes from $\outp f$ via \eqref{eqn.outpn}, evaluated at $\vec{\tilde q}$.
\item The \emph{harmonic spring} covector field on the left-hand side,
\[\omega_1(q_0)=d\bigl[\tfrac{\kappa}{2}(\tilde q_1-q_0)^2\bigr]\big|_{q_0}=-\kappa(\tilde q_1-q_0),\]
comes from \eqref{eqn.omega_n} (with $\inpt f$ and the incoming covector field vacuous, and $dq_0=+1$ via \eqref{eqn.d_potential}).
\item The pullback covector $\xi_Q\in T^*_{\vec{\tilde q}}Q$ comes from \eqref{eqn.covector_triple}, evaluated at $\vec{\tilde q}$. It is given component-wise (via \eqref{eqn.tangent_vec} identifying $T^*_{\vec{\tilde q}}Q\cong Q^*=(\rr^K)^*$) by
\begin{align*}
(\xi_Q)_i&=
\bigl((T_{\vec{\tilde q}}\outp f)^\top\xi_{\fun{ext}}\bigr)_i+\partial_{q_i}U(\vec{\tilde q},q_0)\\&=
\begin{cases}\partial_{q_i}U(\vec{\tilde q},q_0)+0&1\leq i<K\\ \partial_{q_K}U(\vec{\tilde q},q_0)+\xi_{\fun{ext}}&i=K\end{cases}\\&=
\begin{cases}-\kappa(\tilde q_{i-1}-2\tilde q_i+\tilde q_{i+1})&1\leq i<K\\ \kappa(\tilde q_K-\tilde q_{K-1})+\xi_{\fun{ext}}&i=K\end{cases}
\end{align*}
(for $i=1$, read $\tilde q_0\coloneqq q_0$, the external input).
\item The state update applied to $s=(\vec q,\vec\xi)$ comes from \eqref{eqn.state_update}: using our (constant) sharp $\sharpR_{\vec q}(\vec\xi)=\vec\xi/m$ from \eqref{eqn.sharp_chain} and the symplectic sharp $\sharpS_S\colon(\xi,q)\mapsto(q,-\xi)$ from \eqref{eqn.sharp_S}, we get
\[
s'\coloneqq s+\sharpS_S(\xi_Q,\sharpR_{\vec q}(\vec\xi))=(\vec q+\vec\xi/m,\,\vec\xi-\xi_Q)=(\vec{\tilde q},\,\vec\xi-\xi_Q).
\]
In other words, writing $s'\coloneqq(q_1',\ldots,q_K',\xi_1',\ldots,\xi_K')$, we have
\begin{equation}\label{eqn.potlens_composite_final}
q_i'\coloneqq q_i+\xi_i/m=\tilde q_i
\qqand
\xi_i'\coloneqq\begin{cases}\xi_i+\kappa(\tilde q_{i-1}-2\tilde q_i+\tilde q_{i+1})&1\leq i<K\\ \xi_i+\kappa(\tilde q_{K-1}-\tilde q_K)-\xi_{\fun{ext}}&i=K.\end{cases}
\end{equation}
\end{itemize}

\paragraph{Recovering the wave equation.}
To recover the discrete wave equation everywhere, we pin both ends of the chain by fixing the phantom neighbor displacements outside the chain to zero:
\begin{equation}\label{eqn.wave_pin}
  \tilde q_0\coloneqq 0
  \qqand
  \tilde q_{K+1}\coloneqq 0.
\end{equation}
Supplying these via the external input $(\xi_{\fun{ext}},q_0)$ means choosing the following data:
\begin{equation}\label{eqn.wave_bc}
q_0\coloneqq 0\qqand\xi_{\fun{ext}}\coloneqq \kappa\,\tilde q_K,
\end{equation}
where $\xi_{\fun{ext}}$ is the environment's response to the output $\tilde q_K$ it receives.%
\footnote{Equivalently, we can present the data of a pinned chain as a closed system  by composing with a map $\boxob\to I$ whose lens is $\bang$ forward and $q\mapsto 0$ backward, and whose potential is $U'(q_K)\coloneqq\tfrac\kappa2 q_K^2$.
}
With these choices, the boundary case at $i=K$ becomes $\xi_K'=\xi_K+\kappa(\tilde q_{K-1}+\tilde q_{K+1}-2\tilde q_K)$, same as the others, so the above state update formula reduces to the following for all $1\leq i\leq K$:
\begin{equation}\label{eqn.state_update_explicit}
q_i'\coloneqq q_i+\xi_i/m
\qqand
\xi_i'\coloneqq\xi_i+\kappa(\tilde q_{i-1}-2\tilde q_i+\tilde q_{i+1}).
\end{equation}

Following \cref{ex.y_coalgebra_trajectory}, fix any initial $s(0)\coloneqq(\vec{q}(0),\vec{\xi}(0))$ and write $s(t)$ for the resulting trajectory under \eqref{eqn.state_update_explicit}; in the notation of \eqref{eqn.state_update_explicit}, $s(t+1)=s(t)'$. The first equation says $q_i(t+1)=q_i(t)+\xi_i(t)/m$, whose right-hand side is the presented position $\tilde q_i(t)$ \eqref{eqn.presented_position}; so $\tilde q_i(t)=q_i(t+1)$, and solving for the momentum,
\[\xi_i(t)=m\bigl(q_i(t+1)-q_i(t)\bigr).\]
For $t\geq 1$, we can replace $t$ by $t-1$ in the second equation of \eqref{eqn.state_update_explicit} and it reads $\xi_i(t)-\xi_i(t-1)=\kappa\bigl(q_{i-1}(t)-2q_i(t)+q_{i+1}(t)\bigr)$, using $\tilde q_j(t-1)=q_j(t)$. Substituting the momentum gives:
\begin{equation}\label{eqn.recurrence}
m\bigl(q_i(t+1)-q_i(t)\bigr)-m\bigl(q_i(t)-q_i(t-1)\bigr)=\kappa\bigl(q_{i-1}(t)-2q_i(t)+q_{i+1}(t)\bigr)
\end{equation}
Collecting the position terms on the left into the centered second difference \eqref{eqn.doubledot} yields the discrete wave equation \eqref{eqn.discrete_wave}, $m\,\ddot q_i(t)=\kappa\bigl(q_{i-1}(t)-2q_i(t)+q_{i+1}(t)\bigr)$, at all $t\geq1$.

\begin{proposition}\label{prop.wave_equation}
\leavevmode\par\nobreak
With $\Part$, $\fun{wire}_K$, and the boundary choices \eqref{eqn.wave_pin} and \eqref{eqn.wave_bc} as above, the $\Phiphase(\boxob)$-coalgebra
\[
\Phiphase\bigl(\fun{wire}_K(\Part,\ldots,\Part)\bigr)
\]
recovers the discrete wave equation \eqref{eqn.discrete_wave}.
\end{proposition}
\begin{proof}
This is the calculation of \eqref{eqn.state_update_explicit}: writing the state as a function of time and eliminating the momenta gives \eqref{eqn.discrete_wave}.
\end{proof}

\begin{remark}[Adaptive spring material]\label{rmk.adaptive_spring}
The above chain \eqref{eqn.wave_wd} uses a single constant coupling $\kappa$, but the $\sarr$ formalism supports adaptive ones, as mentioned in \cref{sec.intro_smooth_rewiring}. One could give each bond its own spring constant $\kappa_j$ and treat these constants as responsive parameters in their own right, evolving alongside the displacements. How bonds change according to the state---softening under strain, say---is a modeling choice, governed by whatever potential one picks to encode the intended material behavior; the formalism turns that choice into coupled dynamics.\qedhere
\end{remark}

\subsection{Functoriality audit: composing in \texorpdfstring{$\pc$}{Coalg*Poly}}\label{sec.spring_second_pass}

The wave equation was derived in \cref{sec.spring_first_pass}; operad-functoriality of $\Phiphase$ forces equality of the two sides below:
\[
\Phiphase\bigl(\fun{wire}_K(\Part,\ldots,\Part)\bigr)=\Phiphase(\fun{wire}_K)\bigl(\Phiphase(\Part),\ldots,\Phiphase(\Part)\bigr),
\]
We now work out the right-hand side term by term---applying $\Phiphase$ to the particle and to the wiring separately, then composing in $\pc$---as an explicit check on signs, indices, and the routing of inputs through the formulas. 

\paragraph{The particle dynamics.}

We first apply $\Phiphase$ to $\Part\colon\lensunit\to\boxob$ to obtain a coalgebra on
\[
  \left[\Phiphase\lensunit,\Phiphase(\boxob)\right]=
  \left[\yon,\Phiphase(\boxob)\right]\cong\Phiphase(\boxob)=
  \sum_{(q,\omega):\rr\times\Omega(\rr)}\yon^{T^*_q\rr\times\rr}
\]
We took the parameter space to be $\rr$, so by \eqref{eqn.state_space} the state set of $\Phiphase(\Part)$ is $S\coloneqq T^*\rr$, with elements written $s=(q,\xi)$ for position and momentum. The coalgebra function
\[
S\To{\Phiphase(\Part)}
\sum_{(q,\omega):\rr\times\Omega(\rr)}S^{T^*_q\rr\times\rr}
\]
is given by
\[
(q,\xi)\mapsto\Bigl(\bigl(\tilde q,\,q_{\fun{prev}}\mapsto\kappa(q_{\fun{prev}}-\tilde q)\bigr),\;(\xi_{\fun{next}},q_{\fun{prev}})\mapsto\bigl(\tilde q,\,\xi-\xi_{\fun{next}}+\kappa(q_{\fun{prev}}-\tilde q)\bigr)\Bigr),
\]
where $\tilde q\coloneqq q+\xi/m$ is the presented position \eqref{eqn.presented_position}. In other words, starting at position $q$ and momentum $\xi$, it outputs $\tilde q$---the position it is moving to---to the right, along with the covector field $q_{\fun{prev}}\mapsto\kappa(q_{\fun{prev}}-\tilde q)$, which ``pulls'' $q_{\fun{prev}}$ toward $\tilde q$. Then, receiving a covector $\xi_{\fun{next}}$ from the right and a point $q_{\fun{prev}}$ from the left, it moves to $\tilde q$ and updates its momentum by adding the spring force $\kappa(q_{\fun{prev}}-\tilde q)$, which ``pulls'' $\tilde q$ toward $q_{\fun{prev}}$, and subtracting the received covector $\xi_{\fun{next}}$.
 
\paragraph{The static wiring.}

Apply $\Phiphase$ to $\fun{wire}_K\colon\boxob\ltpow{K}\to\boxob$. With trivial parameter $\rr^0$, the state space $T^*\rr^0$ is a single point, so $\Phiphase(\fun{wire}_K)$ is \emph{static}: a function $1\to[\Phiphase(\boxob)^{\otimes K},\allowbreak\Phiphase(\boxob)](1)$ is equivalently (see \eqref{eqn.dirichlet_hom}) a polynomial map $\Phiphase(\boxob)^{\otimes K}\to\allowbreak\Phiphase(\boxob)$, i.e.\
\[
\sum_{(\vec q,\vec\omega):\rr^K\times\Omega(\rr)^K}\yon^{T^*_{\vec q}\rr^K\times\rr^K}
\To{\Phiphase(\fun{wire}_K)}
\sum_{(q,\omega):\rr\times\Omega(\rr)}\yon^{T^*_q\rr\times\rr}.
\]
To each $(\vec q,\vec \omega)$, the vector of outputs of the internal boxes, we assign the external output $(q_K,\omega_1)$ (the rightmost particle's position with the leftmost box's covector field); and to each external input $(\xi_{\fun{ext}},q_0)$ at the external box, we assign the internal inputs $(\omega_2(q_1),q_0)$, $(\omega_3(q_2),q_1)$, $\ldots$, $(\omega_K(q_{K-1}),q_{K-2})$, $(\xi_{\fun{ext}},q_{K-1})$---each box $i<K$ receives its right neighbor's covector field evaluated at $q_i$ in its $\xi$-component, box $K$ gets the external $\xi_{\fun{ext}}$, and each box receives its left neighbor's position. In our concise format,
\begin{equation}\label{eqn.pc_wiring_concise}
(\vec q,\vec \omega)\mapsto\Big((q_K,\omega_1),\;(\xi_{\fun{ext}},q_0)\mapsto\bigl((\omega_2(q_1),\ldots,\omega_K(q_{K-1}),\xi_{\fun{ext}}),\,(q_0,q_1,\ldots,q_{K-1})\bigr)\Big).
\end{equation}
The routing $\omega_{i+1}(q_i)$ into box $i$'s $\xi$-component is how the wiring transmits the spring force from box $i+1$ to box $i$; the only external forcing is via $q_0$ (the leftmost neighbor position, fed to box $1$) and $\xi_{\fun{ext}}$ (the covector at $q_K$, fed to box $K$) in \eqref{eqn.wave_wd}.

\paragraph{Composition.}

The $\pc$-composite $\Phiphase(\fun{wire}_K)\bigl(\Phiphase(\Part),\ldots,\Phiphase(\Part)\bigr)$ is a $0$-ary morphism $\yon\to\Phiphase(\boxob)$ in $\pc$, i.e.\ a $\Phiphase(\boxob)$-coalgebra. By \cref{prop.pc_as_para}, composition in $\pc$ multiplies state spaces and composes the underlying polynomial maps; here $\Phiphase(\fun{wire}_K)$ has trivial state space and each $\Phiphase(\Part)$ has state space $T^*\rr$, so the composite has state space $1\times(T^*\rr)^K=T^*\rr^K$. Concretely, the $K$-many single-particle coalgebra maps $\Phiphase(\Part)\colon T^*\rr\to\Phiphase(\boxob)(T^*\rr)$ combine via the lax monoidal structure of $\blank\coalg$ (\cref{prop.pc}) into a single $\Phiphase(\boxob)^{\otimes K}$-coalgebra with states $T^*\rr^K$, and the static polynomial map $\Phiphase(\fun{wire}_K)\colon\Phiphase(\boxob)^{\otimes K}\to\Phiphase(\boxob)$ then post-composes:
\[
\Store(T^*\rr^K)\To{\Phiphase(\Part)^{\otimes K}}\Phiphase(\boxob)^{\otimes K}\To{\Phiphase(\fun{wire}_K)}\Phiphase(\boxob).
\]
Read back as a coalgebra $T^*\rr^K\to\Phiphase(\boxob)(T^*\rr^K)$, the composite acts on a state $(\vec q,\vec\xi)$ as follows. The parallel particles $\Phiphase(\Part)^{\otimes K}$ output
\begin{equation*}
\Bigl(\bigl(\tilde q_i,\,q_i'\mapsto\kappa(q_i'-\tilde q_i)\bigr)_{i=1}^K,\;(\vec\xi_{\fun{next}},\vec q_{\fun{prev}})\mapsto\bigl(\vec{\tilde q},\,\vec\xi-\vec\xi_{\fun{next}}+\kappa(\vec q_{\fun{prev}}-\vec{\tilde q})\bigr)\Bigr).
\end{equation*}
and the wiring \eqref{eqn.pc_wiring_concise} then threads them as follows.

Beginning with $(\vec q,\vec\xi):T^*\rr^K$, the coalgebras output $(\vec{\tilde q},\vec\omega)\coloneqq\bigl(\tilde q_i,\,q_i'\mapsto\kappa(q_i'-\tilde q_i)\bigr)_{i=1}^K$. The wiring takes that, outputs the last point and the first covector field, $(\tilde q_K,\omega_1)$, and asks for $(\xi_{\fun{ext}},q_0)$. Upon receiving it, it sends $\bigl((\omega_2(\tilde q_1),\ldots,\omega_K(\tilde q_{K-1}),\xi_{\fun{ext}}),\,(q_0,\tilde q_1,\ldots,\tilde q_{K-1})\bigr)$ back to the coalgebras. In other words, it sends $(\vec\xi_{\fun{next}},\vec q_{\fun{prev}})$ where $(\xi_{\fun{next}})_i\coloneqq\omega_{i+1}(\tilde q_i)=\kappa(\tilde q_i-\tilde q_{i+1})$ for $1\leq i<K$ and $(\xi_{\fun{next}})_K\coloneqq\xi_{\fun{ext}}$, and $(q_{\fun{prev}})_i=\tilde q_{i-1}$ (with $\tilde q_0\coloneqq q_0$). This in turn is sent to $(\vec{\tilde q},\,\vec\xi-\vec\xi_{\fun{next}}+\kappa(\vec q_{\fun{prev}}-\vec{\tilde q}))$. This is exactly \eqref{eqn.potlens_composite_final}: the two routes through the functoriality square agree term by term, confirming the signs, indices, and covector-field evaluations in the formulas above.

\section{Graph Laplacian on a directed graph}\label{sec.graph_laplacian}

The chain of \cref{sec.spring_first_pass} is one instance of a more general construction: again using the harmonic particle, any finite directed graph $G=(E\tto V)$ yields a closed system $\lensunit\to\lensunit$ whose image under $\Phiphase$ is the discrete dynamics of the \emph{graph Laplacian} of $G$. The orientation of $G$ sets up the wiring, but we will see that it drops out of the dynamics: because the harmonic potential is symmetric in each edge's endpoints, what we recover is the \emph{symmetrized} Laplacian, so $G$ behaves like its underlying undirected multigraph. This serves as a sanity check on the formalism of this paper.

\paragraph{From directed graph to static arrangement.}

Let $G\coloneqq(V,E,\src,\tgt)$ be a finite directed graph: i.e.\ $V$ is a finite vertex set, $E$ is a finite edge set, and $\src,\tgt\colon E\to V$ are functions assigning a source and target vertex to each edge. For each $v\in V$ write
\[
\outp{v}\coloneqq\src^{-1}(v)\subseteq E\qqand \inpt{v}\coloneqq\tgt^{-1}(v)\subseteq E
\]
for the sets of out-edges and in-edges at $v$:
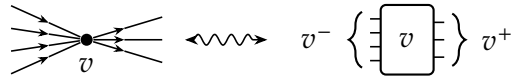
\[
\begin{tikzpicture}[oriented WD, bb port sep=.6, bb min width=.7cm, bbx=.55cm, bby=1.7ex, baseline=(v.center)]
  \def\nin{4}\def\nout{3}
  \node[circle, fill, inner sep=1.6pt] (vtx) {};
  \node[anchor=north] at ($(vtx)+(0,-3pt)$) {$v$};
  \foreach \i in {1,...,\nin}{
    \coordinate (a) at ($(vtx)+(-1cm,{(\i-(\nin+1)/2)*3mm})$);
    \draw[-{Stealth[length=4pt]}] (a) -- ($(a)!.5!(vtx)$);  \draw ($(a)!.5!(vtx)$) -- (vtx);}
  \foreach \i in {1,...,\nout}{
    \coordinate (b) at ($(vtx)+(1cm,{(\i-(\nout+1)/2)*3mm})$);
    \draw[-{Stealth[length=4pt]}] (vtx) -- ($(vtx)!.5!(b)$);  \draw ($(vtx)!.5!(b)$) -- (b);}
  \node[bb={\nin}{\nout}, right=3.8cm of vtx] (v) {$v$};
  \coordinate (itop) at ($(v_in1)!.5!(v.north west)$);
  \coordinate (ibot) at ($(v_in\nin)!.5!(v.south west)$);
  \coordinate (otop) at ($(v_out1)!.5!(v.north east)$);
  \coordinate (obot) at ($(v_out\nout)!.5!(v.south east)$);
  \draw[decorate, line width=.8pt, decoration={brace, amplitude=5pt, mirror}]
    ([xshift=-3pt]v_in1|-itop) -- ([xshift=-3pt]v_in\nin|-ibot)
    node[midway, left=7pt] (inlab) {$\inpt{v}$};
  \draw[decorate, line width=.8pt, decoration={brace, amplitude=5pt}]
    ([xshift=3pt]v_out1|-otop) -- ([xshift=3pt]v_out\nout|-obot)
    node[midway, right=7pt] {$\outp{v}$};
  \draw[{Stealth[length=5pt]}-{Stealth[length=5pt]}, decorate,
    decoration={snake, amplitude=.7mm, segment length=2.2mm, pre length=3pt, post length=3pt}]
    ($(vtx)+(1.3cm,0)$) -- ($(inlab.west|-vtx)+(-3mm,0)$);
\end{tikzpicture}
\]
 Note that the maps $\tgt$ and $\src$ supply canonical bijections $\sum_{v\in V}\inpt{v}\cong E\cong\sum_{v\in V}\outp{v}$. Denote the composite by
\begin{equation}\label{eqn.prism_f}
\inpt{f}\colon\sum_{v\in V}\inpt{v}\too\sum_{v\in V}\outp{v}
\end{equation}

Working in $\Lens{\finset\op}$ (\cref{ex.lens_finsetop}), define the $|V|$-ary \emph{prism wiring} of $G$ to be the lens
\[
\varphi_G\colon\bigltens_{v\in V}\binom{\inpt{v}}{\outp{v}}\to\binom{0}{0}
\]
with $\outp{f}\coloneqq\bang\colon 0\to\sum_{v}\outp{v}$ unique and $\inpt{f}\colon\sum_{v}\inpt{v}\to\sum_{v}\outp{v}$ as in \eqref{eqn.prism_f}.%
\footnote{
In fact finite directed graphs with vertex set $V$ are in bijection with $|V|$-ary prisms in $\Lens{\finset\op}$ whose codomain is $\binom{0}{0}$ (hence $\outp f=\bang$) and whose $\inpt f$ is a bijection.
}
Concretely, each $v$ becomes a box with as many output and input ports as the out-degree and in-degree of $v$; the edges connect these ports one-to-one.

Applying $\rr^{\blank}$ (\cref{lem.lens_pow}) gives a lens
\[
\rr^{\varphi_G}\colon\bigltens_{v\in V}\binom{\rr^{\inpt{v}}}{\rr^{\outp{v}}}\to\lensunit
\]
whose input map $\prod_{v}\rr^{\outp{v}}\to\prod_{v}\rr^{\inpt{v}}$ is again the identity on $\rr^E$, shuttled through the source and target maps. As in \cref{sec.spring_first_pass}, we treat $\rr^{\varphi_G}$ as a static smooth arrangement in $\sarr$ by taking $Q'\coloneqq\rr^0$ and $U'\coloneqq 0$.

\paragraph{The particle at a vertex.}

For each $v\in V$, fix a mass $m_v>0$, and fix a spring constant $\kappa_e>0$ for each $e\in\inpt{v}$. The \emph{harmonic particle at $v$} is the $0$-ary smooth adaptive arrangement
\[
\Part_v\colon\lensunit\to\binom{\rr^{\inpt{v}}}{\rr^{\outp{v}}}
\]
defined by:
\begin{itemize}
\item parameter space $Q_v\coloneqq\rr\in\rvect$ with constant sharp $\sharpR_q(\xi)\coloneqq\xi/m_v$ (independent of $q\in\rr$);
\item output map $\outp{f_v}\colon Q_v\to\rr^{\outp{v}}$ the diagonal $q\mapsto(q)_{e\in\outp{v}}$, broadcasting the particle's position to each of its $|\outp{v}|$-many out-edges;
\item input map $\inpt{f_v}\coloneqq\bang\colon Q_v\times\rr^{\inpt{v}}\to 1$, uniquely determined; and
\item potential $U_v\colon Q_v\times\rr^{\inpt{v}}\to\rr$ given by
\[
U_v\bigl(q,(q'_e)_{e\in\inpt{v}}\bigr)\coloneqq\sum_{e\in\inpt{v}}\tfrac{\kappa_e}{2}\,(q-q'_e)^2.
\]
\end{itemize}
The particle $\Part$ of \cref{sec.spring_first_pass} is the case where $|\outp{v}|=|\inpt{v}|\coloneqq 1$ with $\kappa_e\coloneqq\kappa$ and $m_v\coloneqq m$ for a single $\kappa,m>0$.

\paragraph{The composite.}

Besides the graph $G$, this construction depends on the chosen masses $(m_v)_{v\in V}$ and spring constants $(\kappa_e)_{e\in E}$; we bundle all three into a \emph{weighted directed graph}
\[
\mathbf G\coloneqq\bigl(G,\,(m_v)_{v\in V},\,(\kappa_e)_{e\in E}\bigr),
\]
its vertices weighted by masses and its edges by spring constants. Composing $\rr^{\varphi_G}$ with the family $(\Part_v)_{v\in V}$ in $\sarr$ produces a $0$-ary smooth adaptive arrangement
\begin{equation}\label{eqn.graph_wire}
\fun{wire}_{\mathbf G}\coloneqq\rr^{\varphi_G}\bigl((\Part_v)_{v\in V}\bigr)\colon\lensunit\to\lensunit
\end{equation}
with composite parameter space $Q_G\coloneqq\bigoplus_{v\in V}Q_v=\rr^V$, the composite constant sharp $\sharpR_{\vec q}(\vec\xi)\coloneqq(\xi_v/m_v)_{v\in V}$, vacuous output and input maps, and composite potential $\rr^V\to\rr$ calculated as
\begin{equation}\label{eqn.graph_potential}
U(\vec q)\;=\;\sum_{v\in V}U_v\bigl(q_v,(q_{\src(e)})_{e\in\inpt{v}}\bigr)\;=\;\sum_{e\in E}\tfrac{\kappa_e}{2}\,(q_{\tgt(e)}-q_{\src(e)})^2.
\end{equation}
The first equality comes from $\outp f_v$ and $\inpt f_v$ for each $v\in V$; the second equality uses the partition $E=\bigsqcup_{v\in V}\inpt{v}$.

\paragraph{The discrete dynamics.}

Next we will apply $\Phiphase$ to $\fun{wire}_{\mathbf G}\colon\lensunit\to\lensunit$. Its parameter space is $Q_G=\rr^V$, so by \eqref{eqn.state_space} the carrier of $\Phiphase(\fun{wire}_{\mathbf G})$ is the phase space $T^*Q_G=T^*\rr^V$. Its interface is trivial---$\Phiphase\lensunit=\yon$ (\cref{sec.phase_dynamics})---so $\Phiphase(\fun{wire}_{\mathbf G})$ is a $[\yon,\yon]=\yon$-coalgebra, i.e.\ a function $T^*\rr^V\to T^*\rr^V$. For $(\vec q,\vec\xi)\in T^*\rr^V$, it is given by
\begin{equation}\label{eqn.graph_state_update}
\Phiphase(\fun{wire}_{\mathbf G})(\vec q,\vec\xi)\coloneqq\bigl(\vec q+(\xi_v/m_v)_{v\in V},\,\vec\xi-\xi_Q\bigr),
\end{equation}
where $\xi_Q\in(\rr^V)^*$ is as in \eqref{eqn.covector_triple}, evaluated at the presented position $\vec{\tilde q}\coloneqq\vec q+(\xi_v/m_v)_{v\in V}$ \eqref{eqn.presented_position}. With $\outp f$ and $\inpt f$ both vacuous, the two pullback terms in \eqref{eqn.covector_triple} vanish, leaving $\xi_Q=dU(\vec{\tilde q})$, i.e.\ the differential of \eqref{eqn.graph_potential} at $\vec{\tilde q}$. Component-wise, at a generic argument $\vec q$ (via \eqref{eqn.tangent_vec}):
\begin{align}\label{eqn.graph_gradient}
(\xi_Q)_v
&=\partial_{q_v}\sum_{e\in E}\tfrac{\kappa_e}{2}(q_{\tgt(e)}-q_{\src(e)})^2\\\nonumber
&=\partial_{q_v}\sum_{e\in E}\tfrac{\kappa_e}{2}\bigl(q_{\tgt(e)}^2-2\,q_{\tgt(e)}q_{\src(e)}+q_{\src(e)}^2\bigr)\\\nonumber
&=\sum_{e\in\inpt{v}}\kappa_e(q_v-q_{\src(e)})\;+\;\sum_{e\in\outp{v}}\kappa_e(q_v-q_{\tgt(e)}).
\end{align}
Define a $(V\times V)$-matrix $L_G$ to be the linear function that sends $\vec{q}$ to the vector with the following components:
\begin{equation}\label{eqn.graph_laplacian}
(L_G\,\vec q)_v\;\coloneqq\;\sum_{w\,\To{e}\,v}\kappa_e(q_v-q_w)\;+\;\sum_{v\,\To{e}\,w}\kappa_e(q_v-q_w).
\end{equation}
Comparing the last line of \eqref{eqn.graph_gradient} with this definition, $dU(\vec q)=L_G\,\vec q$ at every $\vec q$, so the backward covector of \eqref{eqn.covector_triple} is
\begin{equation}\label{eqn.xiQ_is_laplacian}
\xi_Q\;=\;dU(\vec{\tilde q})\;=\;L_G\,\vec{\tilde q}.
\end{equation}
We may assume that $\src(e)\neq \tgt(e)$ for all $e$, because self-loops contribute nothing to the sum \eqref{eqn.graph_gradient}. One can check that $L_G$ has diagonal and off-diagonal entries given by
\[
(L_G)_{vv}=\sum_{e\in\inpt{v}\cup\outp{v}}\kappa_e
\qqand
(L_G)_{vw}=-\sum_{v\,\To{e}\,w}\kappa_e-\sum_{w\,\To{e}\,v}\kappa_e
\]
This is the weighted (and symmetrized; see \cref{rmk.symmetric_bonds}) \emph{graph Laplacian} of $G$~\cite{spielman2015spectral}. The symmetrization occurs because the harmonic potential \eqref{eqn.graph_potential} sees each edge only through $(q_{\tgt(e)}-q_{\src(e)})^2$, which is invariant under swapping source and target. Getting the symmetrized Laplacian is exactly what one expects of a conservative network of springs.

Fixing an initial state $(\vec q(0),\vec\xi(0))\in T^*\rr^V$ and iterating $\Phiphase(\fun{wire}_{\mathbf G})$ yields, by \cref{ex.y_coalgebra_trajectory}, a trajectory $\bigl(\vec q(t),\vec\xi(t)\bigr)_{t\geq0}$. As in \eqref{eqn.doubledot}, write $\ddot{q}_v(t)\coloneqq q_v(t+1)-2q_v(t)+q_v(t-1)$ for $t\geq1$.

\begin{proposition}\label{prop.graph_laplacian}
For $\fun{wire}_{\mathbf G}$ as in \eqref{eqn.graph_wire}, every trajectory of $\Phiphase(\fun{wire}_{\mathbf G})\colon T^*\rr^V\to T^*\rr^V$ satisfies
\begin{equation}\label{eqn.graph_laplacian_dynamics}
m_v\,\ddot{q}_v(t)\;=\;-\bigl(L_G\,\vec q(t)\bigr)_v
\qquad(v\in V,\ t\geq1),
\end{equation}
where $L_G$ is the graph Laplacian \eqref{eqn.graph_laplacian} of $G$.
\end{proposition}
\begin{proof}
From the state update equation \eqref{eqn.graph_state_update}, the position update is $q_v(t+1)\coloneqq q_v(t)+\xi_v(t)/m_v$, i.e.\ $\vec{\tilde q}(t)=\vec q(t+1)$; solving for momentum, $\xi_v(t)=m_v\bigl(q_v(t+1)-q_v(t)\bigr)$. The momentum update \eqref{eqn.graph_state_update}, with $t$ replaced by $t-1$, gives $\xi_v(t)-\xi_v(t-1)=-(\xi_Q)_v=-(L_G\vec q(t))_v$, using \eqref{eqn.xiQ_is_laplacian} and $\vec{\tilde q}(t-1)=\vec q(t)$. Substituting the momentum gives
\[
m_v\bigl(q_v(t+1)-q_v(t)\bigr)-m_v\bigl(q_v(t)-q_v(t-1)\bigr)\;=\;-(L_G\vec q(t))_v,
\]
and collecting the position terms on the left, in the notation of \eqref{eqn.doubledot}, yields \eqref{eqn.graph_laplacian_dynamics}.
\end{proof}

The response $\sharpR_{\vec q}(\vec\xi)=(\xi_v/m_v)_{v\in V}$ is constant and symmetric and the potential \eqref{eqn.graph_potential} is quadratic with $dU=L_G$ symmetric---so the arrangement is harmonic (\cref{def.arrangement_terminology}) and \cref{prop.closed_conservation} applies (with $\kappa=L_G$): $\Phiphase(\fun{wire}_{\mathbf G})\colon T^*\rr^V\to T^*\rr^V$ is a map in $\rvect$, hence preserves the canonical symplectic pairing.

\begin{remark}[Heat equation from the configuration integrator]\label{rmk.graph_heat}
Rather than $\Phiphase$, applying $\Phiconf$ to the \emph{very same} smooth adaptive arrangement $\fun{wire}_{\mathbf G}$ of \eqref{eqn.graph_wire}---same graph wiring, same potential, and the same positive inverse-mass response $\sharpR_{\vec q}(\vec\xi)=(\xi_v/m_v)_v$---recovers the discrete \emph{heat} equation on $G$. The state set is now $Q_G=\rr^V$ (no momentum), and the same computation \eqref{eqn.graph_gradient} applies---except that the configuration integrator presents the stored position itself, so here $\xi_Q=dU(\vec q)=L_G\vec q$. Its descent update $\vec q(t+1)=\vec q(t)-\sharpR_{\vec q(t)}(\xi_Q)$ then reads component-wise as
\[
m_v\bigl(q_v(t+1)-q_v(t)\bigr)\;=\;-\bigl(L_G\,\vec q(t)\bigr)_v
\qquad(v\in V,\ t\geq0).
\]
Recall from \eqref{eqn.graph_laplacian} that the $\kappa_e$'s are built into $L_G$. Here, we read $q_v$ as temperature and $m_v$ as heat capacity at vertex $v$, read $\kappa_e$ as the thermal conductance along an edge, and the right-hand side is the net heat flow into $v$.

Writing $\dot q_v(t)\coloneqq q_v(t+1)-q_v(t)$ and $M\coloneqq\operatorname{diag}(m_v)_{v\in V}$ for the diagonal mass matrix, the update is the discrete heat equation
\[
M\,\dot{\vec q}\;=\;-L_G\,\vec q,
\]
diffusion on $G$ with the masses $m_v$ serving as heat capacities.

Thus a single smooth adaptive arrangement gives wave dynamics under the phase integrator and heat dynamics under the configuration integrator: the second-order, symplectic regime and the first-order, descent regime are two semantics of one syntactic object.
\end{remark}

\begin{remark}[Predicting a closed system]\label{rmk.graph_environment}
The environment of \cref{ex.training_data} need not be a dataset; it can be a closed system being observed. Given a closed system $c\colon S\to S$ (a $\yon$-coalgebra, \cref{ex.y_coalgebra_trajectory}) and an observation $o\colon S\to X$, the carrier $S$ with readout and update
\[
\langle o,\ o\circ c\rangle\colon S\to X\times X
\qqand
c\colon S\to S
\]
is a $\cotof X\otimes\cotof X$-coalgebra---an environment $\yon\to\Phiconf(\tbinom1X\ltens\tbinom1X)$ in $\pc$---presenting the observation now, $o(s)$, as input and one step ahead, $o(c(s))$, as label, then advancing by $c$. Closing the learner \eqref{eqn.dl_composite} (with $\rr^m=\Lambda=X$) with it descends the weights \eqref{eqn.dl_gradient_update} toward predicting $o\circ c$ from $o$. If $o\circ c=\bar c\circ o$ for some $\bar c\colon X\to X$, the target is the system's one-step map $\bar c$; in general no such $\bar c$ exists, and the learner is trained only on the observed transition pairs $\bigl(o(s),\,o(c(s))\bigr)$.

For the graph, take the closed system to be $c\coloneqq\Phiphase(\fun{wire}_{\mathbf G})\colon T^*\rr^V\to T^*\rr^V$ \eqref{eqn.graph_state_update}.
With the bundle projection $\pi\colon T^*\rr^V\to\rr^V$ (\cref{lem.TT_projection_diagonal_monoidal}), no $\bar c$ as above exists: only positions are observed, but states with the same position and different momenta have different next positions. 
However, if we instead take $X\coloneqq\rr^V\times\rr^V$ and observe two consecutive configurations, $o\coloneqq\langle\pi,\,\pi\circ c\rangle\colon T^*\rr^V\to X$, sending $(q,\xi)$ to $\bigl(q,\ q+\sharpR_q\xi\bigr)$, then $o\circ c=\bar c\circ o$ for some $\bar c\colon X\to X$.
In this case, the network is fit to the graph's propagator $\bar c$, which is linear in $L_G$.
Any other closed system $c$ in its place---Newton's method, gradient descent---trains a predictor of that one instead.
\end{remark}

\section*{Conclusion}

The operad $\sarr$ of smooth adaptive arrangements provides a single compositional syntax in which gradient-based learning and Hamiltonian-style mechanics coexist as functorial semantics $\Phiconf,\Phiphase\colon\sarr\to\pc$. The key mediating construction is lens internalization, which converts lens interfaces into polynomial interfaces, and integrators, which determine how stored state relates to the parameter space.

The same split appears in the neural-network example of \cref{sec.dl_warmup}. Applying $\Phiconf$ to the parameterized map $F\colon Q\times\rr^m\to\rr^n$ recovers gradient descent; we expect that applying $\Phiphase$ to the same arrangement may yield momentum-based weight dynamics. This would place standard and momentum-based training as two integrators of one smooth adaptive arrangement, exactly as the wave and heat equations are two integrators of one graph wiring (\cref{rmk.graph_heat}).

Several directions remain open. Replacing responsive vector spaces with responsive manifolds equipped with connections (\cref{rmk.rmfdc_generalization}) would extend the framework beyond linear parameter spaces. The monoidality of specific multi-stage integrators such as leapfrog remains to be checked (\cref{rmk.multistage}). A probabilistic extension---replacing the dynamics map $\dyn_Q\colon\Store(S_Q)\to\cotof Q$ of \cref{def.integrator} by a map $\Store(S_Q)\to\cotof Q\circ\rand$, using the polynomial monad $\rand=\sum_{N:\nn}\sum_{P:\fun{Dist}(N)}\yon^N$ introduced in \cref{sec.coalgebras}---might provide finitely-supported stochastic semantics, e.g.\ for stochastic parameter updates in neural networks.

On the suboperad of quadratic arrangements (\cref{prop.suboperads}), we expect the two example semantics $\Phiconf$ and $\Phiphase$ to be related: for every quadratic arrangement $f$ with parameter space $Q$, there should be an arrangement $g$ with parameter space $T^*Q$ such that $\Phiphase(f)=\Phiconf(g)$, functorially in $f$.

The linear arrangements span a suboperad $\sarr^{\tn{lin}}\ss\sarr$ (\cref{prop.suboperads}) on which other semantics may be extracted. The superposition principle of heat and wave systems should extend functorially across the whole linear stratum. We expect the inclusion $\sarr^{\tn{lin}}\ss\sarr$ to factor through a suboperad of ``critically pointed arrangements,'' with second-order Taylor expansion about the critical point providing a retraction onto $\sarr^{\tn{lin}}$.

It would also be interesting to compare this framework with Coya's compositional framework for bond graphs~\cite{coya2017bondgraphs}, particularly by relating his two Lagrangian-relation semantics to the deterministic coalgebra semantics here.

We plan to pursue this symplectic picture and the critical-point linearization in a sequel.

\printbibliography

\end{document}